\documentclass[reqno,12pt,letterpaper]{amsart}
\usepackage[french,english]{babel}
\usepackage{amsmath,amssymb,amsthm,graphicx,mathrsfs,url}
\usepackage[usenames,dvipsnames]{color}
\usepackage[colorlinks=true,linkcolor=Red,citecolor=Green]{hyperref}
\usepackage[usenames,dvipsnames]{xcolor}
\usepackage{tikz}
\usepackage{graphicx}
\usepackage[colorlinks=true]{hyperref}
\usepackage{color}
\usepackage{amsmath,amsfonts}
\usepackage{calrsfs}
\usepackage{amsmath,environ}
\DeclareMathAlphabet{\pazocal}{OMS}{zplm}{m}{n}

\def\smallsection#1{\smallskip\noindent\textbf{#1}.}

\numberwithin{equation}{section}
\usepackage{amsmath, amsthm}
    \usepackage{dsfont}
    \usepackage{fancybox}
    \usepackage{amssymb}

\newcommand{\comp}{{\operatorname{comp}}}
\newcommand{\R}{\mathbb{R}}

\newcommand{\p}{{\partial}}
\newcommand{\olambda}{{\overline{\lambda}}}

\newcommand{\ov}{{\overline{v}}}
\newcommand{\w}{\omega}
\newcommand{\dd}[2]{\dfrac{\partial #1}{\partial #2}}
\newcommand{\matrice}[1]{\left( \begin{matrix}
#1
\end{matrix} \right)}
\newcommand{\Det}{{\operatorname{Det}}}

\newcommand{\Res}{{\operatorname{Res}}}
\newcommand{\sing}{{\operatorname{sing}}}
\newcommand{\DD}{{\mathcal{D}}}

\newcommand{\loc}{{\operatorname{loc}}}
\newcommand{\epsi}{\varepsilon}
\newcommand{\sgn}{{\operatorname{sgn}}}
\newcommand{\trace}{{\operatorname{Tr}}} 
\newcommand{\te}{\theta}
\newcommand{\lr}[1]{\left\langle #1 \right\rangle}
\newcommand{\FF}{\mathcal{F}}

\newcommand{\tK}{{\tilde{K}}}

\newcommand{\WW}{\mathcal{W}}
\newcommand{\AAA}{\mathcal{A}}
\newcommand{\KK}{\mathcal{K}}

\newcommand{\eff}{{\operatorname{eff}}}
\newcommand{\supp}{\mathrm{supp}}

\newcommand{\Bb}{\mathbb{B}}
\newcommand{\Z}{\mathbb{Z}}
\newcommand{\Dd}{\mathbb{D}}

\newcommand{\TT}{\mathcal{T}}
\newcommand{\Ss}{\mathbb{S}}
\newcommand{\Id}{{\operatorname{Id}}}
\newcommand{\BB}{\mathcal{B}}
\newcommand{\VV}{{\mathcal{V}}}

\newcommand{\EE}{\mathcal{E}}

\newcommand{\UU}{{\mathcal{U}}}
\newcommand{\LL}{{\mathcal{L}}}
\newcommand{\HH}{\mathcal{H}}

\newcommand{\SSSS}{{\mathcal{S}}}

\newcommand{\systeme}[1]{\left\{ \begin{matrix} #1 \end{matrix} \right.}

\newcommand{\C}{\mathbb{C}}

\newcommand{\az}{\alpha}

\newcommand{\ek}{{e^{ik\bullet/\epsi}}}

\newcommand{\CC}{\mathcal C}
\renewcommand{\Re}{\operatorname{Re}}
\renewcommand{\Im}{\operatorname{Im}}

\newcommand{\Tt}{{\mathbb{T}}}

\newcommand{\1}{\mathds{1}}
\newcommand{\RR}{\mathcal{R}}

\newcommand{\cross}[2]{\draw[black] (#1-.05,#2-.05) -- (#1+.05,#2+.05); \draw[black] (#1-.05,#2+.05) -- (#1+.05,#2-.05);}
\newcommand{\crossr}[2]{\draw[red] (#1-.05,#2-.05) -- (#1+.05,#2+.05); \draw[red] (#1-.05,#2+.05) -- (#1+.05,#2-.05);}

\newcommand{\Reso}[2]{\draw[red] (#1-.05,#2-.05) -- (#1+.05,#2+.05); 
\draw[red] (#1-.05,#2+.05) -- (#1+.05,#2-.05);
\draw[red] (#1,#2) circle (16pt);}

\newcommand{\Reson}[2]{\draw[blue] (#1-.05,#2-.05) -- (#1+.05,#2+.05); 
\draw[blue] (#1-.05,#2+.05) -- (#1+.05,#2-.05);
\draw[blue] (#1,#2) circle (6pt);}

\newcommand{\Resonn}[2]{\draw[red] (#1-.05,#2-.05) -- (#1+.05,#2+.05); 
\draw[red] (#1-.05,#2+.05) -- (#1+.05,#2-.05);
\draw[red] (#1,#2) circle (23pt);}

\title[Scattering resonances for highly oscillatory potentials.]{Scattering resonances for highly oscillatory potentials.}

\author{Alexis Drouot}
\email{alexis.drouot@gmail.com}

\NewEnviron{equations}{%
\begin{equation}\begin{gathered}
  \BODY
\end{gathered}\end{equation}
}

\NewEnviron{equations*}{%
\begin{equation*}\begin{gathered}
  \BODY
\end{gathered}\end{equation*}
}

\newtheorem{thm}{Theorem}

\newtheorem{lem}{Lemma}[section]

\newtheorem{theorem}[thm]{Theorem}
\theoremstyle{definition}

\newenvironment{abstracte}[1]
  {\bigskip\selectlanguage{#1}%
   \begin{center}\bfseries\abstractname\end{center}}
  {\par\bigskip}

\begin{document}

\maketitle

\begin{abstracte}{english}
We study resonances of compactly supported potentials $ V_\epsi(x)  = W ( x, x/\epsi ) $ where $ 
W : \R^d \times  \R^d / (  2\pi \Z) ^d \to \C $, $ d $ odd. That means that 
$ V_\epsi$ is a sum of a slowly varying potential, $ W_0 $,  and one oscillating at frequency $1/\epsi$. 
When $ W_0 \equiv 0 $ we prove that there are no resonances above the line $\Im \lambda = -A \ln(\epsi^{-1})$, except a simple resonance near $0$ when $ d=1$.  We show that this result is optimal by constructing a one-dimensional example. This settles a conjecture of Duch\^ene-Vuki\'cevi\'c-Weinstein \cite{DucVulWei1}. In the case when $ W_0 \neq 0$ and $W$ smooth we prove that resonances in fixed strips admit an expansion in powers of $\epsi$. The argument provides a method for computing the coefficients of the expansion. We produce an effective potential converging uniformly to $W_0$ as $\epsi \rightarrow 0$ and whose resonances approach resonances of $V_\epsi$ modulo $O(\epsi^4)$. This improves the one-dimensional result of Duch\^ene, Vuki\'cevi\'c and Weinstein and extends it to all odd dimensions.
\end{abstracte}

\section{Introduction}

In this paper we are interested in the poles of the meromorphic continuation of $(-\Delta + \VV - \lambda^2)^{-1}$ where $d$ is odd and $\VV : \R^d \rightarrow \C$ is a bounded compactly supported potential. These poles called scattering resonances appear in many physical situations, for instance their imaginary parts 
are the rates of decay of waves scattered by $\VV$.

Let $-\Delta \geq 0 $ be the free Laplacian on $\R^d$. The operator $R_0(\lambda) = (-\Delta - \lambda^2)^{-1}$, well defined as an operator $L^2(\R^d) \rightarrow H^2(\R^d)$ for $\Im \lambda > 0$, extends to a meromorphic family of bounded operators $L^2_\comp(\R^d) \rightarrow H^2_\loc(\R^d)$ for $\lambda \in \C$ (see \S\ref{subsec:1} for review of notation). This family admits one simple pole at $0$ if $d=1$ and is entire if $d \geq 3$. If $\VV$ is a bounded compactly supported function on $\R^d$ then $R_\VV(\lambda) = (-\Delta+\VV-\lambda^2)^{-1}$ is well defined for $\Im \lambda \gg 1$ as an operator $L^2(\R^d) \rightarrow H^2(\R^d)$. It extends  to a meromorphic family of operators $L^2_\comp(\R^d) \rightarrow H^2_\loc(\R^d)$ -- see \cite{DyaZwo}.

Let $W$ be a {\em bounded} complex valued function with support in $\Bb^d(0,L)\times\Tt^d$.
 We define $V_\epsi$ as
\begin{equation*}
V_\epsi(x) = W\left( x, \dfrac{x}{\epsi} \right).
\end{equation*}
If  $W$ is formally given by
\begin{equation*}
W(x,y) = \sum_{k \in \Z^d} W_k(x) e^{iky}
\end{equation*}
we can write $V_\epsi$ as a highly oscillatory perturbation of $W_0$:
\begin{equation}\label{eq:1a}
V_\epsi(x) =  W_0(x)+V_\sharp(x), \  \ \ \ V_\sharp(x) =  \sum_{k \neq 0} W_k(x) e^{ikx/\epsi}.
\end{equation}
In this paper we study resonances of potentials $V_\epsi$ given by \eqref{eq:1a}. 


\subsection{Main results}


The first theorem concerns the case of a vanishing slowly varying part. In the notations of \eqref{eq:1a} we will assume for this result that $W \in L^\infty_0(\Bb^d(0,L) \times \Tt^d)$ (i.e., $\supp(W)$ is a compact subset of $\Bb^d(0,L) \times \Tt^d$ and $W$ is uniformly bounded) and that moreover,
\begin{equation}\label{eq:9f}\begin{gathered}
\exists s \in (0,1), \ \ \ \sum_{k \neq 0} \dfrac{|W_k|_{H^s}}{|k|^s} < \infty  \text{ if } d=1, \\
\sum_{k \neq 0} \dfrac{\|W_k\|_1}{|k|} < \infty \text{ if } d \geq 3.\end{gathered}
\end{equation}

\begin{theorem}\label{thm:1} Let $W$ be in $L^\infty_0(\Bb^d(0,L) \times \Tt^d,\C)$ such that $W_0 \equiv 0$ and \eqref{eq:9f} holds. Then there exists $C, c, A$ three positive constants such that
\begin{equation*}\begin{gathered}
\text{if } d=1, \ \Res(V_\epsi) \setminus \Dd\left(0,c\epsi^{s/2}\right)\subset \left\{ \lambda \in \C : \ \Im \lambda \leq C-A\ln(\epsi^{-1}) \right\}; \\
\text{if } d \geq 3, \ \Res(V_\epsi) \subset \left\{ \lambda \in \C : \ \Im \lambda \leq C-A\ln(\epsi^{-1}) \right\}.
\end{gathered}
\end{equation*}
\end{theorem}


This settles a conjecture of \cite{DucVulWei1}: for odd dimensions $d \geq 3$ and $\epsi$ small enough the potential $V_\epsi$ does not have a bound state. In \S\ref{subsec:8} we construct a step-like function $W$ such that $V_{\pi/(2n)}$ has a resonance $\lambda_n  \sim -i \ln(n)$ as $n \rightarrow +\infty$. This shows that one cannot improve the rate of escape of resonances given by Theorem \ref{thm:1} in dimension $1$.

In the next statements we always assume that $W$ is smooth. We consider the 
case $ W_0 \neq 0 $. If $\lambda_0$ is a simple resonance of $W_0$ we can write 
\begin{equation}\label{eq:12ac}
R_{W_0}(\lambda) = \dfrac{i u \otimes v}{\lambda - \lambda_0} + H(\lambda), \ H(\lambda) \text{ holomorphic near } \lambda_0, 
\end{equation}
for some functions $u,v \in H^2_\loc(\R^d,\C)$ called resonant states. As the potential $V_\epsi$ given by \eqref{eq:1a} converges weakly to $W_0$ it is natural to expect that resonances of $V_\epsi$ converge to resonances of $W_0$. In fact a much stronger statement holds:

\begin{thm}\label{cor:3q} Let $W$ belong to $C_0^\infty(\Bb^d(0,L) \times \Tt^d,\C)$ and $V_\epsi$ be given by \eqref{eq:1a}. Let $\lambda_0$ be a simple resonance of $W_0$. In a neighborhood of $\lambda_0$ and for $\epsi$ small enough the potential $V_\epsi$ admits a unique resonance $\lambda_\epsi$. Moreover, 
for any $ N $, 
\begin{equation*}
\lambda_\epsi = \lambda_0 + c_2 \epsi^2 + c_3 \epsi^3   + ... + c_{N-1} \epsi^{N-1} +  O ( \epsi^{N} ) , \ \ c_j \in 
\C . 
\end{equation*}
If $u, v$ are the resonant states of \eqref{eq:12ac} then
\begin{gather}\label{eq:La0La1} 
\begin{gathered}
c_2 = i\int_{\R^d} \Lambda_0(x) u(x) v(x) dx, \ \ \ \ c_3=i\int_{\R^d} \Lambda_1(x) u(x) v(x) dx, \\
\Lambda_0 = \sum_{k \neq 0} \dfrac{W_k W_{-k}}{|k|^2}, \ \ \ \ \Lambda_1 = -2\sum_{k \neq 0} \dfrac{W_{-k}( (k \cdot D)  W_k)}{|k|^4}.
\end{gathered}
\end{gather}
\end{thm}

If $W$ is real-valued 
then so are $\Lambda_0$ and $\Lambda_1$.
In \S \ref{subsec:2} we will prove a version of Theorem \ref{cor:3q} for resonances of higher multiplicity. Theorem \ref{cor:3q} implies that perturbations of $W_0$ by a high frequency potential $V_\sharp$ enjoy some similarities with suitable analytic perturbations of $W_0$. In fact we have the following 

\begin{thm}\label{cor:3} Assume that $W$ belongs to $C_0^\infty(\Bb^d(0,L) \times \Tt^d,\C)$ and that $V_\epsi$ is given by \eqref{eq:1a}. Let $V_{\eff,\epsi} = W_0 - \epsi^2 \Lambda_0 - \epsi^3 \Lambda_1$ where $\Lambda_0, \Lambda_1$ are given in \eqref{eq:La0La1}. For every bounded family $ \epsi \mapsto \mu_\epsi$ of simple resonances of $V_{\eff,\epsi}$ there exists a family of resonances $\epsi  \mapsto \lambda_\epsi$ of $V_\epsi$ such that
\begin{equation*}
|\lambda_\epsi - \mu_\epsi| = O(\epsi^4).
\end{equation*}
Conversely for every bounded family $\epsi  \mapsto \lambda_\epsi$ of simple resonances of $V_\epsi$ there exists a family of resonances $\epsi  \mapsto \mu_\epsi$ of $V_{\eff, \epsi}$ such that
\begin{equation*}
|\lambda_\epsi - \mu_\epsi| = O(\epsi^4).
\end{equation*}
\end{thm}

The potential $V_{\eff,\epsi}$ plays the role of an effective potential. In dimension one $ \Lambda_0 $ was already derived in \cite{DucVulWei1}. 


We next give a uniform description of the behavior of resonances of $V_\epsi$ as $\epsi \rightarrow 0$. For $W_0 \in C^\infty_0(\Bb^d(0,L), \C)$ we define $m_{W_0}(\lambda_0)$ the multiplicity of a resonance $\lambda_0$ of $W_0$. If $\epsi,B,c,A$ are given positive constants let $\CC_\epsi, \TT_\epsi$ and $\DD_\epsi$ be the sets
\begin{equation}\label{eq:9g}\begin{gathered}
\CC_\epsi = \bigcup_{\substack{\lambda \in \Res(W_0), \\ \Im \lambda \geq -B }} \Dd\left(\lambda, c \epsi^{2/m_{W_0}(\lambda)}\right), \  \ \ 
\TT_\epsi = \bigcup_{\substack{\lambda \in \Res(W_0), \\ \Im \lambda \leq -B}} \Dd\left(\lambda, \lr{\lambda}^{-d-1}\right) \\
\DD_\epsi = \left\{ \lambda \in \C : \ \Im \lambda \leq -B, \ |\lambda|^{2d+1} \geq A \ln(\epsi^{-1}) \right\}.\end{gathered}
\end{equation}

\begin{thm}\label{cor:3s} Assume that $W$ belongs to $C_0^\infty(\Bb^d(0,L) \times \Tt^d,\C)$ and that $V_\epsi$ is given by \eqref{eq:1a}. There exists $A>0$ with the following. For any $B > 0$, there exists $c>0$ such that for all $\epsi$ small enough if $\CC_\epsi, \TT_\epsi$ and $\DD_\epsi$ are given by \eqref{eq:9g} then
\begin{equation*}
\Res(V_\epsi) \subset \CC_\epsi \cup \TT_\epsi \cup \DD_\epsi.
\end{equation*}
\end{thm}

Therefore assume that $\epsi \mapsto \lambda_\epsi$ is a family of resonance of $V_\epsi$. Then after passing to a subsequence $ \epsi_j \to 0 $, one of the three following scenarios occurs:
\begin{enumerate}
\item[$(i)$] $\lambda_\epsi$ converges to a resonance $\lambda_0$ of $W_0$ and $\lambda_\epsi = \lambda_0 + O(\epsi^{2/m_{W_0}(\lambda_0)})$.
\item[$(ii)$] $\Im \lambda_\epsi \rightarrow -\infty$ and $|\lambda_\epsi|$ grows at least like $\ln(\epsi^{-1})^{1/(2d+1)}$.
\item[$(iii)$] $\Im \lambda_\epsi \rightarrow -\infty$ and $d(\lambda_\epsi, \Res(W_0)) = O(|\lambda_\epsi|^{-d-1})$.
\end{enumerate}
(Here we suppressed the subsequence notation.)

For the case of a potential $W_0$ with simple resonances these results are illustrated on Figure \ref{fig:1} below. \begin{center}
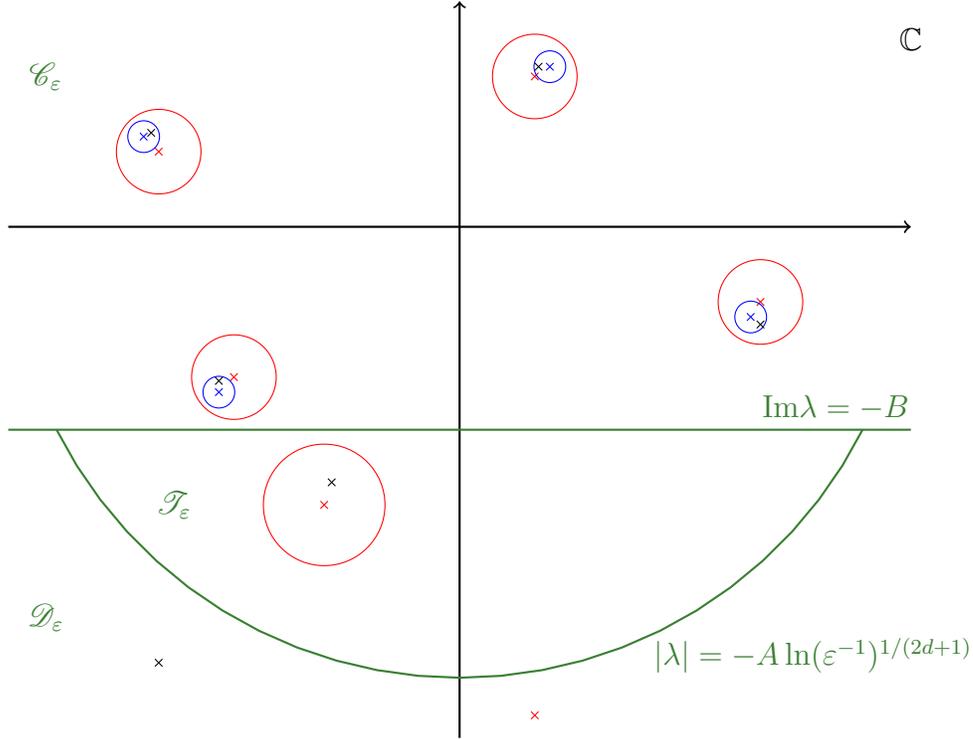
\begin{figure}
\begin{tikzpicture}
\draw[thick,->] (-6,0) -- (6,0) node[anchor=north east] {}; 
\draw[thick,->] (0,-6.8)  -- (0,3) node[anchor=north east] {};
\node at (6,2.5) {$\mathbb{C}$};

\Reso{-4}{1}
\Reson{-4.2}{1.2}
\cross{-4.1}{1.25}
\node[OliveGreen] at (-5.5,2) {$\mathbb{\CC_\epsi}$};

\Reso{1}{2}
\Reson{1.2}{2.13}
\cross{1.05}{2.13}

\Reso{-3}{-2}
\Reson{-3.2}{-2.2}
\cross{-3.2}{-2.05}

\Reso{4}{-1}
\Reson{3.87}{-1.2}
\cross{4}{-1.3}

\node[OliveGreen] at (5,-2.4) {$\text{Im} \lambda =  - B$};

\draw [OliveGreen,thick,domain=206.7:333.3] plot ({6*cos(\x)}, {6*sin(\x)});

\node[OliveGreen] at (-3.8,-3.7) {$\mathbb{\TT_\epsi}$};
\draw[OliveGreen,thick] (-6,-2.7) -- (6,-2.7); 
\node[OliveGreen] at (4.7,-5.7) {$|\lambda| =  - A \ln(\epsi^{-1})^{1/(2d+1)}$};
\node[OliveGreen] at (-5.5,-5.2) {$\mathbb{\DD_\epsi}$};

\Resonn{-1.8}{-3.7}
\cross{-1.7}{-3.4}

\cross{-4}{-5.8}
\crossr{1}{-6.5}
\end{tikzpicture}\label{fig:1}
\caption{The red (resp. black, blue) crosses denote resonances of $W_0$ (resp. $V_\epsi$, $V_{\eff,\epsi}$). Above the line $\Im \lambda = B$ resonances of $V_{\eff,\epsi}$ and $V_\epsi$ lie within red disks of radius $\sim \epsi^2$ centered at resonances of $W_0$. Resonances of $V_{\eff,\epsi}$ and $V_\epsi$ in these disks lie within a distance $\sim \epsi^4$ from each other. In the middle zone resonances of $V_\epsi$ lie within disks of radius $\sim 1$ centered at resonances of $W_0$. Below both curves $\Im \lambda =  - B$ and $|\lambda| =  - A \ln(\epsi^{-1})^{1/(2d+1)}$ resonances of $V_\epsi, V_{\eff,\epsi}$ and $W_0$ are no longer correlated.}
\end{figure}
\end{center}

Theorems \ref{cor:3q}, \ref{cor:3} and \ref{cor:3s} are actually consequences of a stronger result. For $\VV \in L^\infty_0(\Bb^d(0,L),\C)$ and $\rho \in C_0^\infty(\R^d)$ that is $1$ on $\supp(\VV)$ we define $K_\VV(\lambda) = \rho R_0(\lambda)\VV$. If $p \geq d+1$ and $\Psi$ is the entire function defined by
\begin{equation}\label{eq:2d}
\Psi(z) = (1+z) \exp\left(-z+ \dfrac{z^2}{2}- ...  +\dfrac{(-z)^{p-1}}{p-1} \right) - 1
\end{equation}
the operator $\Psi(K_\VV(\lambda))$ is trace class. This allows us to define the Fredholm determinant
\begin{equation}\label{eq:1d}
D_\VV(\lambda) = \Det\left(\Id + \Psi(K_\VV(\lambda))\right).
\end{equation}
Apart from the special case of $0$ in dimension one resonances of $\VV$ are exactly zeros of $D_\VV$ -- see \cite[Theorem 5.4]{GLMZ05}. 
To deal with the particular case of the zero resonance in dimension one we define $X_d=\C$ if $d \geq 3$ and $X_1=\C \setminus \{0\}$. The following result shows that $D_V$ admits an expansion in powers of $\epsi$. 

\begin{theorem}\label{thm:5} Let $W$ in $C_0^\infty(\Bb^d(0,L) \times \Tt^d, \C)$ and $V_\epsi$ be the potential given by \eqref{eq:1a}. Fix $N > 0$ and $p = 4(d+N)N$. If $D_{V_\epsi}(\lambda)$ is the Fredholm determinant defined in \eqref{eq:1d} then there exists $a_0, ..., a_{N-1}$ holomorphic functions of $\lambda \in X_d$ such that uniformly on compact subsets of $X_d$,
\begin{equation*}
D_{V_\epsi}(\lambda) = a_0(\lambda) + \epsi^2 a_2(\lambda) +  \epsi^3 a_3(\lambda) + ... + \epsi^{N-1} a_{N-1}(\lambda) + O(\epsi^N).
\end{equation*}
Moreover if $\Lambda_0$ and $\Lambda_1$ are the potentials defined in Theorem \ref{cor:3q} then $a_0(\lambda) = D_{W_0}(\lambda),$
\begin{equation*}\begin{gathered}
a_2(\lambda) = - D_{W_0}(\lambda) \cdot \trace\left((\Id + K_{W_0})^{-1} (-K_{W_0})^{p-2} K_{\Lambda_0} \right), \\
a_3(\lambda) = - D_{W_0}(\lambda)  \cdot \trace\left((\Id + K_{W_0})^{-1} (-K_{W_0})^{p-2} K_{\Lambda_1} \right).\end{gathered}
\end{equation*}
\end{theorem}

Here again we note that a perturbation of a potential $W_0$ by a highly oscillatory potential enjoys similarities with a suitable analytic perturbation of $W_0$. We will make this observation more precise in \S \ref{sub:3} below.

\subsection{Relation with existing work} Our original motivation for investigating highly oscillatory potentials came from Christiansen \cite{Ch06} where it was shown that certain complex-valued oscillatory potentials have no resonances at all. The proof there is based on a priori estimates on solutions of $(\Id + K_\VV(\lambda))u=0$. Although real valued potentials have infinitely many resonances -- see \cite{SZ}, \cite{HS} and references given there -- similar ideas lead to absence of resonance in strips depending logarithmically on the frequency of oscillations (Theorem \ref{thm:1}).

In dimension one scattering resonances of potentials of the form \eqref{eq:1a} have recently been  extensively studied.  For $W$ with $W_0 \equiv 0$ and $V_\epsi$ given by \eqref{eq:1a} Borisov and Gadyl'shin investigate in \cite{BorGad} the behavior of eigenvalues of the Schr\"odinger operator $D_x^2 + V_\epsi$. They give a sufficient condition for an eigenvalue to exist for small $\epsi$. Under this condition they derive an expansion of the eigenvalue as $\epsi \rightarrow 0$. In \cite{Borisov2} Borisov refines this result by including potentials that are less regular. These two papers focus on the spectrum and on the eigenvalues rather than on scattering resonances. Scattering theory for operators of the form $D_x^2 + V_\epsi$ was systematically presented by Duch\^ene--Weinstein \cite{DucWei}. In that paper the authors study the behavior of the transmission coefficient of such potentials. They prove that away from possible poles, the transmission coefficient of $V_\epsi$ converges to that of $W_0$. They give estimates on the remainder that depend on the regularity of $W$. The study is later continued in \cite{DucVulWei1}. In that paper Duch\^ene, Vuki\'cevi\'c and Weinstein generalize the result of \cite{BorGad} to geneal potentials $V_\epsi$ given by \eqref{eq:1a}. They give conditions for the existence of a bound state of $V_\epsi$ for small $\epsi$ whose energy is expressed in terms of an effective potential which is an analytic perturbation of $W_0$.

Also in dimension one, \cite{Borisov1} studies in detail  the spectrum of Schr\"odinger operators with a potential that is the sum of a compactly supported potential and a periodic potential oscillating at frequency $\sim \epsi^{-1}$. The paper \cite{DucVulWei2} deals with potentials that are a sum of a periodic potential $Q_{\text{per}}$ perturbed by a term $Q_\epsi$ oscillating at frequency $\epsi^{-1}$. As $\epsi \rightarrow 0$ they observe the bifurcation of eigenvalues of $D_x^2  + Q_{\text{av}} + Q_\epsi$ at distance $\epsi^4$ from the edges of the continuous spectrum of $D_x^2 + Q_{\text{av}}$.

In higher dimension the work \cite{GolWei} deals with general perturbations of operators $-\Delta + W_0$. The perturbation $V_\sharp$ needs to be small when measured in a suitable space. They show that simple resonances of perturbed operators depend analytically on $V_\sharp$. Although such a result applies to potentials given by \eqref{eq:1a} it does not yield an expansion of resonances in powers of $\epsi$ because $V_\sharp$ does not depend smoothly on $\epsi$. 

Let us discuss in more detail the relation between our work specialized to dimension one and \cite{DucVulWei1}. By  fine analysis of the scattering coefficients they show that the transmission coefficient of $V_\epsi$ is equal to the transmission coefficient of the effective potential
\begin{equation*}
V_\eff(x) = W_0(x) - \epsi^2  \Lambda_0(x),  \ \ \ \Lambda_0(x) = \sum_{k \neq 0} \dfrac{|W_k(x)|^2}{|k|^2}
\end{equation*}
modulo an error of order $\epsi^3$. This remarkable result provided
further motivation for our investigation. One of the main consequences is \cite[Corollary $3.7$]{DucVulWei1}: in the case $d=1, W_0 \equiv 0$ and for $\epsi$ small enough a ground state emerges from the edge of the continuous spectrum of $D_x^2$, with energy $\lambda_\epsi$ given by
\begin{equation}\label{eq:1k}
\lambda_\epsi = -\dfrac{\epsi^4}{4} \left(\int_\R \Lambda_0(x) dx\right)^2 + O(\epsi^5).
\end{equation}
Theorem \ref{cor:3q} refines \eqref{eq:1k}. Since the functions $u,v$ of \eqref{eq:12ac} are given by $u=v=1/\sqrt{2}$ the energy of the bound state admits the expansion
\begin{equation*}
\lambda_\epsi = -\dfrac{\epsi^4}{4} \left(\int_\R \Lambda_0(x) dx\right)^2 - \dfrac{\epsi^5}{4} \int_\R \Lambda_0(x) dx \int_\R \Lambda_1(x) dx + O(\epsi^6),
\end{equation*}
and in fact $\lambda_\epsi$ is even a smooth function of $\epsi$. In \S \ref{sub:4} we compare numerically the efficiency of the effective potential $V_{\eff,\epsi}$ derived here compared to the efficiency of the effective potential derived in \cite{DucVulWei1}.

\subsection{Numerical results}\label{sub:4}
Let $W$ be the smooth function on $\R \times \Tt^1$ defined by
\begin{equation*}
W(x,y) = \exp\left(-\dfrac{x^2}{1-x^2}\right)\1_{[-1,1]}(x)\left(1 + 2 \cos(x/2+y)\right). 
\end{equation*}
Let $V_\epsi$ be given by \eqref{eq:1a} and $\Lambda_0, \Lambda_1$ the potentials defined in Theorem \ref{cor:3q}. Thanks to a Matlab simulation whose code was transferred to us by Duch\^ene, Vuki\'cevi\'c and Weinstein we computed numerically the transmission coefficients $t_\epsi$ of $V_\epsi$, $t_\epsi^1$ of $V_{\eff,\epsi}^1 = W_0-\epsi^2 \Lambda_0$ (the effective potential as derived in \cite{DucVulWei1}) and $t_\epsi^2$ of $V_{\eff,\epsi}^2 = W_0-\epsi^2\Lambda_0-\epsi^3\Lambda_1$ (the improved effective potential derived here). In Figure \ref{fig:3} we plotted the graphs of $|t_\epsi - t^j_\epsi|$ for different values of $\epsi$ and $j=1,2$. For $\epsi > 0.1$ neither the approximation of $t_\epsi$ by $t_\epsi^1$ nor $t_\epsi^2$ give satisfying results. For $\epsi \in [0.01,0.1]$ it is much better but we still cannot see the improvements induced by chosing $V_{\eff,\epsi}^2$ instead of $V_{\eff,\epsi}^1$. For $\epsi < 0.01$ the approximation of $t_\epsi$ by $t_\epsi^2$ instead of $t_\epsi^1$ gives better results.

\begin{figure}
    \scalebox{0.93}{\includegraphics{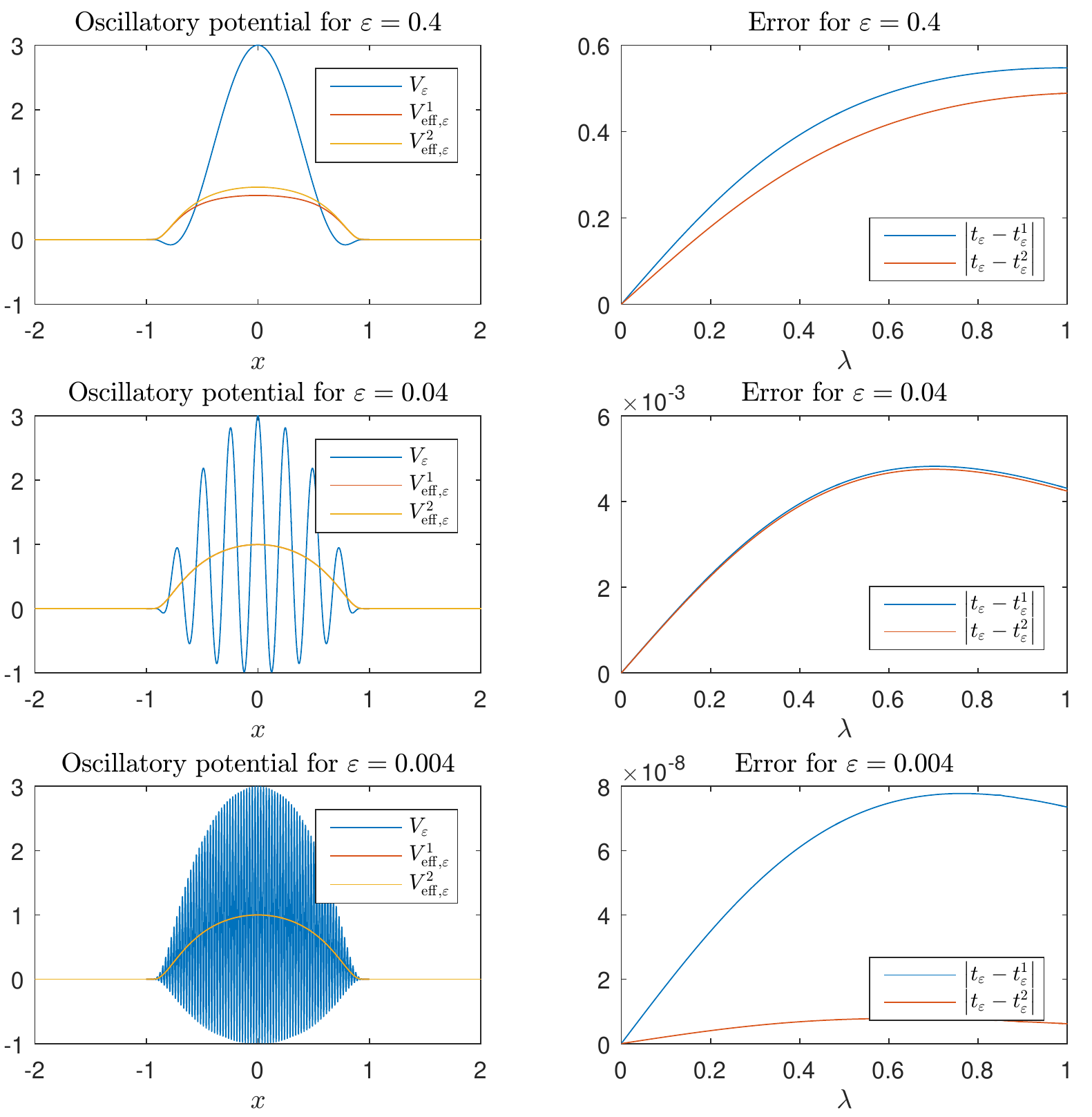}}
    \caption{Oscillatory potential and errors in approximating the transmission coefficient of $V_\epsi$ by the transmission coefficient of $V_{\eff,\epsi}^j$ for different values of $\epsi$ and $j=1,2$.}\label{fig:3}
\end{figure}

\subsection{Plan of the paper} We organize the paper as follows. In \S \ref{sec:2} we focus on the case $W_0 \equiv 0$ and we prove Theorem \ref{thm:1}. The proof relies mainly on an application of the Lippman-Schwinger principle combined with integration by parts. In \S \ref{subsec:8} we construct a step-like potential $V_\epsi$ whose resonances are zeros of a $2 \times 2$ explicit determinant. Uniform estimates on this determinant and arguments from complex analysis show that $V_\epsi$ admits a resonance $\lambda_\epsi \sim i\ln(\epsi)$.

In \S \ref{sec:3} we apply Theorem \ref{thm:5} to prove that resonances of potentials of the form \eqref{eq:1a} admit an expansion in powers of $\epsi$. We compute the first terms in the expansion using a trace estimate. Then we show that resonances of $V_\epsi$ are comparable to the one of the effective potential $V_{\eff,\epsi}$ by comparing two Fredholm determinants. We then prove Theorem \ref{cor:3s} using complex analysis arguments.

The section \ref{sec:5} consists in the proof of Theorem \ref{thm:5}. It is by far the hardest part of the paper. We first describe how an expansion of the determinant $D_{V_\epsi}(\lambda)$ in powers of $\epsi$ can be reduced to an expansion on the trace of an operator that takes a complicated form. We split this operator into two parts in a natural way.  By arguments of combinatorial nature we will prove that the first part is negligible as $\epsi \rightarrow 0$ and therefore produces no term in the expansion of $D_{V_\epsi}$. We will deal with the second part essentially by deriving an operator-valued expansion of $e^{ik\bullet/\epsi} R_0(\lambda) e^{-ik\bullet/\epsi}$ in powers of $\epsi$. The operators in this expansion will produce all the terms in the expansion of $D_V$. The expression of the coefficients in the expansion is theoretically traceable directly from the proof. We compute the first few terms. In dimension one the pole of $R_0(\lambda)$ at $\lambda=0$ will cause some trouble. We will overcome these difficulties by arguments specific to the one-dimensional case but that still rely on trace and determinant computations rather than on ODE techniques.

\subsection{Notation}\label{subsec:1} From now on we drop the subscript $\epsi$ and we fix $L > 0$. Given a function $W \in L^\infty_0(\Bb^d(0,L) \times \Tt^d, \C)$, $V$ is the function associated to $W$ by \eqref{eq:1a}. We will use the following notation:
\begin{itemize}
\item $X^d$ is the set equal to $\C \setminus \{0\}$ when $d=1$ and equal to $\C$ when $d \geq 3$.
\item Any time $\pm$ or $\mp$ appears in an equation, this equation has two meanings: one for the upper subscripts, one for the lower one. For instance, $f(x) = \mp 1$ for $\pm x \geq 1$ means $f(x) = -1$ for $x \geq 1$ and $f(x) = 1$ for $-x \geq 1$.
\item If $x \in \R$, $x_- = \max(0,-x)$.
\item For $x \in \R^n$, $\lr{x} = (1+|x|^2)^{1/2}$.
\item If $z \in \C$ and $r > 0$, $\Dd(z,r)$ denotes the set of $w \in \C$ with $|z-w| < r$.
\item If $x \in \R^d$ and $L > 0$, $\Bb^d(x,L)$ denotes the set of $y \in \R^d$ with $|x-y| < L$. $\Tt^d$ is the $d$-dimensional torus $\R^d/(2\pi \Z)^d$.
\item Let $\HH$ be a space of functions on an open set $\UU \subset \R^d$. We write $f \in \HH_0$ if $f$ belongs to $\HH$ and has compact support in $\UU$ and $f \in \HH_\loc$ if for every $\rho \in C_0^\infty(\R^d)$, $\rho f \in \HH$.
\item For a potential $\VV$, $\Res(\VV)$ is the set of resonances of $\VV$. If $\lambda \in \Res(\VV)$, $m_\VV(\lambda)$ is the geometric multiplicity of $\lambda$ defined by
\begin{equation*}
m_\VV(\lambda) = \text{rank} \oint_\lambda R_\VV(\mu) d\mu.
\end{equation*}
\item If $\HH_1, \HH_2$ are two Hilbert space, we denote by $\BB(\HH_1,\HH_2)$ (resp. $\LL(\HH_1,\HH_2)$) the space of bounded (resp. trace class) operators from $\HH_1$ to $\HH_2$ and by $\BB(\HH_1)$ (resp. $\LL(\HH_1)$) the space of bounded (resp. trace class) operators from $\HH_1$ to itself. If $\HH_1 = L^2(\R^d,\C)$ we simply write $\BB = \BB(\HH_1)$ and $\LL = \LL(\HH_1)$.
\item If $f$ is a function on $\R^d$, $\hat{f}$ and $\FF f$ both denote the Fourier transform of $f$:
\begin{equation*}
\FF f(\xi) = \hat{f}(\xi) = \dfrac{1}{(2\pi)^{d/2}} \int_{\R^d} f(x) e^{-ix\xi}dx.
\end{equation*}
\item We define $H^s(\R^d)$ the space of complex-valued functions $f$ with $\lr{\xi}^s \hat{f}(\xi) \in L^2(\R^d)$. If $s$ is an integer we define $W^s(\R^d)$ the space of functions with $s$ derivatives in $L^\infty(\R^d)$ and we write $|\cdot|_{W^s} = \| \cdot \|_s$. Similarly $W^s_0(\Bb^d(0,L))$ is the space of functions in $W^s(\R^d)$ with support contained in $\Bb^d(0,L)$.
\item For $k \in \Z^d$, $\ek$ denotes the multiplication operator by the function $e^{ikx/\epsi}$.
\item $\rho$ denotes a smooth function that is $1$ on $\Bb^d(0,L)$ and $0$ outside $\Bb^d(0,L+1)$.
\item The operator $D$ is $-i\p_x$. It is a vector-valued operator in dimension $d > 1$. For $k=(k_1, ..., k_d) \in \Z^d$, $k \cdot D$ is the operator $k_1 D_{x_1} + ... + k_d D_{x_d}$.
\item In general if $A(\lambda)$ is a family of operators depending on $\lambda$ we will write $A$ for $A(\lambda)$ unless there is a possible confusion.
\end{itemize}

\smallsection{Acknowledgment} We would like to thank Maciej Zworski for his help and guidance. We also thank Michael Weinstein, Vincent Duch\^ene and Iva Vuki\'cevi\'c for stimulating discussions and for sharing the Matlab codes leading to Figure \ref{fig:3}. This research was partially supported by NSF grant DMS-1500852 and the Fondation CFM pour la recherche.

\section{Resonance escaping in the case $W_0 \equiv 0$}\label{sec:2}
In this part we start with preliminary estimates that will be used all along the paper. Then we prove Theorem \ref{thm:1} and construct in \S \ref{subsec:8} an example of potential that proves that this theorem is optimal.

\subsection{Preliminaries}\label{subsec:6}

For $\VV \in L^\infty_0(\Bb^d(0,L),\C)$ we define $K_\VV$ the operator $\rho R_0(\lambda) \VV$. We start by the following preliminary:

\begin{lem}\label{lem:1h} For all $\az, \beta \in \{0,1,2\}^d$ with $|\az|+|\beta| \leq 2$ and for all $\VV \in W^{|\beta|}_0(\Bb^d(0,L),\C)$, 
\begin{equation*}
\left|D^\az K_\VV D^\beta\right|_{\BB} \leq \systeme{C \lr{\lambda}^{\az+\beta} |\lambda|^{-1} e^{2L(\Im \lambda)_-} \|\VV\|_{|\beta|} \ \ \text{ if } d =1, \\
C \lr{\lambda}^{|\az|+|\beta|-1} e^{2L(\Im \lambda)_-} \|\VV\|_{|\beta|} \ \ \  \text{ if } d \geq 3.}
\end{equation*}
The constant $C$ depends on $d(\supp(\VV), \p \Bb^d(0,L))$ only.
\end{lem}

Such estimates are proved in \cite[Theorem $2.1$]{DyaZwo} and follow from Schur's test. We recall that $X_d = \C$ if $d \geq 3$ and $X_1=\C \setminus \{0\}$. The following lemma characterizes resonances of a potential $\VV$ via a Lippman-Schwinger equation.

\begin{lem}\label{lem:1e} Let $\VV \in L^\infty_0(\Bb^d(0,L),\C)$. $\lambda \in X_d$ is a resonance of $\VV$ if and only if there exists $0 \neq u \in L^2$ such that $u = - K_\VV u$.
\end{lem}

\begin{proof} For $\lambda \in \C$ if $d \geq 3$ and $\lambda \in \C \setminus \{0\}$ the operator $K_\VV$ is compact. Thus $\Id + K_\VV$ is injective if and only if $\Id + K_\VV$ is invertible. For $\Im \lambda \gg 1$ we can invert $\Id + R_0(\lambda) \VV$ via Neumann series. Moreover, 
\begin{align*}
R_\VV(\lambda) & = \left( \Id + R_0(\lambda) \VV \right)^{-1} R_0(\lambda)  = \left(\sum_{n = 0}^\infty (-R_0(\lambda) \VV)^n\right) R_0(\lambda) \\
    & = \left(\sum_{n = 0}^\infty (-K_\VV)^n + (1-\rho) \sum_{n = 1}^\infty (-R_0(\lambda) \VV)^n \right)R_0(\lambda) \\
    & = \left( \Id + (1-\rho)\sum_{n = 1}^\infty (-R_0(\lambda) \VV)^n (\Id + K_\rho)  \right) \left(\Id + K_\VV \right)^{-1} R_0(\lambda) \\
    & = \left( \Id + (1-\rho)\sum_{n = 1}^\infty (-R_0(\lambda) \VV)^n - (-R_0(\lambda) \VV)^{n+1}  \right) \left(\Id + K_\VV \right)^{-1} R_0(\lambda) \\
    & = \left( \Id - (1-\rho)R_0(\lambda) \VV \right) \left( \Id - \left(\Id + K_\VV \right)^{-1} K_\VV \right) R_0(\lambda). 
\end{align*}
The operator $R_0(\lambda)$ meromorphically continues to $\C$ as an operator $L^2_\comp$ to $H^2_\loc$ while the operator $\left(\Id + K_\VV \right)^{-1}$ meromorphically continues to $\C$ as an operator $L^2$ to $L^2$. Thus the identity
\begin{equation}\label{eq:1b}
R_\VV(\lambda) = \left( \Id - (1-\rho)R_0(\lambda) \VV \right) \left( \Id - \left(\Id + K_\VV \right)^{-1} K_\VV \right) R_0(\lambda)
\end{equation}
initially valid for $\Im \lambda \gg 1$ meromorphically continues to all of $\C$. The poles of the RHS are precisely the set of $\lambda$ such that $\Id + K_\VV$ is not invertible (apart from $\lambda = 0$ in dimension one) while the poles of the LHS are the resonances of $\VV$. This proves the lemma.\end{proof}

\subsection{Escaping of resonances.}\label{subsec:7} We prove here Theorem \ref{thm:1} in the case $d=1$. Assume that \eqref{eq:9f} holds. If $\lambda \neq 0$ is a resonance of $V$ then by Lemma \ref{lem:1e} there exists $u$ such that $u = -K_V u$ and $|u|_2=1$. It satisfies the \textit{a priori} estimate
\begin{equation}\label{eq:1oo}
|u|_{H^1} = |K_V u|_{H^1} \leq  |K_V|_{\BB(H^1,L^2)} |u|_2 \leq C\dfrac{\lr{\lambda}e^{2L(\Im\lambda)_-}}{|\lambda|}|W|_\infty |u|_2,
\end{equation}
in particular it belongs to $H^1$. The well-known estimate $|fg|_{H^1} \leq |f|_{H^1} |g|_{H^1}$ (valid in dimension one) implies by duality that $|fg|_{H^{-1}} \leq |f|_{H^1}|g|_{H^{-1}}$. The bound \eqref{eq:1oo} yields
\begin{align}\label{eq:1l}\begin{split}
|u|_2 = |K_V u|_2 & \leq |K_\rho|_{\BB(H^{-1}, L^2)} |Vu|_{H^{-1}} \\ & \leq C \dfrac{\lr{\lambda}e^{2L(\Im \lambda)_-}}{|\lambda|} |V|_{H^{-1}} |u|_{H^1}  \leq C \dfrac{\lr{\lambda}^2e^{4L(\Im \lambda)_-}}{|\lambda|^2} |V|_{H^{-1}} |W|_\infty |u|_2.\end{split}
\end{align}
To estimate $|K_\rho|_{\BB(H^{-1},L^2)}$ we used the adjoint bound i.e. we estimated $|K_\rho(-\olambda)|_{\BB(L^2,H^1)}$ thanks to Lemma \ref{lem:1h}. We claim that $|V|_{H^{-1}} \leq \epsi^s |W|_{X^s}$, where $|W|_{X^s} = \sum_{k \neq 0} |k|^{-s} |W_k|_{H^s}$. Indeed using that $|\lr{\xi}^s \widehat{W_k}|_2 = |W_k|_{H^s}$ and $|V|_{H^{-1}} = |\lr{\xi}^{-1}\hat{V}|_2$ we have 
\begin{align*}
|V|_{H^{-1}}   & = \left| \lr{\xi}^{-1} \sum_{k \neq 0} \widehat{W_k}(\xi -k/\epsi) \right|_2 \\
   & \leq \sum_{k \neq 0} \left| \lr{\xi}^{-1}\lr{\xi -k/\epsi}^{-s}\lr{\xi -k/\epsi}^s\widehat{W_k}(\xi -k/\epsi) \right|_2 \\
   & \leq \sum_{k \neq 0} |\lr{\xi}^{-1}\lr{\xi -k/\epsi}^{-s}|_\infty |W_k|_{H^s} \\
   & \leq \sum_{k \neq 0} |\lr{\xi}^{-s}\lr{\xi -k/\epsi}^{-s}|_\infty |W_k|_{H^s}   \leq C\sum_{k \neq 0} \lr{k/\epsi}^{-s} |W_k|_{H^s} \leq \epsi^{s} |W|_{X^s}.
\end{align*}
In the last line we used Peetre's inequality: for every $x,y \in \R^d$ and $t \geq 0$ there exists a constant $C$ such that 
\begin{equation}\label{eq:7f}
\lr{x}^{-t} \lr{y}^{-t} \leq C\lr{x-y}^{-t}.
\end{equation}
Now combining $u=-K_Vu$ and $|u|_2=1$ with the estimate \eqref{eq:1l} we get
\begin{equation*}
1 \leq C \epsi^s \dfrac{ \lr{\lambda}^2 e^{4L(\Im \lambda)_-}}{|\lambda|^2} |W|_{X^s}^2.
\end{equation*}
Hence either $|\lambda| \leq 1$ and then $|\lambda| \leq c\epsi^{s/2}$ for some constant $c$; or $|\lambda| \geq 1$ and
\begin{equation*}
\Im \lambda \leq \dfrac{1}{4L}\ln\left( C |W|_{X^s}^2 \right) - \dfrac{s}{4L} \ln(\epsi^{-1}).
\end{equation*}
This proves Theorem \ref{thm:1} for $d=1$.

We next prove the theorem in dimension $d \geq 3$. In this case the inequality $|fg|_{H^1} \leq |f|_{H^1} |g|_{H^1}$ no longer holds and we must find another way around. Let $W$ such that $W_0 \equiv 0$ and \eqref{eq:9f} holds 
and $u \neq 0$ with $|u|_2=1$ and
\begin{equation}\label{eq:1o}
u =-K_Vu =  -\sum_{k \neq 0} K_{W_k} e^{ik\bullet/\epsi} u. 
\end{equation}
As in the case $d=1$ $u$ satisfies the \textit{a priori} estimate $|u|_{H^1}  \leq C e^{C(\Im \lambda)_-}  |W|_\infty |u|_2$. Noting that
\begin{equation*}
\ek = \dfrac{\epsi}{|k|}[P_k,\ek] \ \text{ where } \ P_k = \dfrac{k_1 D_{x_1} + ... + k_d D_{x_d}}{|k|}
\end{equation*}
we obtain the commutator identity
\begin{align*}
\epsi^{-1} |k| K_{W_k} \ek  = K_{W_k} P_k \ek -  K_{W_k}  \ek P_k.
\end{align*}
Consequently
\begin{align}\label{eq:1c}\begin{split}
\epsi^{-1} |k| \left|K_{W_k} \ek u \right|_2 & \leq |K_{W_k} P_k \ek u|_2 + |K_{W_k}  \ek P_k u|_2 \\
 & \leq |K_{W_k} P_k|_{\BB} |u|_2 + |K_{W_k}|_{\BB} |P_k u|_2 \\
 & \leq C e^{2L (\Im \lambda)_-} \|W_k\|_1 |u|_2 + C e^{C (\Im \lambda)_-} |W_k|_\infty |u|_{H^1} \\
 & \leq  C e^{4L (\Im \lambda)_-} \|W_k\|_1(1+|W|_\infty) |u|_2.\end{split}
\end{align}
From the second to the third line we used the estimates of Lemma \ref{lem:1h}. From the third to the fourth line we used \eqref{eq:1oo}. Sum \eqref{eq:1c} over $k \in \Z^d \setminus \{0\}$ to obtain
\begin{equation*}
|u|_2 = |K_V u|_2 \leq C  \epsi e^{4L(\Im \lambda)_-} (1+|W|_\infty) \left( \sum_{k \neq 0} \dfrac{\|W_k\|_1}{|k|}\right) |u|_2.
\end{equation*}
It follows that
\begin{equation*}
1 \leq C  \epsi e^{4L(\Im \lambda)_-} (1+|W|_\infty) \left( \sum_{k \neq 0} \dfrac{\|W_k\|_1}{|k|}\right)
\end{equation*}
which implies an upper bound on $\Im \lambda$ of the required form. This ends the proof of Theorem \ref{thm:1}.

\subsection{Construction of an optimal potential}\label{subsec:8}
Here we show that the rate of decay of imaginary parts of resonances of $V_\epsi$ provided by Theorem \ref{thm:1} is optimal in dimension $1$. We construct a function $W$ with $W_0 \equiv 0$ satisfying \eqref{eq:9f} such that the potential $V$ defined by \eqref{eq:1a} has a resonance $\lambda_\epsi \sim -i \ln(\epsi^{-1})$ with $\epsi = \pi/(2n)$. Define $W$ by
\begin{equation*}
W(x,y) = \1_{[-1/2,1/2]}(x) \left( \1_{[0,\pi]}(y) - \1_{[-\pi,0]}(y) \right).
\end{equation*}
The $k$-th Fourier coefficient of $W$ is given by
\begin{equation*}
W_k(x) = \systeme{ 0 & \text{ if } k \text{ is even,} \\ \dfrac{2}{i \pi k} \1_{[-1/2,1/2]}(x) & \text{ if } k \text{ is odd.} }
\end{equation*}
The function $\1_{[-1/2,1/2]}$ belongs to $H^{1/2-\delta}$ for all $1/2 > \delta > 0$ and
\begin{equation*}
\sum_{k \neq 0 } |k|^{-1/2+\delta}|W_k|_{H^{1/2-\delta}} \leq c_\delta \sum_{k \neq 0} |k|^{-3/2+\delta} < \infty.
\end{equation*}
Therefore $W$ satisfies \eqref{eq:9f} for every $s \in (0,1/2)$. The potential $V$ associated to $W$ by \eqref{eq:1a} is plotted on Figure \ref{fig:2}.

\begin{center}
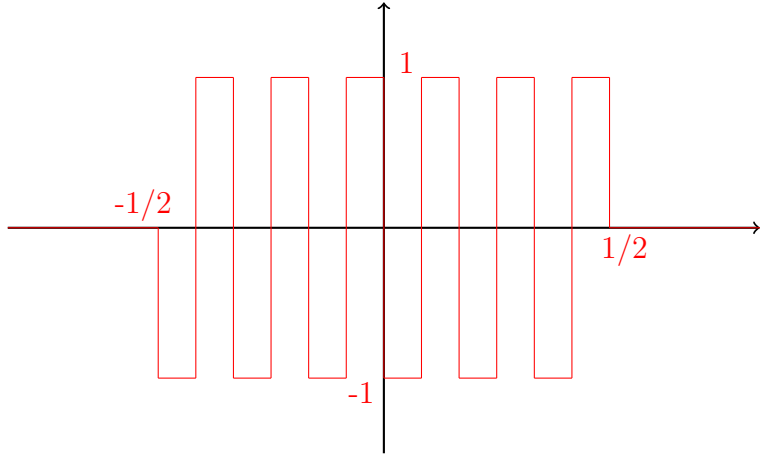
\begin{figure}
\begin{tikzpicture}
\draw[thick,->] (-5,0) -- (5,0) node[anchor=north east] {}; 
\draw[thick,->] (0,-3)  -- (0,3) node[anchor=north east] {};

\draw[red] (-5,0) -- (-3,0);
\draw[red] (5,0) -- (3,0);

\draw[red] (-3,-2) -- (-2.5,-2);
\draw[red] (-2,-2) -- (-1.5,-2);
\draw[red] (-1,-2) -- (-.5,-2);
\draw[red] (0,-2) -- (.5,-2);
\draw[red] (1,-2) -- (1.5,-2);
\draw[red] (2,-2) -- (2.5,-2);

\draw[red] (3,2) -- (2.5,2);
\draw[red] (2,2) -- (1.5,2);
\draw[red] (1,2) -- (.5,2);
\draw[red] (0,2) -- (-.5,2);
\draw[red] (-1,2) -- (-1.5,2);
\draw[red] (-2,2) -- (-2.5,2);

\draw[red] (-3,0) -- (-3,-2);
\draw[red] (-2.5,-2) -- (-2.5,2);
\draw[red] (-2,-2) -- (-2,2);
\draw[red] (-1.5,-2) -- (-1.5,2);
\draw[red] (-1,-2) -- (-1,2);
\draw[red] (-.5,-2) -- (-.5,2);
\draw[red] (0,-2) -- (0,2);
\draw[red] (.5,-2) -- (.5,2);
\draw[red] (1,-2) -- (1,2);
\draw[red] (1.5,-2) -- (1.5,2);
\draw[red] (2,-2) -- (2,2);
\draw[red] (2.5,-2) -- (2.5,2);
\draw[red] (3,0) -- (3,2);

\node[red] at (-3.2,0.3) {-1/2};
\node[red] at (3.2,-0.3) {1/2};
\node[red] at (0.3,2.2) {1};
\node[red] at (-.3,-2.2) {-1};
\end{tikzpicture}
\caption{The potential $V$ for $\epsi = \pi/12$.}\label{fig:2}
\end{figure}
\end{center}

We next characterize resonances of $V$ as zeros of a certain $2 \times 2$ determinant.

\begin{lem} Let $A_\pm$ be the matrix
\begin{equation}\label{eq:1m}
A_\pm = \matrice{0 & 1 \\ \pm 1 -\lambda^2 & 0}.
\end{equation}
Then $\lambda \neq 0$ is a resonance of $V$ for $\epsi = \pi/(2n)$ if and only if $D(\lambda) = 0$ where
\begin{equation*}
D(\lambda) = \Det \left( \left(e^{A_+/2n} e^{A_-/2n} \right)^n \matrice{1 \\ -i\lambda}, \matrice{ 1 \\ i\lambda } \right).
\end{equation*}
Here $\Det(a,b)$ denotes the determinant of two vectors $a, b$ of $\C^2$.
\end{lem}

\begin{proof} We recall that since $d = 1$, $\lambda \neq 0$ is a resonance of $V$ if and only if there exists a non zero function $u \in H^2_\loc$ with
\begin{equation*}
\left\{ \begin{matrix}
-u'' + Vu -\lambda^2 u = 0  & \\
u(x) = a_\pm e^{\pm i \lambda x}, \ \ & \pm x \gg 1
\end{matrix} \right.
\end{equation*}
see \cite[Theorem 2.4]{DyaZwo}. Using standard uniqueness results for ODEs $\lambda \neq 0$ is a resonance of $V$ if and only if there exists $a \in \C$ such that the boundary problem
\begin{equation}\label{eq:3a}
\systeme{-u'' + Vu - \lambda^2 u = 0, \\ u(-1/2)=1, \ u'(-1/2) = -i\lambda, \\ u(1/2)=a, \ u'(1/2)=ia\lambda}
\end{equation}
admits a non-zero solution $u$ in $H^2_\loc$. The ODE
\begin{equation*}
\systeme{-u'' + V u - \lambda^2 u = 0, \\ u(-1/2)=1, \ u'(-1/2) = -i\lambda}
\end{equation*}
admits a unique solution $u \in H^2_\loc$. The coefficients of the ODE are constant equal to $\pm 1$ on intervals of length $\pi/(2n)$. Hence $u$ can be explicitly computed using a matrix exponential. A direct calculation shows that
\begin{equation}\label{eq:3b}
\matrice{ u(1/2) \\ u'(1/2) } = \left(e^{A_+/2n} e^{A_-/2n} \right)^n \matrice{1 \\ -i\lambda}
\end{equation}
where $A_\pm$ are the matrices given by \eqref{eq:1m}. Putting together \eqref{eq:3a} and \eqref{eq:3b} $\lambda \neq 0$ is a resonance if and only if there exists $a$ such that
\begin{equation*}
a \matrice{ 1 \\ i\lambda } = \left(e^{A_+/2n} e^{A_-/2n} \right)^n \matrice{1 \\ -i\lambda}, 
\end{equation*}
that is, if and only if $D(\lambda) = 0$. This ends the proof.\end{proof}

In order to prove that $V_{\pi/(2n)}$ has a resonance $\lambda_n \sim -i\ln(n)$ we study asymptotics of $D(\lambda)$ uniform in the region $\{ (\lambda,n) \ : |\lambda| = O(\ln(n)) \}$. By the Baker-Hausdorff-Campbell formula, there exists a matrix $Z_n \in M_2(\C)$ such that $e^{Z_n} = e^{A_+/n} e^{A_-/n}$. Its asymptotic development is
\begin{equation*}
Z_n = \dfrac{A_++A_-}{2n} + \dfrac{1}{8n^2} [A_+,A_-] + \sum_{m \geq 3} \dfrac{1}{(2n)^m} P_m(A_+,A_-).
\end{equation*}
The terms $P_m(X,Y)$ are homogeneous polynomial of degree $m$ in the non-commuting variables $X, Y$. The expansion converges as long as $|A_+|<2n, |A_-| < 2n$ -- see \cite{BCH}. This is realized as long as $|\lambda| = o(\sqrt{n})$, hence when $\lambda = O(\ln(n))$. It yields
\begin{equation*}
Z_n = \dfrac{A_++A_-}{2n} + \dfrac{1}{8n^2} [A_+,A_-] + O\left(n^{-3} \lambda^6\right) \text{ when }
 \lambda = O(\ln(n)).
\end{equation*}
Therefore
\begin{align*}
e^{nZ_n} & = \exp\left(\dfrac{A_++A_-}{2} + \dfrac{1}{8n}[A_+,A_-] + O\left(n^{-2} \lambda^6\right)\right) \\ & = \exp\left(\dfrac{A_++A_-}{2} + \dfrac{1}{8n}[A_+,A_-]\right)\left( 1+O\left(n^{-2} \lambda^6\right) \right).
\end{align*}
A direct computation leads to
\begin{equation*}
\dfrac{A_++A_-}{2} + \dfrac{1}{8n}[A_+,A_-] = \matrice{ -1/4n & 1 \\ -\lambda^2 & 1/4n}.
\end{equation*}
The eigenvalues are $\pm \nu, \ \  \nu=i\sqrt{\lambda^2-(4n)^{-2}}$ and therefore
\begin{equation*}
\dfrac{A_++A_-}{2} + \dfrac{1}{8n}[A_+,A_-] = \Omega \Delta \Omega^{-1} \text{ with } \Delta = \matrice{-\nu & 0 \\ 0 & \nu} \text{ and } \Omega = \matrice{ 1 & 1 \\ -\nu + (4n)^{-1} & \nu + (4n)^{-1}}.
\end{equation*}
Another direct computation gives
\begin{align*}
 D(\lambda) & = \Det(\Omega) \Det\left( e^{\Delta} \Omega^{-1} \matrice{1 \\ -i\lambda}, \Omega^{-1} \matrice{1 \\ i\lambda} \right)\left(1 + O\left(n^{-2}\lambda^6\right)\right) \\
   & = -\dfrac{\lambda^2e^{-\nu}}{2\nu } \left(  \left(\dfrac{\nu}{i\lambda} + 1\right)^2 + (4n \lambda)^{-2}  - e^{2\nu}\left(\left(\dfrac{\nu}{i\lambda} - 1\right)^2 + (4n \lambda)^{-2}\right)\right)\left(1 + O\left(n^{-2}\lambda^6\right)\right) \\ 
   & =  -\dfrac{\lambda^2e^{-\nu}}{2\nu } \left( 4 + O\left((n\lambda)^{-2}\right) - \dfrac{e^{2i\lambda}}{(4n\lambda)^2} \left(1+O\left(n^{-2}\lambda^{-1}\right)\right) \right)\left(1 + O\left(n^{-2}\lambda^6\right)\right)
\end{align*}
as long as $\lambda = O(\ln(n))$. In order to investigate the behavior of zeros of $D(\lambda)$ we investigate first the behavior of zeros of the function $f$ given by
\begin{equation*}
f(\lambda) = 4 - \dfrac{e^{2i\lambda}}{(4n\lambda)^2}.
\end{equation*}

\begin{lem}\label{lem:1f} The zeros of $f$ are given by $\lambda_\nu^\pm = i\WW_\nu(\pm i/8n), \ \ \nu \in \Z
$ where $\WW_\nu$ is the $\nu$-th branch of the Lambert function -- see \cite{Lambert}. In particular as $n$ goes to infinity $\lambda^+_1 \sim -i \ln(n)$. Moreover, there exists $r_0$ (independent on $n$) such that for all $n$ large enough and $\te \in \Ss^1$,
\begin{equation}\label{eq:2c}
|f(\lambda_1^+ + r_0e^{i\te})| \geq 3r_0.
\end{equation}
\end{lem}

\begin{proof} The equation $f(\lambda) = 0$ is equivalent to
\begin{equation*}
-i\lambda e^{-i\lambda} = \pm \dfrac{i}{8n}.
\end{equation*}
Therefore zeros of $f$ are given by $-i\WW_\nu(\pm i/8n)$. From \cite[equation $(4.20)$]{Lambert} we obtain the asymptotic $\lambda^+_1 \sim -i \ln(n)$. In order to show the lower bound \eqref{eq:2c} we consider $r \in (0,1)$. We prove some estimates that are uniform in $n$ and $\te \in \Ss^1$ as $r \rightarrow 0$. The identity $f(\lambda_1^+)=0$ yields
\begin{equation*}
f(\lambda_1^++re^{i\te})  = 4- 4\left(\dfrac{e^{re^{i\te}}}{1+re^{i\te}/\lambda_1^+}\right)^2.
\end{equation*}
As $r \rightarrow 0$, $e^{re^{i\te}} = 1+re^{i\te} + o(r)$, therefore
\begin{equation*}
1 - \dfrac{e^{re^{i\te}}}{1+re^{i\te}/\lambda_1^+} = \dfrac{re^{i\te} (1-\lambda_1^+) + o(r)}{1+re^{i\te}/\lambda_1^+}.
\end{equation*}
For $n$ large enough we have $\lambda_1^+ \sim -i\ln(n)$ and thus a fortiori $|\lambda_1^+| \geq 2$. This implies
\begin{equation*}
\left| 1 - \dfrac{e^{re^{i\te}}}{1+re^{i\te}/\lambda_1^+} \right| \geq \dfrac{r/2+o(r)}{1+r/2} = r/2 + o(r).
\end{equation*}
Similarly,
\begin{equation*}
\left| 1 + \dfrac{e^{re^{i\te}}}{1+re^{i\te}/\lambda_1^+} \right| \geq 2 + O(r).
\end{equation*}
Therefore for $r$ small enough
\begin{equation*}
|f(\lambda_1^+ + re^{i\te})| \geq 4r + o(r) \geq 3r.
\end{equation*}
This completes the proof of the lemma. \end{proof}

For $\lambda \in \p\Dd(\lambda_1^+,r_0)$, $f(\lambda)$ is bounded from below uniformly as $n \rightarrow \infty$. Hence for $\lambda \in \p \Dd(\lambda_1^+,r_0)$,
\begin{equations*}
4 + O\left((n\lambda)^{-2}\right) - \dfrac{e^{2i\lambda}}{(4n\lambda)^2}\left(1+O\left(n^{-2}\lambda^{-1}\right)\right) = f(\lambda) \left(1+O\left(n^{-2}\lambda^{-1}\right)\right) \\
 = f(\lambda) \left(1+O\left(n^{-2}\ln(n)^{-1}\right)\right). 
\end{equations*}
This implies that for $\lambda \in \p \Dd(\lambda_1^+,r_0)$,
\begin{align*}
D(\lambda) & = -\dfrac{\lambda^2e^{-\nu}}{2\nu}f(\lambda)\left(1+O\left(n^{-2}\ln(n)^{-1}\right)\right)\left(1+O\left(n^{-2}\ln(n)^6\right)\right) \\
   & =-\dfrac{\lambda^2 e^{-\nu}}{2\nu}f(\lambda)\left(1+O\left(n^{-2}\ln(n)^6\right)\right).
\end{align*}
By Rouch\'e's theorem this is enough to ensure that for $n$ large enough, $D(\lambda)$ has exactly one zero on $\CC(\lambda_1^+,r_0)$. This proves that there exists a resonance behaving like $-i\ln(n)$.

\section{Applications of Theorem \ref{thm:5}}\label{sec:3}
Here we consider $W \in C_0^\infty(\Bb^d(0,L) \times \Tt^d,\C)$ and $V_\epsi$ given by \eqref{eq:1a}. We assume that Theorem \ref{thm:5} holds and we get directly to the applications. We prove that resonances of $V_\epsi$ in compact sets admit a full expansion as $\epsi \rightarrow 0$ (Theorem \ref{cor:3q}); that they can be well approximated by a small perturbation $V_{\eff,\epsi}$ of $W_0$ (Theorem \ref{cor:3}); and we give a description of the localization of resonances of $V_\epsi$ (Theorem \ref{cor:3s}).

\subsection{Expansion of resonances in powers of $\epsi$}\label{subsec:2} In this paragraph we prove Theorem \ref{cor:3q}. We start with the case $d \geq 3$ or $\lambda_0 \neq 0$.

\begin{proof}[Proof of Theorem \ref{cor:3q} assuming $d \geq 3$ or $\lambda_0 \neq 0$.] Let $\lambda_0$ be a simple resonance of $W_0$ with $\lambda_0 \neq 0$ if $d=1$. For $N > 0$ and $p=4N(d+N)$ consider $D_V(\lambda)$ given in \eqref{eq:1d}. This is a holomorphic function of $\lambda$ near $\lambda_0$. By Theorem \ref{thm:5} it converges to $D_{W_0}$ as $\epsi \rightarrow 0$ uniformly on a neighborhood of $\lambda_0$. Thus by Hurwitz's theorem $D_V$ has exactly one zero $\lambda_\epsi$ that converges to $\lambda_0$. It follows that for $\epsi$ small enough and $r_0$ small enough $\lambda_\epsi$ is the only resonance of $V$ on $\Dd(\lambda_0,r_0)$.

Define $f(\lambda,\epsi) = D_V(\lambda)$ if $\epsi \neq 0$ and $f(\lambda,0) = D_{W_0}(\lambda)$ otherwise. By Theorem \ref{thm:5} the function $f$ is of class $C^{N-1}$ in a neighborhood of $(\lambda_0,0)$. In addition since 
\begin{equation*}
\dd{f}{\lambda}(\lambda_0,0) = D_{W_0}'(\lambda_0) \neq 0
\end{equation*}
the implicit function theorem implies that the equation $f(\lambda,\epsi)=0$ has exactly one solution in a neighborhood of $(\lambda_0,0)$. Using uniqueness it must be $(\lambda_\epsi, \epsi)$. It follows that the function $\epsi \rightarrow \lambda_\epsi$ is $C^{N-1}$. As $N$ was arbitrary we conclude that $\epsi \rightarrow \lambda_\epsi$ is $C^\infty$ for $\epsi$ near $0$. Thus for all $N$, 
\begin{equation*}
\lambda_\epsi = \lambda_0 + \epsi c_1 + ... + \epsi^{N-1} c_{N-1} + O(\epsi^N), \ \  c_j \in \C.
\end{equation*} 

We now derive the values of $c_1, c_2, c_3$. Let $R_{W_0}(\lambda)$ be the meromorphic continuation of the operator $(-\Delta - \lambda^2 + W_0)^{-1}$. Since $\lambda_0$ is a simple resonance of $W_0$ there exists $u \in H^2_\loc(\R^d,\C)$, $v \in \DD'(\R^d,\C)$ such that 
\begin{equation*}
R_{W_0}(\lambda) = \dfrac{i u \otimes v}{\lambda - \lambda_0} + H(\lambda)
\end{equation*}
where $H(\lambda) : L^2_\comp \rightarrow H^2_\loc$ is a family of operators holomorphic near $\lambda_0$. Let $f$ be a smooth compactly supported function on $\R^d$. Since $R_{W_0}(\lambda)(P_V-\lambda^2)f=f$ we have
\begin{equation*}
0 = (iu \otimes v)(P_V-\lambda_0^2)f  = iu \lr{\ov, (P_V-\lambda_0^2)f}_{\DD'}  = iu \lr{(P_V-\lambda^2_0)^* \ov, f}_{\DD'}.
\end{equation*}
Since this is valid for arbitrary $f$ it yields $(P_V-\lambda^2_0)^* \ov=0$. Thus $v \in H^2_\loc$ and $(P_V-\lambda_0^2)v=0$ which implies $v+R_0(\lambda_0) W_0 v = 0$.

Let $\Pi_0$ be the operator $-i\rho (u \otimes v) W_0$. We claim that the family of operators
\begin{equation}\label{eq:12ab}
(-K_{W_0})^{p-2} \left( \Id + K_{W_0} \right)^{-1} - \dfrac{\Pi_0}{\lambda-\lambda_0}
\end{equation}
is holomorphic in a neighborhood of $\lambda_0$. Indeed since $(\Id + K_{W_0})^{-1} = \Id -\rho R_{W_0}(\lambda)W_0$ there exists a family of operators $B(\lambda)$ holomorphic near $\lambda_0$ such that
\begin{equation*}
(\Id + K_{W_0})^{-1} =   \dfrac{\Pi_0}{\lambda - \lambda_0}  + B(\lambda).
\end{equation*}
It leads to
\begin{equations*}
(-K_{W_0})^{p-2} \left( \Id + K_{W_0} \right)^{-1} - \dfrac{\Pi_0}{\lambda-\lambda_0} = (-K_{W_0})^{p-2} \left( \Id + K_{W_0} \right)^{-1} - (\Id + K_{W_0})^{-1} + B(\lambda) \\
   = -\left(\Id - (-K_{W_0})^{p-2}\right) \left( \Id + K_{W_0} \right)^{-1} + B(\lambda) 
   = - (\Id + ... + (-K_{W_0})^{p-3}) + B(\lambda).
\end{equations*}
This is as claimed holomorphic near $\lambda_0$.

Let $\Lambda \in L^\infty(\Bb^d(0,L),\C)$. We now compute the trace $\trace\left((-K_{W_0})^{p-2}(\Id + K_{W_0})^{-1} K_\Lambda\right)$ modulo a holomorphic function. Since the operator given by \eqref{eq:12ab} is holomorphic near $\lambda_0$ and trace class there exists a function $\varphi$ holomorphic near $\lambda_0$ such that
\begin{equation*}
\trace\left((-K_{W_0})^{p-2}(\Id + K_{W_0})^{-1} K_\Lambda\right) = \dfrac{\trace(\Pi_0K_\Lambda)}{\lambda-\lambda_0} + \varphi(\lambda).
\end{equation*}
Using $\Pi_0=-i \rho u \otimes v W_0$ and $v+R_0(\lambda_0) W_0 v = 0$ we get 
\begin{equations*}
\trace(\Pi_0 K_\Lambda)(\lambda_0)  = -i\int_{\R^d} \rho(x) u(x) v(y) W_0(y) R_0(\lambda_0,y,x) \Lambda(x) dx dy\\
  = -i\int_{\R^d}  u(x) \Lambda(x) \left( \int_{\R^d} R_0(\lambda_0,x,y)W_0(y) v(y) dy \right)dx \\
  = -i\int_{\R^d}  u(x) \Lambda(x) (R_0(\lambda_0) W_0v)(x)dx  = i\int_{\R^d}   \Lambda(x) u(x) v(x) dx.
\end{equations*}
It follows that
\begin{equation}\label{eq:11z}\begin{split}
\trace\left((-K_{W_0})^{p-2}(\Id + K_{W_0})^{-1} K_\Lambda\right) = \dfrac{i}{\lambda-\lambda_0}\left(\int_{\R^d}   \Lambda u v\right) + \varphi(\lambda).\end{split}
\end{equation}
Apply the formula \eqref{eq:11z} to $\Lambda = \epsi^2 \Lambda_0$ to obtain
\begin{align*}
D_{V}(\lambda) & = D_{W_0}(\lambda)\left( 1-  \trace\left((-K_{W_0})^{p-2}(\Id + K_{W_0})^{-1} K_{\epsi^2 \Lambda_0}\right) \right) + O(\epsi^3)\\ 
& = D_{W_0}(\lambda) \left(1-\dfrac{i\epsi^2}{\lambda-\lambda_0} \left(\int_{\R^d}   \Lambda_0 u v\right) - \epsi^2 \varphi_0(\lambda) \right) + O(\epsi^3).
\end{align*}
Here the function $\varphi_0$ is holomorphic near $\lambda_0$ and does not depend on $\epsi$. If $g$ is the holomorphic function such that $g(\lambda) (\lambda-\lambda_0) = D_{W_0}(\lambda)$ then
\begin{equation}\label{eq:1u}
D_V(\lambda) = g(\lambda)\left( \lambda-\lambda_0 - i\epsi^2\left(\int_{\R^d}   \Lambda_0 u v\right) - \epsi^2 (\lambda-\lambda_0) \varphi_0(\lambda)  \right) + O(\epsi^3).
\end{equation}
Note that as $\epsi \rightarrow 0$ we have $g(\lambda_\epsi) \rightarrow D_{W_0}'(\lambda_0) \neq 0$. Thus specializing the identity \eqref{eq:1u} at $\lambda=\lambda_\epsi$ leads to
\begin{equation*}
0 = \lambda_\epsi - \lambda_0 - i\epsi^2 \left(\int_{\R^d}   \Lambda_0 u v\right) - \epsi^2 (\lambda_\epsi-\lambda_0) \varphi_0(\lambda_\epsi) + O(\epsi^3).
\end{equation*}
Since $\lambda_\epsi - \lambda_0 = O(\epsi)$ and $\varphi_0(\lambda_\epsi) \rightarrow \varphi_0(\lambda_0)$ as $\epsi \rightarrow 0$ we obtain 
\begin{equation}\label{eq:1v}
\lambda_\epsi = \lambda_0 + i\epsi^2 \left(\int_{\R^d}   \Lambda_0 u v\right) + O(\epsi^3).
\end{equation}
This recovers the result of \cite{DucVulWei1}.

Now to get the second order correction we apply \eqref{eq:11z} successively to $\Lambda = \epsi^2 \Lambda_0$ and $\Lambda=\epsi^3 \Lambda_1$. The same operations as in the previous paragraph lead to
\begin{align*}
D_V(\lambda) & = D_{W_0}(\lambda)\left( 1-  \trace\left((-K_{W_0})^{p-2}(\Id + K_{W_0})^{-1} K_{\epsi^2 \Lambda_0+\epsi^3\Lambda} \right) \right) + O(\epsi^4) \\ 
& = g(\lambda)\left( \lambda-\lambda_0 -i \left(\int_{\R^d} \left( \epsi^2 \Lambda_0 + \epsi^3\Lambda_1\right) u v\right) - (\lambda-\lambda_0)(\epsi^2 \varphi_0(\lambda)+ \epsi^3\varphi_1(\lambda))\right) + O(\epsi^4)
\end{align*}
for a function $\varphi_1$ holomorphic near $\lambda_0$. Here again specialize this identity at $\lambda = \lambda_\epsi$ and use $g(\lambda_\epsi) \rightarrow g(\lambda_0) \neq 0$ to obtain
\begin{equation*}
0 = \lambda_\epsi - \lambda_0 - i \left(\int_{\R^d} \left( \epsi^2 \Lambda_0 + \epsi^3\Lambda_1\right) u v\right) - (\lambda_\epsi-\lambda_0)(\epsi^2 \varphi_0(\lambda_\epsi)+ \epsi^3\varphi_1(\lambda_\epsi)) + O(\epsi^4).
\end{equation*}
This time by \eqref{eq:1v} we know that $\lambda_\epsi - \lambda_0 = O(\epsi^2)$. It follows that
\begin{equation*}
\lambda_\epsi = \lambda_0 + i \epsi^2\left(\int_{\R^d}\Lambda_0 uv \right) + i\epsi^3 \left(\int_{\R^d} \Lambda_1 u v\right) + O(\epsi^4).
\end{equation*}
This proves the theorem.
\end{proof}

In the case $\lambda_0=0$ and $d=1$ we use the following refinement of Theorem \ref{thm:5}:

\begin{lem}\label{lem:2s} Let $W$ belong to $C_0^\infty([-L,L] \times \Tt^1, \C)$ and $V$ be given by \eqref{eq:1a}. There exists an entire function $h_V$ satisfying the following:
\begin{itemize}
\item[$(i)$] $\lambda_0$ is a resonance of $V$ of multiplicity $m$ if and only if it is a zero of $h_V$ of multiplicity $m$.
\item[$(ii)$] There exists $h_4, ..., h_{N-1}$ such that locally uniformly on  $\C$
\begin{equation*}
h_V(\lambda) = \lambda d_{W_0}(\lambda) \left( 1 - \trace\left( (\Id + K_{W_0})^{-1} K_\Lambda \right) \right) + \epsi^4 h_4(\lambda) + ... + \epsi^{N-1} h_{N-1}(\lambda) + O(\epsi^N)
\end{equation*}
where $d_{W_0}(\lambda) = \Det(\Id + K_{W_0})$ and $\Lambda$ is the potential given by
\begin{equation*}
\Lambda = \epsi^2 \Lambda_0 + \epsi^3 \Lambda_1 =  \epsi^2 \sum_{k \neq 0} \dfrac{W_k W_{-k}}{k^2} - 2 \epsi^3 \sum_{k \neq 0} \dfrac{W_k (DW_{-k})}{k^3}.
\end{equation*}
\end{itemize}
\end{lem}

We defer the proof of Lemma \ref{lem:2s} to \S \ref{subsec:4}. The proof of Theorem \ref{cor:3q} in the case $\lambda_0=0$ and $d=1$ is the same as in the case $d \neq 1$ or $\lambda_0 \neq 0$ using $h_V$ instead of $D_V$ and we skip the details. We end this part with a version of Theorem \ref{cor:3q} for resonances $\lambda_0$ of $W_0$ with higher multiplicity.

\begin{thm}\label{cor:3r} Assume that $W$ belongs to $C_0^\infty(\Bb^d(0,L) \times \Tt^d,\C)$ and that $\lambda_0$ is a resonance of $W_0$ with multiplicity $m$. Then in a neighborhood of $\lambda_0$ the potential $V_\epsi$ has exactly $m$ resonances $\lambda_{1,\epsi}, ..., \lambda_{m,\epsi}$ for $\epsi$ small enough. In addition for every $j \in [1,m]$ and $N >0$,
\begin{equation*}
\lambda_{j,\epsi} = \lambda_0 + c_{j,2} \epsi^{2/m} + c_{j,3} \epsi^{3/m} + ... + c_{j,N-1} \epsi^{(N-1)/m} + O(\epsi^{N/m}), \ \  c_{j,n} \in \C.
\end{equation*}
\end{thm} 

\begin{proof} Let $\lambda_0 \in X_d$ be a resonance of $W_0$ of multiplicity $m > 1$. Fix $N > 0$ and $p$, $D_V$ given by Theorem \ref{thm:5}. Since locally uniformly on $\C$ we have $D_V(\lambda) \rightarrow D_{W_0}(\lambda)$, by Hurwitz's theorem the function $D_V$ has exactly $m$ zeros (counted with multiplicity) converging to $\lambda_0$. These zeros admit a Puiseux expansion: there exists $c_{1,1}, ..., c_{m,N-1}$ such that the zeros $\lambda_{1,\epsi}, ..., \lambda_{m,\epsi}$ of $D_V$ near $\lambda_0$ are given by
\begin{equation*}
\lambda_{j,\epsi} = \lambda_0 + \epsi^{1/m} c_{j,1} + ... + \epsi^{(N-1)/m} c_{j,N-1} + O(\epsi^{N/m}).
\end{equation*}
Now since $D_V(\lambda) = D_{W_0}(\lambda) + O(\epsi^2)$, $c_{j,1} = 0$. In the case $\lambda_0=0$ in dimension one the proof can be modified by considering $h_V$ instead of $D_V$. This proves Theorem \ref{cor:3r}.\end{proof}

\subsection{Derivation of an effective potential}\label{sub:3} In this part we prove Theorem \ref{cor:3}. We start by giving a few preliminaries concerning trace class operators and Fredholm determinant. The reader can consult \cite[Chapter B]{DyaZwo} for a complete introduction. The singular values of a compact operator $X : \HH \rightarrow \HH$ are defined as the nonincreasing sequence $s_j(X) = \lambda_j((X^*X)^{1/2})$. In particular $s_0(X) = |X|_{\BB(\HH)}$. The singular values satisfy two remarkable inequalities. If $Y$ is another compact operator then for every $j,\ell$, 
\begin{equation*}
s_{j+\ell}(X+Y) \leq s_j(X) + s_\ell(Y),
\end{equation*}
\begin{equation*}
s_{j\ell}(XY) \leq s_j(X) s_\ell(Y).
\end{equation*}
We say that a compact operator $X$ is trace class if the sequence $s_j(X)$ is summable. The trace class norm of $X$ denoted by $|X|_\LL$ is the sum of the series. If $X$ trace class we can define the trace of $X$ and the Fredholm determinant $\Det(\Id + X)$. This determinant vanishes if and only if $\Id + X$ is not invertible. Recall that $X8d=\C$ for $d \geq 3$, $X_1=\C \setminus \{0\}$ and that $K_\VV = \rho R_0(\lambda) \VV$.

\begin{lem}\label{lem:1b} Let $\VV$ in $L^\infty(\Bb^d(0,L),\C)$. Uniformly on $\{\Im \lambda \geq 1\}$ and locally uniformly on $X_d$, $s_j(K_\VV) \leq C  |\VV|_\infty j^{-2/d}$. Consequently if $p \geq d$ is an integer the operator $K_\VV^p$ is trace class and locally uniformly in $X_d$, uniformly in $\{\Im \lambda \geq 1\}$, $|K_\VV^p|_{\LL} \leq C |\VV|_\infty^p$.
\end{lem}

\begin{proof} We combine \cite[Equation (B.3.9]{DyaZwo} with Lemma \ref{lem:1h}. This gives:
\begin{equation*}
s_j(K_\VV) \leq C j^{-2/d} |\lr{D}^2K_\VV|_\BB  \leq C  |\VV|_\infty j^{-2/d}.
\end{equation*}
This estimate works both locally uniformly on $X_d$ and uniformly on $\{\Im \lambda \geq 1\}$. In order to prove that the operator $K_\VV^p$ belongs to $\LL$ for $p \geq d$ it suffices to prove that the sequence of singular values $s_j(K_\VV^p)$ is summable. Using the properties of the singular values,
\begin{equation*} 
\sum_{j = 0}^\infty s_j(K_\VV^p)  \leq p\sum_{j=0}^\infty s_{pj}(K_\VV^p) \leq p \sum_{j=0}^\infty s_{j}(K_\VV)^p 
\leq C  |\VV|_\infty^p \sum_{j=0}^\infty j^{-2p/d}.
\end{equation*}
Since $p \geq d$ the series converges and the lemma follows. \end{proof}

This lemma implies that for $\VV \in L^\infty(\Bb^d(0,L),\C)$ the Fredholm determinant 
\begin{equation*}
D_\VV(\lambda) = \Det(\Id + \Psi(K_\VV)), \ \ \  \Psi(z) = (1+z) \exp\left(-z+ \dfrac{z^2}{2}- ...  +\dfrac{(-z)^{p-1}}{p-1} \right) - 1
\end{equation*}
is well defined when $\lambda \in X_d$ -- see \cite[Lemma 6.1]{Si77}. It is an entire function of $\lambda$ for $d \geq 3$ and is a meromorphic function of $\lambda$ with a pole at $\lambda=0$ for $d=1$. We now show the seemingly unknown:

\begin{lem}\label{lem:7} Let $W_0, \Lambda \in L^\infty(\Bb^d(0,L),\C)$. If $p \geq d$ and $D_{W_0+\epsi \Lambda}$ is the Fredholm determinant given by \eqref{eq:1d} then there exists $b_0, b_1, ...$ holomorphic functions of $\lambda \in X_d$ such that locally uniformly on  $X_d$,
\begin{equation*}
D_{W_0+\epsi \Lambda}(\lambda) = \sum_{j = 0}^\infty b_j(\lambda) \epsi^j.
\end{equation*} 
In addition $b_0(\lambda) = D_{W_0}(\lambda)$ and
\begin{equation*}
b_1(\lambda) = D_{W_0}(\lambda) \cdot \trace \left((\Id + K_{W_0})^{-1} (-K_{W_0})^{p-1} K_\Lambda\right).
\end{equation*}
\end{lem}

\begin{proof} Let $W_0, \Lambda \in L^\infty(\Bb^d(0,L),\C)$. By \cite[Theorem 3.3]{Si77} if $p \geq d$ and $\Psi$ is given by \eqref{eq:2d} the determinant 
$(\epsi, \lambda) \mapsto D_{W_0+\epsi \Lambda}(\lambda) = \Det(\Id + \Psi(K_{W_0+\epsi \Lambda}))$
is an entire function of $\epsi$ (with $\lambda \in X_d$ fixed) and a holomorphic function of $\lambda$ on $X_d$ (with $\epsi$ fixed). Thus by Hartogs's theorem it is analytic on $\C \times X_d$. Write a power expansion of $D_{W_0+\epsi \Lambda}$ as follows:
$D_{W_0+\epsi \Lambda}(\lambda) = \sum_{n=0}^\infty b_n(\lambda) \epsi^n$.
Since 
\begin{equation*}
b_n(\lambda) = \dfrac{1}{n!} \left.\dd{^n D_{W_0+\epsi\Lambda}}{\epsi^n}\right|_{\epsi=0}(\lambda)
\end{equation*}
the function $b_n$ is holomorphic on $X_d$. We next identify the coefficients $b_0(\lambda)$ and $b_1(\lambda)$.

Fix $m \geq d$ and assume that $\lambda \in \Dd(\lambda_0,1)$, $\Im \lambda_0 \gg 1$. By Lemma \ref{lem:1h} and Lemma \ref{lem:1b},
\begin{equation*}
\left|K_{W_0+\epsi \Lambda}^m \right|_\LL \leq \left|K_{W_0+\epsi \Lambda}^{m-d} \right|_\BB |K_{W_0+\epsi \Lambda}^d|_\LL  \leq \dfrac{C^m}{|\lambda|^{m-d}}.
\end{equation*}
It follows that the series 
\begin{equation*}
\sum_{m=p}^\infty (-1)^m\dfrac{K_{W_0+\epsi \Lambda}^m}{m}
\end{equation*}
converges absolutely in $\LL$ for $\Im \lambda \gg 1$ and in addition
\begin{equation}\label{eq:0b}
D_{W_0+\epsi \Lambda}(\lambda) = \exp \left(-\sum_{m=p}^\infty (-1)^m\dfrac{\trace\left(K_{W_0+\epsi \Lambda}^m\right)}{m}\right),
\end{equation}
see \cite[Theorem 6.2]{Si77}. If $d=1$ then $\trace(K_{W_0+\epsi \Lambda}) = \trace(K_{W_0}) + \epsi \trace(K_\Lambda)$. We now obtain a first order Taylor expansion of $\trace\left(K_{W_0+\epsi \Lambda}^m\right)$ for $m \geq d$. Using the binomial expansion, the cyclicity of the trace and the Taylor-Lagrange inequality,
\begin{equations}\label{eq:0a}
\trace\left(K_{W_0+\epsi \Lambda}^m\right) = \trace\left(K_{W_0}^m\right) + m \epsi \trace\left(K_{W_0}^{m-1}K_\Lambda\right) + r_m(\epsi), \\
|r_m(\epsi)| \leq \dfrac{1}{2} \sup_{\epsi' \in [0,\epsi]} \dd{^2\trace\left(K_{W_0+\epsi \Lambda}^m\right)}{\epsi^2}(\epsi').
\end{equations}
We claim that $|r_m(\epsi)| \leq \epsi^2$ for $\Im \lambda$ large enough. The function $\epsi \mapsto \trace\left(K_{W_0+\epsi \Lambda}^m\right)$ is holomorphic and satisfies
\begin{equation*}
\left| \trace\left(K_{W_0+\epsi \Lambda}^m\right) \right| \leq |K_{W_0+\epsi \Lambda}^m|_\LL^d |K_{W_0+\epsi \Lambda}^m|_\LL^{m-d} \leq \dfrac{C^m }{\lr{\lambda}^{m-d}} |W_0+\epsi \Lambda|^m_\infty. 
\end{equation*}
when $\Im \lambda \geq 1$. Therefore the Cauchy estimate for derivatives of holomorphic functions shows that $|r_m(\epsi)| \leq C^m \epsi^2 \lr{\lambda}^{m-d} (|W_0|_\infty+|\Lambda|_\infty)^m$ when $\Im \lambda \geq 1$. This proves the claim. \eqref{eq:0a} implies then
\begin{equations*}
\sum_{m=p}^\infty (-1)^m\dfrac{\trace\left(K_{W_0+\epsi \Lambda}^m\right)}{m} = \sum_{m=p}^\infty (-1)^m\dfrac{\trace\left(K_{W_0}^m\right)}{m} + \epsi \sum_{m=p}^\infty (-1)^m\trace\left(K_{W_0}^{m-1}K_\Lambda\right) + O(\epsi^2) \\ 
    = \sum_{m=p}^\infty (-1)^m\dfrac{\trace\left(K_{W_0}^m\right)}{m} - \epsi \trace\left((-K_{W_0})^{p-1}(\Id + K_{W_0})^{-1}K_\Lambda\right) + O(\epsi^2).
\end{equations*}
when $\Im \lambda \geq 1$. The following determinant asymptotic follows: for $\Im \lambda$ large enough,
\begin{equations*}
D_{W_0 +  \epsi \Lambda}(\lambda)  = \exp\left(-\sum_{m=p}^\infty (-1)^m\dfrac{\trace\left(K_{W_0}^m\right)}{m} + \epsi \trace\left((-K_{W_0})^{p-1}(\Id + K_{W_0})^{-1}K_\Lambda\right) + O(\epsi^2)\right) \\
     = D_{W_0}(\lambda)\left(1 + \epsi \trace\left((-K_{W_0})^{p-1}(\Id + K_{W_0})^{-1}K_\Lambda\right)\right) + O(\epsi^2).
\end{equations*}
Thus $b_0(\lambda) = D_{W_0}(\lambda)$ and $b_1(\lambda) = D_{W_0}(\lambda) \trace\left((-K_{W_0})^{p-1}(\Id + K_{W_0})^{-1}K_\Lambda\right)$ for $\Im \lambda \gg 1$. Since the functions $b_0, b_1$ are holomorphic by the unique continuation principle these identities must also hold on $X_d$. This ends the proof of the theorem.\end{proof}

We are now ready to prove Theorem \ref{cor:3}. It is the special case $m=1$ of

\begin{thm}\label{cor:3a} Let $V_{\eff} = W_0 + \epsi^2 \Lambda_0 + \epsi^3 \Lambda_1$ were $\Lambda_0, \Lambda_1$ where defined in Theorem \ref{thm:5}. Let $\mu_\epsi$ be a family of resonances of $V_{\eff,\epsi}$ with multiplicity $m$. For every $\epsi > 0$ there exist $m$ resonances counted with multiplicity $\lambda_{1,\epsi}, ..., \lambda_{m,\epsi}$ of $V_\epsi$ such that
\begin{equation*}
|\lambda_{j,\epsi} - \mu_\epsi| = O(\epsi^{4/m}).
\end{equation*}
Conversely let $\lambda_\epsi$ be a family of resonances of $V_{\epsi}$ with multiplicity $m$. For every $\epsi > 0$ there exist $m$ resonances counted with multiplicity $\mu_{1,\epsi}, ..., \mu_{m,\epsi}$ of $V_{\eff,\epsi}$ such that
\begin{equation*}
|\lambda_{j,\epsi} - \mu_{j,\epsi}| = O(\epsi^{4/m}).
\end{equation*}
\end{thm}

\begin{proof} Assume $d \geq 3$. Fix $N = 4$, $p = 16(d+4)$ and $D_V$ given in Theorem \ref{thm:5}. Let $V_\eff = W_0 - \epsi^2 \Lambda_0 - \epsi^3 \Lambda_1$. By Theorem \ref{thm:5},
\begin{equation}\label{eq:6f}
D_V = D_{W_0}(\lambda) \left(1 + \trace\left( (\Id + K_{W_0})^{-1} (-K_{W_0})^{p-2} K_{-\epsi^2 \Lambda_0 - \epsi^3 \Lambda_1} \right) \right) + O(\epsi^4).
\end{equation}
Define $\DD_\VV$ the Fredholm determinant
\begin{equation*}
\DD_\VV(\lambda) = \Det(\Id + \psi(K_\VV)), \ \ \ \psi(z) = \exp\left( \dfrac{(-z)^{p-1}}{p-1} \right) \Psi(z).
\end{equation*}
The Fredholm determinants $D_\VV$ defined in \eqref{eq:1d} and $\DD_\VV$ are related through
\begin{equation*}
D_\VV(\lambda) = \exp \left(\dfrac{\trace((-K_\VV)^{p-1})}{p-1}\right) \DD_\VV(\lambda). 
\end{equation*}
Therefore \eqref{eq:6f} implies $ D_V(\lambda) = $
\begin{equation*}
\exp \left(\dfrac{\trace((-K_{W_0})^{p-1})}{p-1}\right) \DD_{W_0}(\lambda) \left(1 + \trace\left( (\Id + K_{W_0})^{-1} (-K_{W_0})^{p-2} K_{-\epsi^2 \Lambda_0 - \epsi^3 \Lambda_1} \right) \right) + O(\epsi^4).
\end{equation*}
Lemma \ref{lem:7} leads to
\begin{equation*}
D_V(\lambda) = \exp \left(\dfrac{\trace((-K_{W_0})^{p-1})}{p-1}\right) \DD_{V_\eff}(\lambda) + O(\epsi^4)
\end{equation*}
where $V_\eff = W_0 - \epsi^2 \Lambda_0 - \epsi^3 \Lambda_1$. Consider now $\mu_\epsi$ a bounded family of resonances of $V_\eff$ of multiplicity $m$. As $\mu_\epsi$ is bounded there exist $C,r$ such that for every $\lambda \in \Dd(\mu_\epsi,r)$, 
\begin{equation}\label{eq:7g}
\left|\exp \left(\dfrac{\trace((-K_{W_0})^{p-1})}{p-1}\right) \DD_{V_\eff}(\lambda)\right| \geq C |\lambda-\mu_\epsi|^m.
\end{equation}
Let $\gamma_\epsi = \p \Dd(\mu_\epsi,c\epsi^{4/m})$. If $c$ is small enough then by \eqref{eq:7g} for every $\lambda \in \gamma_\epsi$,
\begin{equation*}
\left|D_V(\lambda) - \exp \left(\dfrac{\trace((-K_{W_0})^{p-1})}{p-1}\right) \DD_{V_\eff}(\lambda) \right| < \left|\exp \left(\dfrac{\trace((-K_{W_0})^{p-1})}{p-1}\right) \DD_{V_\eff}(\lambda)\right|.
\end{equation*}
By Rouch\'e's theorem this implies that $V$ and $V_\eff$ have the same number of resonances inside the disk $\Dd(\mu_\epsi, c\epsi^{4/m})$. The proof of the convert part is similar and we omit it. This proves Theorem \ref{cor:3a} away from the resonance $0$ in dimension one.

We now concentrate on $d=1$. In this case by Theorem \ref{thm:5} and Lemma \ref{lem:7} the function $h_V$ of Lemma \ref{lem:2s} satisfies $h_V(\lambda)=\lambda d_{V_\eff}(\lambda) + O(\epsi^4)$ locally uniformly on  $X_d$. The functions $h_V$ and $d_{V_\eff}$ are both entire. By a Cauchy formula, if $\lambda \in \Dd(0,1)$ then 
\begin{align*}
h_V(\lambda) = \dfrac{1}{2\pi i} \oint_{\p \Dd(0,2)} \dfrac{\lambda d_{V_\eff}(\mu) d\mu}{\mu-\lambda} + O(\epsi^4)
\end{align*}  
and this holds uniformly on   $\Dd(0,1)$. Thus the estimate $h_V(\lambda)=\lambda d_{V_\eff}(\lambda) + O(\epsi^4)$ holds locally uniformly on   $\C$. The end of the proof is the same as in the case $d \geq 3$.\end{proof}

\subsection{Uniform description of the resonant set} Here we prove Theorem \ref{cor:3s}. Let $W \in C_0^\infty(\Bb^d(0,L) \times \Tt^d, \C)$ and $V$ associated to $W$ by \eqref{eq:1a}. Fix $B > 0$. We first localize resonances of $V$ that are above the line $\Im \lambda = -B$. According to \eqref{eq:1b} the set of resonances of $V$ in $X_d$ is the set of $\lambda$ such that the operator $\Id + K_V(\lambda)$ is not invertible on $L^2$. Thus if $\lambda \in X_d$ is a resonance then $|K_V|_\BB \geq 1$. Since for $\Im \lambda \geq -B$, $|K_V|_\BB \leq {C |V|_\infty e^{2L B}}/{|\lambda|}$, for $\epsi$ small enough resonances of $V$ and $W_0$ in the half plane $\Im \lambda \geq -B$ all belong to a same disk $\Dd(0,\rho)$. By Theorem \ref{thm:5},
\begin{equation*}
D_V(\lambda) = D_{W_0}(\lambda) + O(\epsi^2) \text{ uniformly on } \Dd(0,\rho). 
\end{equation*}
As $D_{W_0}$ has no zero on $\p \Dd(0,\rho)$ we have
\begin{equation*}
\dfrac{1}{2\pi i} \oint_{\p \Dd(0,\rho)} \dfrac{D'_V(\lambda)}{D_V(\lambda)} d\lambda \rightarrow \dfrac{1}{2\pi i} \oint_{\p \Dd(0,\rho)} \dfrac{D_{W_0}'(\lambda)}{D_{W_0}(\lambda)} d\lambda.
\end{equation*}
Therefore $W_0$ and $V$ have the same (finite) number of resonances on $\Dd(0,\rho)$ for $\epsi$ small enough. By Theorem \ref{cor:3r} there exists $c > 0$ such that these resonances belong to 
\begin{equation*}
\CC_\epsi = \bigcup_{\substack{\lambda_0 \in \Res(W_0), \\ \Im \lambda_0 \geq -B}} \Dd\left(\lambda_0, c\epsi^{2/m_{W_0}(\lambda_0)}\right).
\end{equation*}

Now assume that $\lambda \in \Res(V)$ satisfies $\Im \lambda \leq - B$ and that $\lambda$ does not belong to the set $\TT_\epsi$ defined in \eqref{eq:9g}. This means
\begin{equation*}
\lambda \notin \bigcup_{\substack{\lambda_0 \in \Res(W_0), \\ \Im \lambda_0 \leq -B}} \Dd\left(\lambda_0,\lr{\lambda_0}^{-d- 1}\right).
\end{equation*}
Then (see the proof of \cite[Theorem $3.49$]{DyaZwo}):
\begin{equation*} 
\left|(\Id+K_{W_0})^{-1}\right|_{\BB} \leq e^{C \lr{\lambda}^{2d+1}}.
\end{equation*}
We now reproduce the proof of Theorem \ref{thm:1} for $d \geq 3$. Since $\lambda \in \Res(V)$ there must exist $u \in L^2$ with 
$u = -K_Vu$. In particular $u$ belongs to $H^1$ with $|u|_{H^1} \leq C e^{C (\Im \lambda)_-} |W|_\infty |u|_2$. The equation $u=-K_Vu$ is equivalent to 
\begin{equation*}
u = -\left( \Id + K_{W_0} \right)^{-1} K_{V_\sharp} u = -\left( \Id + K_{W_0} \right)^{-1} \sum_{k \neq 0} K_{W_k} \ek u
\end{equation*}
where $V_\sharp(x)=\sum_{k \neq 0} W_k(x) e^{ikx/\epsi}$. As in the proof of Theorem \ref{thm:1}, we perform an integration by parts on the term $K_{W_k} \ek u$:
\begin{equation*}
\dfrac{|k|}{\epsi}K_{W_k}\ek u = K_{W_k}P_k \ek u - K_{W_k} \ek P_k u.
\end{equation*}
This yields
\begin{equation*}
\dfrac{|k|}{\epsi}|K_{W_k} \ek u|_2 \leq C e^{2L(\Im \lambda)_-} \|W_k\|_1 |u|_2 + C e^{2L(\Im \lambda)_-} |W_k|_\infty |u|_{H^1}.
\end{equation*}
Using the \textit{a priori} bound on $|u|_{H^1}$ and summing over $k \neq 0$ we obtain
\begin{equation*}
|K_{V_\sharp} u|_2 \leq C \epsi e^{2L(\Im \lambda)_-} |W|_\infty \left( \sum_{k \neq 0} \dfrac{\|W_k\|_1}{|k|} \right) |u|_2.
\end{equation*}
It follows that
\begin{equation*}
|u|_2 = \left| (\Id + K_{W_0})^{-1} K_{V_\sharp}u \right|_2 \leq C \epsi e^{2L(\Im \lambda)_-} e^{C \lr{\lambda}^{2d+1}}  |W|_\infty \left( \sum_{k \neq 0} \dfrac{\|W_k\|_1}{|k|} \right)  |u|_2.
\end{equation*}
Since $u \neq 0$, this implies a lower bound on $|\lambda|$ of the form $|\lambda| \geq A- C\ln(\epsi^{-1})^{1/(2d+1)}$. Thus $\lambda$ belongs to the set $\DD_\epsi$ defined in \eqref{eq:9g}. This ends the proof of Theorem \ref{cor:3s}.

\section{Proof of Theorem \ref{thm:5}}\label{sec:5} 
We now get to the core of the paper: the proof of Theorem \ref{thm:5}. We first explain the ideas. If $D_V$ is the determinant given by \eqref{eq:1d} we can write formally 
\begin{equation*}
D_V(\lambda) = \exp\left( \sum_{m=p}^\infty \dfrac{(-1)^m}{m} \trace\left(K_V^m\right) \right).
\end{equation*}
It order to prove Theorem \ref{thm:5} it seems necessary to obtain an expansion in powers of $\epsi$ of $\trace(K_V^m)$. For a potential $V$ given by $V(x) = \sum_{k \in \Z^d} W_k(x) e^{ikx/\epsi}$ then $\trace(K_V^m)$ can be decomposed as a sum of terms of the form
\begin{equation*}
T[k_1, ..., k_m] = \trace\left( \prod_{j=1}^m K_{W_{k_j}} e^{ik_j\bullet/\epsi}\right)
\end{equation*}
where $k_1, ..., k_m \in \Z^d$. We now explain how to obtain an expansion for $T[k_1, ..., k_m]$. We say that the sequence $k_1, ..., k_m$ is constructive if $k_1+...+k_m = 0$ and destructive otherwise. We use this terminology for the following reason. In the case of a destructive sequence the behavior of the oscillatory terms $e^{ik_j x/\epsi}$ imply $\prod_{j=1}^m e^{ik_j x/\epsi} \rightarrow 0$ weakly as $\epsi \rightarrow 0$. We will prove that in this case $T[k_1, ..., k_m]$ is of order $O(\epsi^N)$ and thus produces no term in the expansion provided by Theorem \ref{thm:5}. Now if $k_1, ..., k_m$ is a constructive sequence let $R(\xi) = (\xi^2-\lambda^2)^{-1}$ so that formally $R_0(\lambda) = R(D)$. Using the commutation relation $e^{-ik\bullet/\epsi} D e^{ik\bullet/\epsi} = D+k/\epsi$, we have
\begin{equation}\label{eq:65a}\begin{gathered}
T[k_1, ..., k_m] = \trace\left( \prod_{j=1}^m \rho R(D) W_{k_j} e^{ik_j\bullet/\epsi}\right) 
       = \trace\left( \prod_{j=1}^m \rho R(D+\sigma_j/\epsi) W_{k_j} \right)
      \end{gathered}
\end{equation}
where $\sigma_j = k_j+...+k_m$. We note that there are no more oscillatory terms in the second line of \eqref{eq:65a}. An expansion of $T[k_1, ..., k_m]$ follows then from an operator-valued expansion of the operator $R(D+\sigma_j/\epsi)$, which in turn follows from an expansion of the function $R(\xi+\sigma_j/\epsi)$.

\subsection{Preliminaries on Fredholm determinants}
We start by giving a formula for general Fredholm determinants as infinite series. Consider $X,Y$ two trace class operators on $L^2$ and assume that $\Id + X$ is invertible. Define the Fredholm determinant
\begin{equation*}
D(\mu) = \Det(\Id + X + \mu Y).
\end{equation*}
This is a holomorphic function of the variable $\mu$, satisfying the bound $|D(\mu)| \leq e^{|X|_\LL + \mu |Y|_\LL}$. Expand it in power series: there exists a sequence $\omega_n(X,Y)$ such that
\begin{equation}\label{eq:2a}
D(\mu) = \sum_{n=0}^\infty \dfrac{\mu^n}{n!} \omega_n(X,Y).
\end{equation}
The terms $\omega_n(X,Y)$ are given by the $n \times n$ determinant
\begin{equation}\label{eq:2z}
\omega_n(X,Y) = \Det(\Id +X) \left| \begin{matrix}
\tau_1 & n-1 & 0 & \hdots & 0 \\
\tau_2 & \tau_1 & n-2 & \hdots & 0 \\
\vdots & \ddots & \ddots & \ddots & \vdots \\
\tau_{n-1} & \ddots & \ddots & \ddots & 1 \\
\tau_n & \tau_{n-1} & \hdots & \tau_2 & \tau_1
\end{matrix}\right|,\end{equation}
where $\tau_j = \trace\left(((\Id + X)^{-1} Y)^j\right)$ -- see \cite[Theorem 6.8]{Si77}. 

\begin{lem} Let $s \geq 0$ and assume that $\lr{D}^s X$ and $\lr{D}^s Y$, initially defined as operators from $L^2$ to $H^{-s}$, are trace class operator on $L^2$. Then
\begin{equation}\label{eq:11b}
|\omega_n(X,Y)| \leq  |\lr{D}^s Y \lr{D}^{-s}|_\LL^n e^{|\lr{D}^s X \lr{D}^{-s}|_\LL}.
\end{equation}
\end{lem}

\begin{proof} First note that since $\lr{D}^{-s} \in \BB$ and $\lr{D}^s (X + \mu Y) \in \LL$ we can use the cyclicity of the determinant to get
\begin{equation*}
\Det(\Id + X + \mu Y) = \Det(\Id + \lr{D}^s (X+\mu Y) \lr{D}^{-s}).
\end{equation*}
Therefore 
\begin{equation*}
|\Det(\Id + X + \mu Y)| \leq e^{|\lr{D}^s X \lr{D}^{-s}|_\LL+|\mu||\lr{D}^s Y \lr{D}^{-s}|_\LL}.
\end{equation*}
This proves that $\Det(\Id + X + \mu Y)$ is an entire function of order $1$. Therefore by Cauchy estimates the coefficients $\omega_n(X,Y)$ must satisfy \eqref{eq:11b}. This completes the proof. \end{proof}

\subsection{Reduction to a trace expansion}

We now  start the proof of Theorem \ref{thm:5}. Fix without loss of generality $N \geq d+1$ and $p=4(d+N)N$. Let $W$ in $C_0^\infty(\Bb^d(0,L)\times \Tt^d, \C)$, $V, V_\sharp \in C_0^\infty(\R^d,\C)$ be given by
\begin{equation*}
W(x,y) = \sum_{k \in \Z^d} W_k(x) e^{iky}, \ \ \ 
V(x) = W_0(x) + V_\sharp(x), \ \ \ V_\sharp(x) = \sum_{k \neq 0} W_k(x) e^{ikx/\epsi}.
\end{equation*}
We define $|W|_{Z^s} = \sum_{k \in \Z^d} \| W_k \|_s$. This quantity is finite for every $s \geq 0$.

Let $X$ and $Y$ be the trace class operators given by
\begin{gather}\label{eq:2f}\begin{split}
X = \Psi(K_{W_0}), \ \  \ \ Y = \Psi(K_V) - \Psi(K_{W_0}), \ \ \ \ \ \ \ \\ 
\Psi(z) = (1+z) \exp\left(-z+ \dfrac{z^2}{2}- ...  +\dfrac{(-z)^{p-1}}{p-1} \right) - 1.\end{split}
\end{gather}
The expansion \eqref{eq:2a} yields
\begin{align*}
D_V(\lambda) = \Det(\Id + X+Y) = \sum_{n=0}^\infty \dfrac{1}{n!} \omega_n(X,Y).
\end{align*}
We now reduce this exact infinite expansion to a finite expansion modulo a term of order $O(\epsi^{2N})$. We recall that $X_d = \C$ if $d \geq 3$ and $X_1 = \C \setminus \{0\}$.

\begin{lem}\label{lem:1j} Locally uniformly on $X_d$, 
\begin{equation*}
D_V(\lambda) = \sum_{n=0}^N \dfrac{1}{n!} \omega_n(X,Y) + O(\epsi^{2N}).
\end{equation*}
\end{lem}

\begin{proof} It is enough to show that the coefficients $\omega_n(X,Y)$ satisfy the inequality
\begin{equation}\label{eq:4d}
|\omega_n(X,Y)| \leq (C \epsi^2)^n
\end{equation}
for all $n \geq 0$. Because of \eqref{eq:11b} it suffices then to estimate $|Y|_\LL$. Recall that the first $p-1$ derivatives of $\Psi$ vanish at $0$ and write a power series expansion of $\Psi$ as
\begin{equation*}
\Psi(z) = \sum_{m=p}^\infty \az_m z^m, \ \ \ \az_m = \dfrac{1}{m!} \dfrac{d^m\Psi}{dz^m}(0).
\end{equation*}
Since the function $\Psi$ is entire of order $p-1$ and type $(p-1)^{-1}$ the coefficients $\az_m$ satisfy the estimate
\begin{equation}\label{eq:2e}
 |\az_m| \leq C \left(m^{-m} e^{m}\right)^{1/(p-1)} \leq C (m^{1/2}/m!)^{1/(p-1)}.
\end{equation}
-- see for instance \cite{T}. Next write
\begin{align*}
Y = \Psi(K_V) - \Psi(K_{W_0}) & = \int_0^1 \dfrac{d}{dt} \Psi(K_{W_0+tV_\sharp}) dt  = \int_0^1 \sum_{m=p}^\infty \az_m \sum_{\ell=0}^{m-1} K_{W_0+tV_\sharp}^\ell K_{V_\sharp} K_{W_0+tV_\sharp}^{m-\ell-1} dt.
\end{align*}
This yields $\lr{D}^2 \left(\Psi(K_V) - \Psi(K_{W_0})\right)\lr{D}^{-2}=$
\begin{equation*}
\int_0^1 \sum_{m=p}^\infty \az_m \sum_{\ell=0}^{m-1} \left(\lr{D}^2 K_{W_0+tV_\sharp}\lr{D}^{-2}\right)^{\ell-1} \lr{D}^2 K_{W_0+tV_\sharp}  K_{V_\sharp} \left(\lr{D}^2K_{W_0+tV_\sharp}\lr{D}^{-2}\right)^{m-\ell-1} dt.
\end{equation*} 
The singular values of $\lr{D}^2 K_{W_0+tV_\sharp}\lr{D}^{-2}$ are bounded as follows:
\begin{equation*}\begin{gathered}
s_j\left(\lr{D}^2 K_{W_0+tV_\sharp}\lr{D}^{-2}\right)  \leq \left|\lr{D}^2 K_{W_0+tV_\sharp}\right|_\BB s_j\left(\rho\lr{D}^{-2}\right)
 \leq C |W|_\infty s_j\left(\rho\lr{D}^{-2}\right).\end{gathered}
\end{equation*}
To estimate $s_j\left(\rho\lr{D}^{-2}\right)$ we note that as the singular values of an operator $X$ are the squareroot of the eigenvalues of $XX^*$,
\begin{equation}\label{eq:2i}
s_j\left(\rho \lr{D}^{-2}\right) =  \lambda_j\left(\rho \lr{D}^{-4} \rho\right)^{1/2}  \leq s_j\left(\rho \lr{D}^{-4} \rho\right)^{1/2} \leq C j^{-2/d}.
\end{equation}
In the last line we used \cite[(B.3.9)]{DyaZwo}. It follows that $s_j\left(\lr{D}^2 K_{W_0+tV_\sharp}\lr{D}^{-2}\right)$  $\leq C |W|_\infty j^{-2/d}$. In addition using the commutation relation
\begin{equation*}
\ek = \dfrac{\epsi}{|k|}[P_k,\ek], \ \ \ P_k = \dfrac{k_1 D_{x_1} + ... + k_d D_{x_d}}{|k|}
\end{equation*}
we obtain
\begin{equations}\label{eq:2h}
|K_{V_\sharp}\lr{D}^{-2}|_\BB   \leq |K_\rho \lr{D}^2|_\BB |\lr{D}^{-2} V_\sharp \lr{D}^{-2}|_\BB \\  
 \leq \sum_{k \neq 0} |K_\rho \lr{D}^2|_\BB |\lr{D}^{-2} W_k \ek \lr{D}^{-2}|_\BB \\ 
     \leq \sum_{k \neq 0} \dfrac{\epsi^2}{|k|^2} | K_\rho \lr{D}^2|_\BB |\lr{D}^{-2} W_k [P_k,[P_k,\ek]] \lr{D}^{-2}|_\BB  \leq C \epsi^2 |W|_{Z^2}.
\end{equations}
Consequently, 
\begin{equation*}\begin{gathered}
 s_{(m-2)j}\left( \left(\lr{D}^2 K_{W_0+tV_\sharp}\lr{D}^{-2}\right)^{\ell-1} \lr{D}^2 K_{W_0+tV_\sharp}  K_{V_\sharp} \lr{D}^{-2} \left(\lr{D}^2K_{W_0+tV_\sharp}\lr{D}^{-2}\right)^{m-\ell-1} \right) \\ 
 \leq s_j\left(\lr{D}^2 K_{W_0+tV_\sharp}\lr{D}^{-2}\right)^{m-2} |\lr{D}^2 K_{W_0+tV_\sharp}|_\BB |K_{V_\sharp} \lr{D}^{-2}|_\BB \\
 \leq   C^m \epsi^2|W|_\infty^{m-1} |W|_{Z^2}  j^{-2(m-2)/d}.\end{gathered}
\end{equation*}
Sum over $\ell \in [0,m-1], j \geq 0$ and note that $m \geq p \geq d+2$ to obtain the bound
\begin{equation*}
\left|\sum_{\ell=0}^{m-1} \lr{D}^2 K_{W_0+tV_\sharp}^\ell K_{V_\sharp} K_{W_0+tV_\sharp}^{m-1-\ell} \lr{D}^{-2}\right|_\LL \leq m^2 C^m\epsi^2 |W|_\infty^{m-1} |W|_{Z^2}.
\end{equation*}
This yields
\begin{equation*}
\left|\lr{D}^2\left(\Psi(K_V) - \Psi(K_{W_0})\right)\lr{D}^{-2}\right|_\LL \leq \sum_{m=p}^\infty m^2 |\az_m|  C^m\epsi^2 |W|_\infty^{m-1} |W|_{Z^2}  \leq C \epsi^2,
\end{equation*}
where the series indeed converges because of the decay of the coefficients $\az_m$ proved in \eqref{eq:2e}. This ends the proof of the lemma. \end{proof}

We now show that Theorem \ref{thm:5} can be reduced to the following key result:

\begin{lem}\label{lem:2} Let $X, Y$ be given by \eqref{eq:2f} and $\TT_X$ be the holomorphic continuation of the operator $\Det(\Id + X) (\Id + X)^{-1}$ given in Appendix \ref{app:1}. There exist $N$ functions $c_0, c_1, ... c_{N-1}$ holomorphic on $X_d$ such that for all $1 \leq a \leq N$,
\begin{equation*}
\trace\left( (\TT_X Y)^a \right) = c_0(\lambda) + \epsi c_1(\lambda) + ... + \epsi^{N-1} c_{N-1}(\lambda) + O(\epsi^N).
\end{equation*} 
This holds uniformly locally on $X_d$.
\end{lem}

Assuming that this lemma holds Theorem \ref{thm:5} is only a consequence of a complex analysis argument resumed in

\begin{lem}\label{lem:1i} Let $E = \C$ or $\C \setminus \{0\}$, $S_0$ be a discrete subset of $E$. Let $(\lambda,\epsi) \rightarrow f(\lambda,\epsi), g(\lambda,\epsi)$ two functions such that $f(\cdot,\epsi) ,g(\cdot, \epsi)$ are meromorphic with poles in $S_0$ and such that $h(\cdot,\epsi)=f(\cdot,\epsi)g(\cdot,\epsi)$ is holomorphic on $E$. Assume moreover that locally uniformly on $E \setminus S_0$ we have
\begin{equation}\label{eq:4b}\begin{gathered}
f(\lambda,\epsi) = f_0(\lambda) + \epsi f_1(\lambda) + ... + \epsi^{N-1}f_{N-1}(\lambda) + O(\epsi^N) \\
g(\lambda,\epsi) = g_0(\lambda) + \epsi g_1(\lambda) + ... + \epsi^{N-1}g_{N-1}(\lambda) + O(\epsi^N)
\end{gathered}
\end{equation} 
where $f_0, g_0, ..., f_{N-1}, g_{N-1}$ are meromorphic functions of $\lambda \in \C$. Then there exist holomorphic functions  $h_0, ..., h_{N-1}$ on $E$ such that uniformly locally on $E$,
\begin{equation}\label{eq:4c}
h(\lambda,\epsi) = h_0(\lambda) + \epsi h_1(\lambda) + ... + \epsi^{N-1}h_{N-1}(\lambda) + O(\epsi^N).
\end{equation}
\end{lem}

\begin{proof} First note that \eqref{eq:4b} and the fact that $h=fg$ imply that the expansion \eqref{eq:4c} holds for $\lambda$ away from $S_0$. It remains to show that the functions $h_j$ are holomorphic on $E$ and that the expansion holds locally uniformly on $E$. We first note that locally uniformly on $E \setminus S_0$,
\begin{equation*}
f_j(\lambda) = \lim_{\epsi \rightarrow 0} \dfrac{f(\lambda) - f_0(\lambda) - ... - \epsi^{j-1} f_{j-1}(\lambda)}{\epsi^j}
\end{equation*}
where by convention $f_{-1} = 0$. A uniform limit of holomorphic functions is holomorphic; thus by an immediate recursion $f_0$, ..., $f_{N-1}$ must be holomorphic on $E$. The poles of the $f_n$ are then a subset of the poles of $f$ and thus they all belong to $S_0$. The same holds for the poles of $g_n$. Consequently the poles of the $h_n$ belong to $S_0$. Let $n$ minimal so that $h_n$ has a singularity at a point $\lambda_0 \in S_0$. For $r$ small enough $\lambda_0$ is the unique singularity of $h_n$ on $\Dd(\lambda_0,2r)$. For every $\epsi > 0$, the function
\begin{equation*}
H_n(\cdot,\epsi) = \dfrac{h(\cdot,\epsi) - \epsi h_1 - ... - \epsi^{n-1}h_{n-1}}{\epsi^n}
\end{equation*}
is holomorphic on $\Dd(\lambda_0,2r)$. As $\epsi \rightarrow 0$, $H_n(\lambda,\epsi) = O(1)$ and $H_n(\lambda,\epsi) \rightarrow h_n(\lambda)$, both holding uniformly locally in $\Dd(\lambda_0,2r) \setminus \{\lambda_0\}$. By the maximum principle there exists $M > 0$ such that for every $\lambda \in \Dd(\lambda_0,r) \setminus \{\lambda_0\}$,
\begin{equation*}
|h_n(\lambda)| = \lim_{\epsi \rightarrow 0} |H_n(\lambda,\epsi)| \leq \limsup_{\epsi \rightarrow 0} \sup_{\mu \in \p \Dd(\lambda_0,r)} |H_n(\mu,\epsi)| \leq M.
\end{equation*}
Therefore $h_n$ is uniformly bounded in a neighborhood of $\lambda_0$ and its singularity is removable. It follows that all the $h_j$ are holomorphic on $E$. Now to prove that \eqref{eq:4c} holds uniformly locally on $E$ we recall that it already holds uniformly locally on $E \setminus S_0$. Now if $\lambda_0 \in S_0$ and $r > 0$ is such that $\overline{\Dd(\lambda_0,r)} \subset E$ and $\p \Dd(\lambda_0,r) \subset E \setminus S_0$ then Cauchy's formula shows
\begin{equations*}
h(\lambda) = \dfrac{1}{2\pi i} \oint_{\p \Dd(\lambda,r)} \dfrac{h(\mu)}{\mu-\lambda} d\mu = \dfrac{1}{2\pi i} \oint_{\p \Dd(\lambda,r)} \dfrac{h_0(\mu) + ... \epsi^{N-1}h_{N-1}(\mu) + O(\epsi^N)}{\mu-\lambda} d\mu \\
 = h_0(\lambda) + ... \epsi^{N-1}h_{N-1}(\lambda) + O(\epsi^N)
\end{equations*} 
with convergence realized uniformly in $\Dd(\lambda,r)$. This ends the proof.\end{proof}

\begin{proof}[Proof of Theorem \ref{thm:5} assuming Lemma \ref{lem:2}.] By Lemma \ref{lem:1j} it suffices to prove that for every $n \in [0,N]$, $\omega_n(X,Y)$ admits an expansion in powers of $\epsi$ at order $N$. By \eqref{eq:2z}, $\omega_n(X,\Det(\Id +X)Y)$ is a finite combination of terms of the form $\trace((\TT_X Y)^\alpha)$, $1 \leq a \leq N$. Thus by Lemma \ref{lem:2}, $\omega_n(X, \Det(\Id +X)Y)$ has an expansion of the form
\begin{equation}\label{eq:11j}
\omega_n(X,\Det(\Id +X)Y) = f_0(\lambda) + \epsi f_1(\lambda) + ... + \epsi^{N-1} f_{N-1}(\lambda) + O(\epsi^N).
\end{equation}
Here the convergence holds locally uniformly on $X_d$. In addition,
\begin{equation*}
\omega_n(X,Y) = \dfrac{1}{\det(\Id + X)^n}\omega_n(X,\Det(\Id +X) Y).
\end{equation*}
Now apply Lemma \ref{lem:1i} to $E=X_d$, $S_0 = \Res(W_0)$, $f=\det(\Id +X)^{-n}$ and $g=\omega_n(X,\Det(\Id+X) Y)$. The meromorphic function $f$ does not depend on $\epsi$ and its poles in $E$ are  exactly the resonances of $W_0$. The function $g$ is holomorphic on $E$, depends on $\epsi$ and admits an expansion given by \eqref{eq:11j}. The product $h=fg$ is then meromorphic; by \eqref{eq:4d} it is locally uniformly bounded on $E$ and consequently it is holomorphic on $E$. Thus $\omega_n(X,Y)$ admits an expansion in powers of $\epsi$ at order $N$ and Theorem \ref{thm:5} follows. We will compute the first few terms in \S \ref{subsec:3} below.\end{proof}

The next sections are devoted to the proof of Lemma \ref{lem:2}. We first simplify the expression $\trace\left(\TT_X Y)^a\right)$. 

\begin{lem}\label{lem:4} For $a \in [1,N]$, $\trace((\TT_X Y)^a)$ can be written modulo $O(\epsi^N)$ as a finite sum of expressions of the form $\trace( \TT_X F_{n_1} ... \TT_X F_{n_a} )$ where $1 \leq n_j \leq 2N-1$ and
\begin{equation}\label{eq:0e}
F_n = \sum_{m=p}^\infty \az_m \sum_{\ell_0+...+\ell_n +n =m} K_{W_0}^{\ell_1} K_{V_\sharp} ... K_{W_0}^{\ell_{n-1}} K_{V_\sharp} K_{W_0}^{\ell_n}, \ \ \ \ \az_m = \dfrac{1}{m!}\dfrac{d^m \Psi}{z^m}(0).
\end{equation}
This holds uniformly locally on $X_d$.
\end{lem}

\begin{proof} Fix $1 \leq a \leq N$ and define $\KK_\VV = \lr{D} K_\VV \lr{D}^{-1}$. Using the cyclicity of the trace, 
\begin{gather*}
\trace\left( (\TT_X Y)^a \right) = \trace\left( (\TT_{X'} Y')^a \right) \\
X' = \Det(\Id + \Psi(\KK_{W_0})) (\Id + \Psi(\KK_{W_0}))^{-1}, \ \ \ Y' = \Psi(\KK_V) - \Psi(\KK_{W_0}).
\end{gather*}
Define
\begin{equation}\label{eq:2k}
\EE_{m,n} = \sum_{\ell_0+...+\ell_n +n = m} \KK_{W_0}^{\ell_0} \KK_{V_\sharp} ... \KK_{W_0}^{\ell_{n-1}} \KK_{V_\sharp} \KK_{W_0}^{\ell_{n}}, \ \ \ \ \FF_n = \sum_{m=p}^\infty \az_m \EE_{m,n}.
\end{equation}
The index $n$ has the following significance: $\EE_{m,n}$ is the sum of monomials in $\KK_{W_0}$, $\KK_{V_\sharp}$ with exactly $n$ factors equal to $\KK_{V_\sharp}$. Using the power series expansion of $\Psi$ and $Y'=\Psi(\KK_V)-\Psi(\KK_{W_0})$ we obtain
\begin{align*}
Y'  = \sum_{m = p}^\infty \az_m(\KK_{W_0} + \KK_{V_\sharp})^m - \sum_{m = p}^\infty \az_m \KK_{W_0}^m = \sum_{m = p}^\infty \az_m \left( \EE_{m,1} + ... + \EE_{m,m} \right)  = \sum_{n=1}^\infty \FF_n.
\end{align*}
We claim that
\begin{equation}\label{eq:2l}
\left|\sum_{n = 2N}^\infty \FF_n \right|_\LL = O(\epsi^N).
\end{equation}
In order to prove this start by fixing $\ell_0, ..., \ell_n$ with $\ell_0+...+\ell_n +n = m \geq p$. Since $\KK_{V_\sharp}$ appears exactly $n$ times in the product $\KK_{W_0}^{\ell_0} \KK_{V_\sharp} ... \KK_{W_0}^{\ell_{n-1}} \KK_{V_\sharp} \KK_{W_0}^{\ell_n}$ we have
\begin{equation}\label{eq:2j}
s_{nj}\left(\KK_{W_0}^{\ell_0} \KK_{V_\sharp} ... \KK_{W_0}^{\ell_{n-1}} \KK_{V_\sharp} \KK_{W_0}^{\ell_n}\right) \leq s_j(\KK_{V_\sharp})^n |\KK_{W_0}|^{m-n}_\BB.
\end{equation}
We now prove some estimates on $s_j(\KK_{V_\sharp})$. On one hand by the same argument as in \eqref{eq:2h} we have 
\begin{align*}
s_j(\KK_{V_\sharp}) \leq |\KK_{V_\sharp}|_\BB & \leq |\lr{D} K_\rho \lr{D}|_\BB \cdot |\lr{D}^{-1} V_\sharp \lr{D}^{-1}|_\BB  \leq C \epsi |W|_{Z^1}.  
\end{align*}
On the other hand by arguments similar to \eqref{eq:2i} we have
\begin{align*}
s_j(\KK_{V_\sharp}) & \leq |\lr{D} K_{V_\sharp}|_\BB \cdot s_j(\rho \lr{D}^{-1}) \leq C|W|_\infty j^{-1/d}.
\end{align*}
Interpolating both inequalities yields $s_j(\KK_{V_\sharp}) \leq C \epsi^{1/2} |W|_{Z^1} j^{-1/(2d)}$.
Coming back to \eqref{eq:2j} we obtain
\begin{equation}\label{eq:0c}
s_{mj}\left(\KK_{W_0}^{\ell_0} \KK_{V_\sharp} ... \KK_{W_0}^{\ell_{n-1}} \KK_{V_\sharp} \KK_{W_0}^{\ell_n}\right) \leq C^m |W|_{Z^1}^m \epsi^{n/2} j^{-n/(2d)}.
\end{equation}
Since $n \geq 2N \geq 2d+2$ the RHS of \eqref{eq:0c} is summable. Summation over $j$ leads
\begin{equation*}
\left|\KK_{W_0}^{\ell_0} \KK_{V_\sharp} ... \KK_{W_0}^{\ell_{n-1}} \KK_{V_\sharp} \KK_{W_0}^{\ell_n}\right|_\LL \leq m \epsi^{n/2} (C|W|_{Z^1})^m
\end{equation*}
Consequently if $\EE_{m,n}$ is given by \eqref{eq:2k} then for $n \geq 2N$ 
\begin{equation}\label{eq:0d}
|\EE_{m,n}|_\LL \leq  m \matrice{ m \\ n } \epsi^N (C|W|_{Z^1})^m.
\end{equation}
The claim \eqref{eq:2l} follows then from \eqref{eq:0d} and the estimate \eqref{eq:2e} on $\az_m$:
\begin{equations*}
\left|\sum_{n = 2N}^\infty \FF_n \right|_\LL =  \left|\sum_{m = p}^\infty \az_m \left( \EE_{m,2N} + ... + \EE_{m,m} \right)\right|_{\LL} \leq  \epsi^{N} \sum_{m=p}^\infty m |\az_m| (C|W|_{Z^1})^m \left(\matrice{ m \\ 2N } + ... + \matrice{ m \\ m } \right)  \\
\leq  \epsi^{N} \sum_{m=p}^\infty m |\az_m| (2C|W|_{Z^1})^m = O(\epsi^{N}).
\end{equations*}
It follows that we can write $Y'$ as a the sum of a finite combination of the operators $\FF_n$ with $1 \leq n \leq m$ and a small error in $\LL$:
\begin{equation*}
Y' = \sum_{n=1}^\infty \FF_n = \sum_{n=1}^{2N-1} \FF_n + O_\LL(\epsi^N).
\end{equation*}
Therefore $\trace((\TT_{X'} Y')^a)$ is modulo $O(\epsi^N)$ a finite sum of expressions of the form
\begin{equation*}
\trace( \TT_{X'} \FF_{n_1} ... \TT_{X'} \FF_{n_a} ),
\end{equation*} 
where $1 \leq n_j \leq 2N-1$. Now as $X = \lr{D}^{-1} X' \lr{D}$, $F_n = \lr{D}^{-1} \FF_n \lr{D}$ and $\trace((\TT_{X'} Y')^a) = \trace((\TT_X Y)^a$ this completes the proof of the lemma. \end{proof}

To sum up we have proved that Theorem \ref{thm:5} holds if Lemma \ref{lem:2} holds, that is if for $a \in [1,N]$, $\trace((\TT_X Y)^a)$ admits an expansion in powers of $\epsi$. In addition Lemma \ref{lem:2} holds if for all $n_j \in [1,2N-1]$, $\trace(\TT_X F_{n_1} ... \TT_X F_{n_a})$ admits an expansion in powers of $\epsi$.

We write the operator $F_n$ given in \eqref{eq:0e} in the following form:
\begin{equations}\label{eq:1p}
F_n  =  \sum_{m=p}^\infty \az_m \sum_{\{k_\ell\} \in \SSSS_m^n} \left( \prod_{\ell=1}^m K_{W_{k_\ell}} e^{ik_\ell\bullet/\epsi} \right)
\end{equations}
where $\SSSS_m^n$ is the collection of sequences $d$-tuples $(k_1, ..., k_m)$, with exactly $n$ non-vanishing terms. Because of the conclusion of Lemma \ref{lem:4} we can restrict our attention to operators $F_n$ with $n \leq 2N-1$. For $n \leq 2N-1$ and $m \geq p$ the sequences of $\SSSS_m^n$ have much more vanishing terms than non vanishing terms. This will allow us to use some arguments of combinatorial nature. The expansion of $F_n$ given by \eqref{eq:1p} leads to
\begin{align*}
\prod_{j=1}^a \TT_X F_{n_j} & = \sum_{m_1, ..., m_a=p}^\infty \ \ \  \sum_{\{k_\ell^1\} \in \SSSS_{m_1}^{n_1}, \ ..., \ \{k_\ell^a\} \in \SSSS_{m_a}^{n_a}} \ \ \  \left(\prod_{j=1}^a \az_{m_j} \TT_X \prod_{\ell=1}^{m _j}K_{W_{k_\ell^j}} e^{ik_\ell^j\bullet/\epsi}\right) \\
   & = \DD_{n_1, ..., n_a} + \CC_{n_1,..., n_a}
\end{align*}
where 
\begin{equation}\label{eq:2w}\begin{gathered}
\DD_{n_1, ..., n_a} = \sum_{m_1, ..., m_a=p}^\infty \ \ \  \sum_{\substack{\{k_\ell^1\} \in \SSSS_{m_1}^{n_1}, \ ..., \ \{k_\ell^a\} \in \SSSS_{m_a}^{n_a}, \\ k_1^1+...+k_{m_a}^a \neq 0 }} \ \ \  \left(\prod_{j=1}^a \az_{m_j} \TT_X \prod_{\ell=1}^{m _j}K_{W_{k_\ell^j}} e^{ik_\ell^j\bullet/\epsi}\right), \\
\CC_{n_1, ..., n_a} = \sum_{m_1, ..., m_a=p}^\infty \ \ \  \sum_{\substack{\{k_\ell^1\} \in \SSSS_{m_1}^{n_1}, \ ..., \ \{k_\ell^a\} \in \SSSS_{m_a}^{n_a}, \\ k_1^1+...+k_{m_a}^a = 0}} \ \ \  \left(\prod_{j=1}^a \az_{m_j} \TT_X \prod_{\ell=1}^{m _j}K_{W_{k_\ell^j}} e^{ik_\ell^j\bullet/\epsi}\right).\end{gathered}
\end{equation}
In the next subsection we study the trace of the operator $\DD_{n_1, ..., n_a}$.

\subsection{Destructive interaction} The main result of this part is the following:

\begin{lem}\label{lem:5} For $1 \leq a \leq N$ and $n_1, ..., n_a \in [0,2N-1]$ let $\DD_{n_1, ..., n_a}$ be the trace class operator given by \eqref{eq:2w}. Then locally uniformly on $X_d$,
\begin{equation*}
\trace\left(\DD_{n_1, ..., n_a}\right) = O(\epsi^N).
\end{equation*}
\end{lem}

We start with a few definitions. Let $\{k_\ell\}_{1 \leq \ell \leq \nu}$ a sequence of $d$-tuples in $\Z^d$ of length $\nu$ . We say that $\{k_\ell\}_{1 \leq \ell \leq \nu}$ is constructive if it satisfies $k_1+...+k_\nu = 0$ and destructive otherwise. A sequence of $d$-tuples $\{k_\ell\}_{1 \leq \ell \leq \nu}$ is said to be admissible if 
\begin{enumerate}
\item[$(i)$] It is destructive.
\item[$(ii)$] It starts and ends with $N$ vanishing terms.
\end{enumerate}
A sequence $\{k_\ell\}_{1 \leq \ell \leq \nu'}$ with exactly $\gamma$ non-vanishing terms is said to be good if
\begin{enumerate}
\item[$(i)$] It is admissible.
\item[$(ii)$] $\nu' \leq N+N \gamma+1$.
\end{enumerate}
A subsequence of an admissible sequence $\{k_\ell\}_{1 \leq \ell \leq \nu}$ is a good subsequence if it takes the form $\{k_\ell\}_{q+1 \leq \ell \leq q+\nu'}$ for some $q, \nu'$ and if the sequence $\{k_{\ell+q}\}_{1 \leq \ell \leq \nu'}$ is good.

A cyclic permutation of $\{k_\ell\}_{1 \leq \ell \leq \nu}$ is a sequence equal to
\begin{equation*}
(k_{L+1}, ..., k_\nu, k_1, ..., k_L)
\end{equation*}
for some $L \geq 0$. We will use below the following version of the pigeonhole principle. Let $\{k_\ell\}_{1 \leq \ell \leq \nu}$ a sequence with exactly $\gamma$ non-vanishing terms. If $\nu \geq N(\gamma+1)$, there exists a subsequence of $\{k_\ell\}_{1 \leq \ell \leq \nu}$ made of $N$ consecutive vanishing $d$-tuples. The next lemma is a combinatorial result allowing us to extract good subsequences out of admissible subsequences.

\begin{lem} Every admissible sequence $\{k_\ell\}_{1 \leq \ell \leq \nu}$ admits a good subsequence.
\end{lem}

\begin{proof} We prove this lemma by recursion on $\nu$. We start by noting that any admissible sequence with length $\nu \leq 2N+1$ has at least one non-vanishing term and therefore it is a good sequence. We now fix $\nu \geq 2N+2$ and we assume that all admissible sequences of length strictly less than $\nu$ admit a good subsequence. Let $\{k_\ell\}_{1 \leq \ell \leq \nu}$ be an admissible sequence with $\gamma$ non-vanishing terms. 

If $\nu \leq N+\gamma N+1$ then $\{k_\ell\}_{1 \leq \ell \leq \nu}$ is good. Therefore we assume that $\nu \geq N+\gamma N+2$. Consider the subsequence of minimal length of consecutive $d$-tuples starting at $k_1$, containing at least one non-zero term and ending with $N$ zeros: $(k_1,...,k_{\nu'})$. Let $\gamma'$ be the number of non-zero terms in this subsequence. Since this sequence is of minimal length the pigeonhole principle implies $\nu' \leq N+\gamma'N+1$. Hence if $k_1+...+k_{\nu'} \neq 0$ then this subsequence is good and therefore we are done.

Otherwise the sequence $\{k_\ell\}_{\nu'-N+1 \leq \ell \leq \nu}$ is admissible. Indeed it starts and ends with $N$ zeros and it is destructive since $k_1+...+k_{\nu'}=0$ and $k_1+...+k_\nu \neq 0$. Therefore we can apply the induction hypothesis: it admits a good subsequence. This completes the recursion and the proof.\end{proof}

\begin{lem}\label{lem:1} Let $\{k_\ell\}_{1 \leq \ell \leq \nu}$ be an admissible sequence. Then locally uniformly on $X_d$,
\begin{equation}\label{eq:11e}
\left| \prod_{\ell=1}^\nu K_{W_{k_\ell}} e^{ik_\ell\bullet/\epsi}\right|_\BB \leq C^{\nu^2} \left(\prod_{i=1}^\nu \|W_{k_\ell}\|_{2\nu}\right) \epsi^N.
\end{equation}
\end{lem}

This lemma is the key to prove Lemma \ref{lem:5}. We start by a preliminary result:

\begin{lem}\label{lem:1a} The operator $I_{\VV,1} (\lambda) = K_\VV(\lambda) - K_\VV(-\lambda)$ is a smoothing operator. In addition there exists a constant $C$ such that uniformly in $\lambda \in \C \setminus \Dd(0,1)$,
\begin{equation}\label{eq:3d}
\left|(D^2-\lambda^2)^N I_{\VV,1}(\lambda) \right|_\BB \leq C \lr{\lambda}^{2N+d} e^{2L|\lambda|}|\VV|_\infty.
\end{equation}
\end{lem}

\begin{proof} The operator $I_{\VV,1}(\lambda)$ is smoothing as the kernel of the operator $R_0(\lambda)-R_0(-\lambda)$ is given by the smooth function
\begin{equation*}
I_{\VV,1}(\lambda,x,y) = \dfrac{i}{2} \dfrac{\lambda^{d-2}}{(2\pi)^{d-1}} \int_{\Ss^{d-1}} e^{i\lambda \lr{w,x-y}} d\w,
\end{equation*}
see \cite[Theorem 3.4]{DyaZwo}. In order to prove the estimate \eqref{eq:3d} we note that by the product rule for derivatives the operator $(D^2-\lambda^2)^N I_{\VV,1}(\lambda)$ is a finite sum of operators of the form
\begin{equation}\label{eq:3e}
\dfrac{i}{2} \dfrac{\lambda^{d-2+2t}}{(2\pi)^{d-1}}  \chi D^{2\tau} (R_0(\lambda) - R_0(-\lambda)) \VV
\end{equation}
where $t \in [0,N]$, $t+\tau \leq N$ and $\chi \in \{ \rho^{(\az)}, \ |\az| \leq N\}$. The operators of the form \eqref{eq:3e} have kernel given by
\begin{equation*}
(x,y) \mapsto \dfrac{i}{2}\dfrac{\lambda^{d-2+2N-2t}}{(2\pi)^{d-1}} \chi(x) \left(D_x^{2\tau} \int_{\Ss^{d-1}} e^{i\lambda \lr{w,\cdot-y}} d\w\right)(x) \VV(y).
\end{equation*}
Since $D_x^{2\tau} e^{i\lambda \lr{w,\cdot-y}}(x) = \lambda^{2\tau} e^{i\lambda \lr{\w, x-y}}$ we have the estimate
\begin{equation}\label{eq:3f}
\left|\left(D_x^{2\tau} \int_{\Ss^{d-1}} e^{i\lambda \lr{w,\cdot-y}} d\w\right)(x)\right| \leq C \lr{\lambda}^{2\tau} e^{|\lambda| |x-y|}
\end{equation}
uniformly on $\C \setminus \Dd(0,1)$. Since $\chi$ and $\VV$ are compactly supported the $\BB$-norm of operators of the form \eqref{eq:3e} can be estimated by Schur's lemma and the bound \eqref{eq:3f}. Recalling that $t+\tau \leq N$ it leads to
\begin{equation*}
\left|\dfrac{i}{2}\dfrac{\lambda^{d-2+2t}}{(2\pi)^{d-1}} \chi D^{2t} (R_0(\lambda) - R_0(-\lambda)) \VV\right|_\BB \leq C |\chi|_\infty \lr{\lambda}^{2N+d} e^{2L|\lambda|}|\VV|_\infty. 
\end{equation*}
To conclude it suffices to recall that the operator $(D^2-\lambda^2)^N I_{\VV,1}(\lambda)$ is a finite sum of operators of the form \eqref{eq:3e}. This completes the proof of \eqref{eq:3d}. \end{proof}

\begin{proof}[Proof of Lemma \ref{lem:1}.] We divide the proof in three main steps.

1. Fix $M > 1$. We first show that
\begin{equation}\label{eq:11v}
\left| \prod_{\ell=1}^\nu K_{W_{k_\ell}} e^{ik_\ell\bullet/\epsi}\right|_\BB \leq C^{\nu^2} \lr{\lambda}^\nu \left(\prod_{\ell=1}^\nu \|W_{k_\ell}\|_{2\nu}\right) \epsi^N, \ \ \ \Im \lambda \in [1,M].
\end{equation}
uniformly on the set $\{ \lambda : \ \Im \lambda \in [1,M]\}$. Let $R(\xi,\lambda) = (\xi^2 - \lambda^2)^{-1}$ and
\begin{equation*}
A(k,\lambda) = R(D+k/\epsi,\lambda) = e^{-ik\bullet/\epsi} R_0(\lambda) e^{ik\bullet/\epsi}.
\end{equation*}
Define $\sigma_\ell = k_\ell+...+k_\nu$. The commutation relation $e^{-ik\bullet/\epsi} D e^{ik\bullet/\epsi} = D+k/\epsi$ shows
\begin{equations*}
 K_{W_{k_1}} e^{ik_1\bullet/\epsi}...K_{W_{k_\nu}} e^{ik_\nu\bullet/\epsi} 
=  \rho A(0,\lambda) W_{k_1} e^{ik_1\bullet/\epsi} A(0,\lambda) W_{k_2} e^{ik_2\bullet/\epsi} ... A(0,\lambda) W_{k_\nu} e^{ik_\nu \bullet/\epsi} \\
=  e^{i\sigma_1 \bullet/\epsi}\rho A(\sigma_1,\lambda) W_{k_1} A(\sigma_2,\lambda) W_{k_2}  ...  A(\sigma_\nu,\lambda) W_{k_\nu} .
\end{equations*}
Now define $T_j(\lambda) = A(\sigma_j,\lambda) ... A(\sigma_\nu,\lambda)$. Since we are working in the half plane $\{\Im \lambda \geq 1\}$ the operator $T_j(\lambda)$ is well defined and bounded from $H^{-2(\nu-j)}$ to $L^2$. It admits a bounded inverse $T_j(\lambda)^{-1}$ from $L^2$ to $H^{-2(\nu-j)}$. Thus $A(\sigma_{j-1},\lambda) = T_{j-1}(\lambda) T_j(\lambda)^{-1}$ as an operator on $L^2$. This yields
\begin{equation}\label{eq:2x}\begin{gathered} 
  K_{W_{k_1}} e^{ik_1\bullet/\epsi}...K_{W_{k_\nu}} e^{ik_\nu\bullet/\epsi} \\
=   e^{i\sigma_1 \bullet/\epsi} \rho T_1(\lambda) \left(T_2(\lambda)^{-1} W_{k_1} T_2(\lambda)\right) ... \left( T_\nu(\lambda)^{-1} W_{k_{\nu-1}} T_\nu(\lambda)\right) W_{k_\nu}.\end{gathered}
\end{equation}
The estimate \eqref{eq:11e} for $\Im \lambda \in [1,M]$ follows then from a bound on $|T_j(\lambda)^{-1} W_{k_j} T_j(\lambda)|_\BB$ and a bound on $|T_1(\lambda)|_\BB$. We start with the bound on $|T_1(\lambda)|_\BB$. Since this operator is a Fourier multiplier we have
\begin{align*}
|T_1(\lambda)|_\BB = \sup_{\xi \in \R^d} \left|\prod_{j=1}^{\nu} R(\xi+\sigma_j/\epsi,\lambda)\right|.
\end{align*}
We reduce this estimate for $\Im \lambda \in [1,M]$ to an estimate for $\lambda = i$. For  $\xi \in \R^d$ and $\Im \lambda \in [1,M]$ we have $\left|(\xi^2+1)/(\xi^2-\lambda^2)\right| \leq C\lr{\lambda}$. It implies
\begin{equations*}
\sup_{\xi \in \R^d} \left|\prod_{j=1}^{\nu} R(\xi+\sigma_j/\epsi,\lambda)\right| = \sup_{\xi \in \R^d} \prod_{j=1}^{\nu} \left|R(\xi+\sigma_j/\epsi,i)\right| \cdot \left|\dfrac{(\xi+\sigma_j/\epsi)^2+1}{(\xi+\sigma_j/\epsi)^2-\lambda^2}\right| \\
  \leq  (C\lr{\lambda})^\nu \sup_{\xi \in \R^d} \left|\prod_{j=1}^{\nu} \lr{\xi+\sigma_j/\epsi}^{-2}\right|.
\end{equations*}
Since the sequence $\{k_\ell\}_{1 \leq \ell \leq \nu}$ is admissible we have $\sigma_1=...=\sigma_N \neq 0$ and $\sigma_{\nu-N+1} = ... = \sigma_{\nu} = 0$. Thus the sequence $\{\sigma_\ell\}_{1 \leq \ell \leq \nu-1}$ starts with $N$ equal non-vanishing terms and ends with $N$ vanishing terms. Peetre's inequality (see equation \eqref{eq:7f}) implies
\begin{align*} 
\sup_{\xi \in \R^d} \left|\prod_{j=1}^{\nu} \lr{\xi+\sigma_j/\epsi}^{-2}\right| & \leq |\lr{\xi+\sigma_1/\epsi}^{-2N}  \lr{\xi}^{-2N}|_\BB  \leq C^\nu \epsi^{2N}.
\end{align*}
It follows that for $\lambda \in [1,M]$, $|T_1|_\BB \leq C^\nu\lr{\lambda}^\nu \epsi^N$.

We next estimate on $|T_j^{-1} W_{k} T_j|_\BB$. We show that for every $j \in [2,\nu]$,
\begin{equation}\label{eq:11d}
|T_j^{-1} W_{k} T_j|_\BB \leq C^{\nu-j} \|W_{k_j}\|_{2(\nu-j)}
\end{equation}
using a descendent recursion on $j$. If $j=\nu$ then $T_j = A(\sigma)$ for some $\sigma \in \Z^d$. Thus 
\begin{align*}
A(\sigma)^{-1} W_k A(\sigma) & = W_k + [W_k, (D+\sigma/\epsi)^2-\lambda^2] A(\sigma) \\
   & = W_k + (D^2W_k) A(\sigma) + 2(D W_k)\cdot (D+\sigma/\epsi) A(\sigma).
\end{align*}
The operator $A(\sigma) = e^{-i\sigma\bullet/\epsi} R_0(\lambda) e^{i\sigma\bullet/\epsi}$ is bounded on $L^2$ with uniform bound when $\Im \lambda \geq 1$. The operator $(D+\sigma/\epsi)A(\sigma) = e^{-i\sigma\bullet/\epsi} DR_0(\lambda) e^{i\sigma/\epsi}$ is also bounded on $L^2$ with uniform bound when $\Im \lambda \geq 1$ as $DR_0(\lambda)$ is bounded on $L^2$ with uniform bound. Therefore:
\begin{equation*}
|A(\sigma)^{-1} W_k A(\sigma)|_\BB \leq C\|W_k\|_2.
\end{equation*}
This proves the case $j=\nu$ of \eqref{eq:11d}. Now assume that \eqref{eq:11d} holds for some $j \in [3,\nu]$ and let us prove that it also holds for $j-1$. Write $T_{j-1} = A(\sigma)T_j$ for some $\sigma$ so that
\begin{align*}
T_{j-1}^{-1} W_k T_{j-1} & = T_j^{-1}A(\sigma)^{-1} W_k A(\sigma)T_j \\
      & = T_j^{-1} \left( W_k + (D^2W_k) A(\sigma) + 2(D W_k) \cdot (D+\sigma/\epsi) A(\sigma) \right) T_j \\
     & = \left(T_j^{-1}W_k T_j\right) + \left(T_j^{-1} (D^2W_k) T_j\right) A(\sigma) + 2\left(T_j^{-1} (D W_k) T_j\right) \cdot (D+\sigma/\epsi)A(\sigma).
\end{align*}
Therefore the bounds follows from the recursion hypothesis applied to the operators $T_j^{-1}W_k T_j$, $T_j^{-1} (D^2W_k) T_j$ and $T_j^{-1} (D W_k) T_j$:
\begin{align*}
\left|T_{j-1}^{-1} W_k T_{j-1}\right|_\BB & \leq C^{\nu-j+1} \|W_k\|_{2(\nu-j)} + C^{\nu-j+1} \|DW_k\|_{2(\nu-j)} + C^{\nu-j+1} \|D^2W_k\|_{2(\nu-j)} \\
    & \leq C^{\nu-j+1} \|W_k\|_{2(\nu-j+1)}.
\end{align*}
This ends the recursion and thus the proof of \eqref{eq:11d}. The estimate \eqref{eq:11v} follows then from the identity \eqref{eq:2x}, and the bounds on $|T_j^{-1} W_{k} T_j|_\BB$, $|T_1|_\BB$. 

2. We show that an estimate similar to \eqref{eq:11v} holds for $\Im \lambda \in [-M,-1]$. Write $K_\VV(\lambda) = I_{\VV,0}(\lambda)+I_{\VV,1}(\lambda)$ where $I_{\VV,0}(\lambda) = K_\VV(-\lambda)$ and $I_{\VV,1}(\lambda)$ was defined in Lemma \ref{lem:1a}. This yields
\begin{equation*}
\prod_{\ell=1}^\nu K_{W_{k_\ell}} e^{ik_\ell\bullet/\epsi} = \sum_{\epsilon_1, ..., \epsilon_\nu \in \{0,1\}^\nu} \prod_{\ell=1}^\nu I_{W_{k_\ell},\epsilon_\ell}(\lambda)  e^{ik_\ell\bullet/\epsi}.
\end{equation*}
Fix a sequence $\epsilon_1, ..., \epsilon_\nu \in \{0,1\}^\nu$. If all the terms vanish, then 
\begin{equation*}
\prod_{\ell=1}^\nu I_{W_{k_\ell},\epsilon_\ell}(\lambda)  e^{ik_\ell\bullet/\epsi} = \prod_{\ell=1}^\nu K_{W_{k_\ell}}(-\lambda) e^{ik_\ell\bullet/\epsi}.
\end{equation*}
As $\Im(-\lambda) \in [1,M]$ we can bound the norm of this operator by directly applying \eqref{eq:11v}. Now assume that at least one of the $\epsilon_\ell$ is equal to $1$. The indexes $\ell_1, ..., \ell_s$ with $\epsilon_{\ell_1} = ... = \epsilon_{\ell_s} = 1$ split the sequence $k_1, ..., k_\nu$ in $s+1$ subsequences, of the form
\begin{equation}\label{eq:11y}
(k_1, ..., k_{\ell_1-1}), \ (k_{\ell_1}, ..., k_{\ell_2-1}), \ ..., (k_{\ell_s}, ..., k_\nu).
\end{equation}
At least one of these subsequences is destructive. Let us assume that it is the first one. Then $(k_1, ..., k_{\ell_1-1})$ is destructive and starts with $N$ zeros. It does not necessarily end with $N$ zeros. Write\begin{equation}\label{eq:3c}\begin{gathered}
  \left|\prod_{\ell=1}^{\nu} I_{W_{k_\ell},\epsilon_\ell}(\lambda)  e^{ik_\ell\bullet/\epsi}\right|_\BB \\ = \left|\left(\prod_{\ell=1}^{\ell_1-1} K_{W_{k_\ell}}(-\lambda)  e^{ik_\ell\bullet/\epsi}\right) I_{W_{k_{\ell_1}},1}(\lambda)  e^{ik_{\ell_1}\bullet/\epsi}\right|_\BB \left|\prod_{\ell=\ell_1+1}^{\nu} I_{W_{k_\ell},\epsilon_\ell}(\lambda)  e^{ik_\ell\bullet/\epsi}\right|_\BB.\end{gathered}
\end{equation}
The second factor of the RHS of \eqref{eq:3c} can be controlled by the estimates of Lemma \ref{lem:1h}:
\begin{equation*}
\left|\prod_{\ell=\ell_1+1}^{\nu} I_{W_{k_\ell},\epsilon_\ell}(\lambda)  e^{ik_\ell\bullet/\epsi}\right|_\BB \leq  \prod_{\ell=\ell_1+1}^{\nu} C_M |W_{k_\ell}|_\infty
\end{equation*}
for a constant $C_M$ depending on $M$. We deal next with the first factor in the RHS of \eqref{eq:3c}. Let $\chi \in \C_0^\infty(\Bb^d(0,L))$ be equal to $1$ on $\supp(\rho)$ and define $\tK_\rho(\lambda) = \chi R_0(\lambda) \chi$. Since $\Im \lambda \leq -1$, $\tK_\rho(-\lambda)^N (D^2-\lambda^2)^N \rho = \Id$. It follows that
\begin{equation}\label{eq:3g}\begin{gathered} 
\left|\left(\prod_{\ell=1}^{\ell_1-1} K_{W_{k_\ell}}(-\lambda)  e^{ik_\ell\bullet/\epsi}\right) I_{W_{k_{\ell_1}},1}(\lambda)  e^{ik_{\ell_1}\bullet/\epsi}\right|_\BB \\
=  \left|\left(\prod_{\ell=1}^{\ell_1-1} K_{W_{k_\ell}}(-\lambda) e^{ik_\ell\bullet/\epsi}\right) \tK_\rho(-\lambda)^N (D^2-\lambda^2)^N I_{W_{k_{\ell_1}},1}(\lambda) \right|_\BB  \\
   \leq \left|\left(\prod_{\ell=1}^{\ell_1-1} K_{W_{k_\ell}}(-\lambda)  e^{ik_\ell\bullet/\epsi}\right) \tK_\rho(-\lambda)^N \right|_\BB \left|(D^2-\lambda^2)^N I_{W_{k_{\ell_1}},1}(\lambda) \right|_\BB .\end{gathered}
\end{equation}
The same arguments used to show \eqref{eq:11v} yield that for $\lambda \in [1,M]$
\begin{equation*}
\left|\left(\prod_{\ell=1}^{\ell_1-1} K_{W_{k_\ell}}(-\lambda)  e^{ik_\ell\bullet/\epsi}\right) \tK_\rho(-\lambda)^N \right|_\BB \leq  C^{\ell_1^2} \lr{\lambda}^{\ell_1} \left(\prod_{\ell=1}^{\ell_1-1} \|W_{k_\ell}\|_{2\nu}\right) \epsi^N.
\end{equation*}
By Lemma \ref{lem:1a}, $\left|(D^2-\lambda^2)^N I_{W_{k_{\ell_1}},1}(\lambda) \right|_\BB \leq \lr{\lambda}^{2N+d} e^{2L|\lambda|} |W_{k_{\ell_1}}|_\infty$ for $\Im \lambda \in [-1,-M]$. Coming back to \eqref{eq:3g} and putting these bounds together we obtain
\begin{equation*}
\left|\left(\prod_{\ell=1}^{\ell_1-1} K_{W_{k_\ell}}(-\lambda)  e^{ik_\ell\bullet/\epsi}\right) I_{W_{k_{\ell_1}},1}(\lambda)  e^{ik_{\ell_1}\bullet/\epsi}\right|_\BB \leq C^{\ell_1^2} \lr{\lambda}^{\ell_1+2N+d} e^{2L|\lambda|} \left(\prod_{\ell=1}^{\ell_1} \|W_{k_\ell}\|_{2\nu}\right) \epsi^N.
\end{equation*}
By \eqref{eq:3c} we conclude that if the first sequence among \eqref{eq:11y} is destructive we have
\begin{equation*}
\left|\prod_{\ell=1}^{\nu} I_{W_{k_\ell},\epsilon_\ell}(\lambda)  e^{ik_\ell\bullet/\epsi}\right|_\BB \leq C_M^{\nu^2} \lr{\lambda}^{\nu+2N+d} e^{2L|\lambda|} \left(\prod_{\ell=1}^{\ell_1} \|W_{k_\ell}\|_{2\nu}\right) \epsi^N
\end{equation*}
uniformly for $\lambda$ with $\Im \lambda \leq M$. In the case where the first subsequence among \eqref{eq:11y} is not destructive we know that at least one of the subsequence in \eqref{eq:11y} is destructive. This subsequence might not start nor end with $N$ vanishing term. Here again using that the operator $I_{\VV,1}(\lambda)$ is smoothing we can overcome this difficulty. We skip the additional details. It leads to the bound 
\begin{equation}\label{eq:11x}
\left|\prod_{\ell=1}^{\nu} I_{W_{k_\ell},\epsilon_\ell}(\lambda)  e^{ik_\ell\bullet/\epsi}\right|_\BB \leq C_M^{\nu^2} \lr{\lambda}^{2\nu+4N+2d} e^{4L|\lambda|} \left(\prod_{\ell=1}^{\ell_1} \|W_{k_\ell}\|_{2\nu}\right) \epsi^N.
\end{equation}
Sum the bound \eqref{eq:11x} over $\epsilon_1, ..., \epsilon_\nu \in \{0,1\}^\nu$ to get that when $\Im \lambda \in [-1,-M]$,
\begin{equation}\label{eq:2v}
\left|\prod_{\ell=1}^\nu K_{W_{k_\ell}} e^{ik_\ell\bullet/\epsi}\right|_\BB \leq C_M^{\nu^2} \lr{\lambda}^{2\nu+4N+2d} e^{4L|\lambda|} \left(\prod_{\ell=1}^{\nu} \|W_{k_\ell}\|_{2\nu}\right) \epsi^N.
\end{equation}

3. We conclude the proof by a complex analysis argument. The estimates \eqref{eq:11v} and \eqref{eq:2v} show that \eqref{eq:11e} holds locally for $|\Im \lambda| \geq 1$. Thus it remains to show that it holds locally for $|\Im \lambda| \leq 1$. Fix $u,v \in L^2$ and consider
\begin{equation*}
f(\lambda) = \dfrac{\lambda^\nu e^{-\lambda^2}}{(\lambda+2i)^{2\nu +4N+2d}} \lr{\prod_{\ell=1}^\nu K_{W_{k_\ell}} e^{ik_\ell\bullet/\epsi}u,v}.
\end{equation*}
This function is holomorphic in the strip $|\Im \lambda| \leq 1$. For $\lambda$ with $|\Im \lambda| \leq 1$, $-\Re(\lambda^2) = -|\Re \lambda|^2 + |\Im \lambda|^2 \leq 1$ which shows that $e^{-\lambda^2}$ is uniformly bounded in this strip. Hence
\begin{equation*}
|\Im \lambda| \leq 1 \ \Rightarrow \ |f(\lambda)| \leq C^{\nu} \left(\prod_{\ell=1}^\nu |W_{k_\ell}|_\infty\right) |u|_2 |v|_2.
\end{equation*}
In addition 
\begin{equation*}
|\Im \lambda| = 1 \ \Rightarrow \ -\Re(\lambda^2) +4L|\lambda| 
 \leq 1+4L+4L^2.
\end{equation*}
Therefore $e^{-\lambda^2+4L|\lambda|}$ is uniformly bounded on $|\Im \lambda| = 1$. Hence on the boundary of the strip $|\Im \lambda| \leq 1$ the estimates \eqref{eq:11v} and \eqref{eq:2v} imply the bound
\begin{equation}\label{eq:11w}
|f(\lambda)| \leq C^{\nu^2} \left(\prod_{\ell=1}^\nu \|W_{k_\ell} \|_{2\nu}\right)\epsi^N |u|_2 |v|_2.
\end{equation}
Therefore by the three lines theorem the function $f$ satisfies \eqref{eq:11w} for all $\lambda$ with $|\Im \lambda| \leq 1$. Taking the supremum over $u,v \in L^2$ shows that \eqref{eq:11e} holds for $|\Im \lambda| \leq 1$. This ends the proof of the lemma. \end{proof}

Lemma \ref{lem:1} is somehow unsatisfying. The bound \eqref{eq:11e} involves a constant $C^{\nu^2}$ and the norm $\| W_k\|_{2\nu}$. Both $C^{\nu^2}$ and $\|W_k\|_{2\nu}$ grow too fast as $\nu \rightarrow \infty$. The next result uses combinatorial arguments to refine Lemma \ref{lem:1}.

\begin{lem}\label{lem:3} Let $\{k_\ell\}_{1 \leq \ell \leq \nu}$ be a destructive sequence with exactly $\gamma$ non-vanishing terms. Let $s = 2(N+\gamma N+1)$. If $\nu \geq (2N+2d)(\gamma+1)$ then locally uniformly on $X_d$
\begin{equation*}
\left|\trace\left(\prod_{\ell=1}^\nu K_{W_{k_\ell}} e^{ik_\ell\bullet/\epsi}\right) \right| \leq C^{\nu+s^2} \epsi^N. \prod_{\ell=1}^{\nu} \|W_{k_\ell}\|_s.
\end{equation*}
If moreover the sequence $\{k_\ell\}_{1 \leq \ell \leq \nu}$ starts and ends with $N+d$ zeros then locally uniformly on $X_d$
\begin{equation}\label{eq:11h}
\left|\prod_{\ell=1}^\nu K_{W_{k_\ell}} e^{ik_\ell\bullet/\epsi} \right|_\LL \leq C^{\nu+s^2} \epsi^N \prod_{\ell=1}^{\nu} \|W_{k_\ell}\|_s.
\end{equation}
\end{lem}

\begin{proof} First note that since $\nu \geq (2N+2d)(\gamma+1)$ by the pigeonhole principle there exists a cyclic permutation (in the sense precised above) of $\{k_\ell\}$ that starts and ends with $N+d$ zeros. Using the cyclicity of the trace we can assume that the sequence $\{k_\ell\}$ starts and ends with $N+d$ zeros. In particular we are reduced to prove \eqref{eq:11h}. Since the sequence $\{k_\ell\}$ is now admissible it admits a good subsequence $\{k_\ell\}_{q+1 \leq \ell \leq \nu'+q}$. Without loss of generality $q \geq d$. Write 
\begin{equation*}\begin{gathered}
 \left| \prod_{\ell=1}^\nu K_{W_{k_\ell}} e^{ik_\ell\bullet/\epsi} \right|_\LL \\
\leq  \left| \prod_{\ell=1}^d K_{W_{k_\ell}} e^{ik_\ell\bullet/\epsi} \right|_\LL \left| \prod_{\ell=d+1}^q K_{W_{k_\ell}} e^{ik_\ell\bullet/\epsi} \right|_\BB \left| \prod_{\ell=q+1}^{\nu'+q} K_{W_{k_\ell}} e^{ik_\ell\bullet/\epsi} \right|_\BB \left| \prod_{\ell=\nu'+q+1}^\nu K_{W_{k_\ell}} e^{ik_\ell\bullet/\epsi} \right|_\BB.\end{gathered}
\end{equation*}
For $\lambda$ in compact subsets of $X_d$ the first, second and fourth factor are estimated by Lemma \ref{lem:1h}. The third factor is controlled by \eqref{eq:11e}. It leads to
\begin{equation*}
\left| \prod_{\ell=1}^\nu K_{W_{k_\ell}} e^{ik_\ell\bullet/\epsi} \right|_\LL  \leq  C^{\nu+\nu'^2} \epsi^N \left(\prod_{\ell \leq q, \ \ell \geq \nu'+q+1} |W_{k_\ell}|_\infty \right) \left(\prod_{\ell=q+1}^{\nu'+q} \|W_{k_\ell}\|_{2\nu}\right)  \leq C^{\nu+s^2} \epsi^N \prod_{\ell=1}^\nu \|W_{k_\ell}\|_s.
\end{equation*}
This completes the proof of the lemma. \end{proof}

With this refinement we are now ready for the proof of Lemma \ref{lem:5}.

\begin{proof}[Proof of Lemma \ref{lem:5}] We divide the proof in 5 main steps. 

1. Let $a \in [1,N]$ and $n_1, ..., n_a \in [1, 2N-1]$. The function $z \mapsto (1 + \Psi(z))^{-1}$ is meromorphic with a simple pole at $z=-1$. Write a Taylor expansion of $z \mapsto (1+\Psi(z))^{-1}$ at $z=0$:
\begin{equation*}
(\Id + \Psi(z))^{-1} = P_N(z) + z^{2N+2d} \rho_N(z).
\end{equation*}
Here $P_N$ is a polynomial of degree $2N+2d-1$ and $\rho_N$ is a holomorphic function on $\C \setminus \{-1\}$. The pole at $-1$ is of multiplicity one. Away from resonances of $W_0$,
\begin{equation*}
(\Id + \Psi(K_{W_0}))^{-1} = P_N(K_{W_0}) + K_{W_0}^{N+d} \rho_N(K_{W_0}) K_{W_0}^{N+d}.
\end{equation*}
The operator $B_{W_0} = \Det(\Id+\Psi(K_{W_0})) \rho_N(K_{W_0})$, well defined on $\C \setminus \Res(K_{W_0})$, extends to an entire family of operators by Appendix \ref{app:1}. Let us write $\TT_X = C_0+C_1$ where
\begin{equation}\label{eq:1s}\begin{gathered}
C_0 = \Det(\Id+\Psi(K_{W_0})) \cdot P_N(K_{W_0}), \ \ \ \ 
  C_1 =K_{W_0}^{N+d} \cdot B_{W_0} \cdot K_{W_0}^{N+d}.\end{gathered}
\end{equation}
Fix  $m_1, ..., m_a \geq p$ and for each $1 \leq j \leq a$ a sequence $\{k_\ell^j\} \in \SSSS_{m_j}^{n_j}$, with $k_1^1 + ... + k_{m_a}^a \neq 0$. We define $\gamma=n_1+...+n_a$ and $\nu = m_1+...+m_a$. Using $\TT_X = C_0+C_1$ we get
\begin{equation*}
 \trace\left(  \prod_{j=1}^a \TT_X \prod_{\ell=1}^{m_j} K_{W_{k_\ell^j}} e^{ik_\ell^j \bullet/\epsi} \right)  = \sum_{\epsilon_1, ... \epsilon_a \in \{0,1\}^a}\trace\left( \prod_{j=1}^a C_{\epsilon_j} \prod_{\ell=1}^{m_j} K_{W_{k_\ell^j}} e^{ik_\ell^j \bullet/\epsi} \right).
\end{equation*}
In the following steps we study separately the terms of the RHS sum, depeding on the value of $\epsilon_1, ..., \epsilon_a \in \{0,1\}^a$.

2. Assume that $\epsilon_1= ... = \epsilon_a=0$. Then
\begin{equation*}
\trace\left( \prod_{j=1}^a C_{\epsilon_j} \prod_{\ell=1}^{m_j} K_{W_{k_\ell^j}} e^{ik_\ell^j \bullet/\epsi} \right) = \trace\left( \prod_{j=1}^a C_0 \prod_{\ell=1}^{m_j} K_{W_{k_\ell^j}} e^{ik_\ell^j \bullet/\epsi} \right).
\end{equation*}
The sequence $\{k_\ell^j\}$ is destructive, $\nu \geq  p a \geq 2(N+d) \cdot 2Na$ and $2N a \geq \gamma+1$. This implies $\nu \geq 2(N+d)(\gamma+1)$. Hence for $s=2(N+2N^3 +1)$ we have $s \geq 2(N+\gamma N +1)$. The assumptions of Lemma  \ref{lem:3} are satisfied thus
\begin{equation*}
\left|\trace\left( \prod_{j=1}^a C_0 \prod_{\ell=1}^{m_j} K_{W_{k_\ell^j}} e^{ik_\ell^j \bullet/\epsi} \right)\right| \leq C^{\nu} \epsi^N \prod_{\ell=1}^{\nu} \|W_{k_\ell}\|_s
\end{equation*}
for a constant $C$ depending only on $N,d$ and $|W_0|_\infty$.

3. Assume that exactly one of the $\epsilon_1, ..., \epsilon_a \in \{0,1\}^a$ is equal to $1$. Using the cyclicity of the trace we can assume without loss of generality that $\epsilon_1=1$. Hence
\begin{equation*}\begin{gathered}
 \trace\left( \prod_{j=1}^a C_{\epsilon_j} \prod_{\ell=1}^{m_j} K_{W_{k_\ell^j}} e^{ik_\ell^j \bullet/\epsi} \right) \\
  =  \trace\left( B_{W_0} K_{W_0}^{N+d} \left( \prod_{\ell=1}^{m_1} K_{W_{k_\ell^1}} e^{ik_\ell^1 \bullet/\epsi}\right) \left(\prod_{j=2}^a C_0 \prod_{\ell=1}^{m_j} K_{W_{k_\ell^j}} e^{ik_\ell^j \bullet/\epsi} \right)  K_{W_0}^{N+d} \right).\end{gathered}
\end{equation*}
Using \eqref{eq:11h} we obtain again
\begin{equation*}\begin{gathered}
  \left|\trace\left( \prod_{j=1}^a C_{\epsilon_j} \prod_{\ell=1}^{m_j} K_{W_{k_\ell^j}} e^{ik_\ell^j \bullet/\epsi} \right)\right| \\ \leq  |B_{W_0}|_\BB \cdot \left|K_{W_0}^{N+d}\left( \prod_{\ell=1}^{m_1} K_{W_{k_\ell^1}} e^{ik_\ell^1 \bullet/\epsi} \right) \left(\prod_{j=2}^a C_0 \prod_{\ell=1}^{m_j} K_{W_{k_\ell^j}} e^{ik_\ell^j \bullet/\epsi}\right) K_{W_0}^{N+d}\right|_\LL.\end{gathered}
\end{equation*}
The second factor in the second line is a finite sum of terms studied in Lemma \ref{lem:3}. Consequently we obtain the bound
\begin{equation*}
\left|\trace\left( \prod_{j=1}^a C_{\epsilon_j} \prod_{\ell=1}^{m_j} K_{W_{k_\ell^j}} e^{ik_\ell^j \bullet/\epsi} \right)\right| \leq C^\nu \epsi^N |B_{W_0}|_\BB \prod_{\ell=1}^{\nu} \|W_{k_\ell}\|_s.
\end{equation*}

4. Assume that $2$ or more terms among $\epsilon_1, ..., \epsilon_a \in \{0,1\}^a$ are equal to $1$. Using a circular permutation we can assume without loss of generality that $\epsilon_1 = 1$. Let us prove the following statement: there exists two indexes $j_1, j_2 \in [1,a]$ such that
\begin{enumerate}
\item[$(i)$] the sequence $\{k_\ell^j\}^{j_1 \leq j < j_2}_{1 \leq \ell \leq m_j}$ is destructive;
\item[$(ii)$] $\epsilon_j = 0$ for all $j$ in the interval $(j_1,j_2)$.
\end{enumerate}
We process by recursion on $a$. If $a=2$ this is obvious: either the sequence $\{k_\ell^1\}_{1 \leq \ell \leq m_1}$ or the sequence $\{k_\ell^2\}_{1 \leq \ell \leq m_2}$ is destructive. Now assume that the statement holds true for all $a' \leq a-1$. Let us prove it for $a$. Let $j_0$ be the smallest index with $\epsilon_{j_0}=1$ and $j_0 > 1$. Then either the sequence $\{k_\ell^j\}^{1 \leq j < j_0}_{1 \leq \ell \leq m_j}$ is destructive and we are done, or it is constructive. But then the sequence   $\{k_\ell^j\}^{j_0 \leq j \leq a}_{1 \leq \ell \leq m_j}$ is destructive and so we can apply the recursion hypothesis to it. This proves the above claim.
 
Again using a circular permutation we can assume that $j_1=1$. Hence
\begin{equation*}\begin{gathered}
\left|\trace\left( \prod_{j=1}^a C_{\epsilon_j} \prod_{\ell=1}^{m_j} K_{W_{k_\ell^j}} e^{ik_\ell^j \bullet/\epsi} \right)\right| \\
\leq |B_{W_0}|_\BB \left| K_{W_0}^{2(N+d)} \left(\prod_{\ell=1}^{m_1} K_{W_{k_\ell^j}} e^{ik_\ell^j \bullet/\epsi} \right) \left( \prod_{j=2}^{j_2-1} C_0 \prod_{\ell=1}^{m_j} K_{W_{k_\ell^j}} e^{ik_\ell^j \bullet/\epsi} \right) K_{W_0}^{2(N+d)} \right|_\LL |B_{W_0}|_\BB \\ 
\cdot \left| \left(\prod_{\ell=1}^{m_{j_2}} K_{W_{k_\ell^j}} e^{ik_\ell^j \bullet/\epsi} \right) \left(\prod_{j=j_2+2}^{a} C_{\epsilon_j} \prod_{\ell=1}^{m_j} K_{W_{k_\ell^j}} e^{ik_\ell^j \bullet/\epsi}\right) \right|_\BB.\end{gathered}
\end{equation*} 
The first line is a finite sum of terms estimated by Lemma \ref{lem:3}. The second line is controlled by the standards bounds of Lemma \ref{lem:1h}. It leads to
\begin{equation}\label{eq:1q}
\left|\trace\left( \prod_{j=1}^a C_{\epsilon_j} \prod_{\ell=1}^{m_j} K_{W_{k_\ell^j}} e^{ik_\ell^j \bullet/\epsi} \right)\right| \leq C^{\nu} \epsi^N \prod_{\ell=1}^{\nu} \|W_{k_\ell}\|_s.
\end{equation}

5. Points $2,3,4$ show that \eqref{eq:1q} holds for all sequences $\epsilon_1, ..., \epsilon_a \in \{0,1\}^a$. Summing this estimate over all possible $\epsilon_1, ..., \epsilon_a$ to get
\begin{equation}\label{eq:2u}
\left|\trace\left(  \prod_{j=1}^a \TT_X \prod_{\ell=1}^{m_j} K_{W_{k_\ell^j}} e^{ik_\ell^j \bullet/\epsi} \right)\right| \leq C^{\nu} \epsi^N \prod_{\ell=1}^{\nu} \|W_{k_\ell}\|_s.
\end{equation}
The last step of the proof is to sum the bound \eqref{eq:2u}. Recall that
\begin{equation*}
\DD_{n_1, ..., n_a} = \sum_{m_1, ..., m_a=p}^\infty \ \ \  \sum_{\substack{\{k_\ell^1\} \in \SSSS_{m_1}^{n_1}, \ ..., \ \{k_\ell^a\} \in \SSSS_{m_a}^{n_a}, \\ k_1^1+...+k_{m_a}^a \neq 0 }} \ \ \  \left(\prod_{j=1}^a \az_{m_j} \TT_X \prod_{\ell=1}^{m _j}K_{W_{k_\ell^j}} e^{ik_\ell^j\bullet/\epsi}\right)
\end{equation*}
where $\SSSS_m^n$ is the set of sequences of length $m$ with $n$ non-vanishing terms. Hence
\begin{align*}
\left|\trace\left( \DD_{n_1, ..., n_a} \right)\right| & \leq  \sum_{m_1, ..., m_a=p}^\infty \ \ \  \sum_{\substack{\{k_\ell^1\} \in \SSSS_{m_1}^{n_1}, \ ..., \ \{k_\ell^a\} \in \SSSS_{m_a}^{n_a}, \\ k_1^1+...+k_{m_a}^a \neq 0 }} \ \ \ |\az_{m_1} ... \az_{m_a}| \left|\trace\left( \prod_{j=1}^a \TT_X \prod_{\ell=1}^{m _j}K_{W_{k_\ell^j}} e^{ik_\ell^j\bullet/\epsi}\right)\right| \\
   & \leq \epsi^N \sum_{m_1, ..., m_a=p}^\infty  |\az_{m_1} ... \az_{m_a}| C^{m_1+...+m_a} \sum_{\substack{\{k_\ell^1\} \in \SSSS_{m_1}^{n_1}, \ ..., \ \{k_\ell^a\} \in \SSSS_{m_a}^{n_a}, \\ k_1^1+...+k_{m_a}^a \neq 0 }} \  \prod_{\ell=1}^{m_1+...+m_a} \|W_{k_\ell}\|_s \\
   & \leq \epsi^N \sum_{m_1, ..., m_a=p}^\infty  |\az_{m_1} ... \az_{m_a}| C^{m_1+...+m_a} |W|_{Z^s}^{m_1+...+m_a}  = \epsi^N \left(\Phi(C|W|_{Z^s})\right)^a,
\end{align*}
where $\Phi(z) = \sum_{m=p}^\infty |\az_m|z^m $. Since $\Phi$ is entire, $\Phi(C|W|_{Z^s}) < \infty$. Hence $\trace(\DD_{n_1, ..., n_a}) = O(\epsi^N)$ which completes the proof. \end{proof}

\subsection{Constructive interaction} In this paragraph we prove the following lemma:

\begin{lem}\label{lem:61} For $1 \leq a \leq N$ and $n_1, ..., n_a \in [1,2N-1]$ let $\CC_{n_1, ..., n_a}$ be the trace class operator given by \eqref{eq:2w}. There exist $\varphi_0, ..., \varphi_{N-1}$ holomorphic functions on $X_d$ such that
\begin{equation*}
\trace\left(\CC_{n_1, ..., n_a}\right) = \varphi_0(\lambda) + \epsi \varphi_1(\lambda)  + ... + \epsi^{N-1} \varphi_{N-1}(\lambda)  + O(\epsi^N)
\end{equation*} 
locally uniformly on $X_d$.
\end{lem}

The first step is an operator valued expansion for $e^{-ik\bullet/\epsi} K_{W_k} e^{ik\bullet/\epsi}$

\begin{lem}\label{lem:9} For every $n \geq 0$ there exists some operators $A_0, ..., A_{n-1}, \mathfrak{R}_n$ with
\begin{equation}\label{eq:11p}
e^{-ik\bullet/\epsi} K_{W_k} e^{ik\bullet/\epsi} = A_0 + ... +\epsi^{n-1} A_{n-1} + \epsi^n \mathfrak{R}_n,
\end{equation}
such that:
 \begin{enumerate}
\item[$(i)$] $A_j$ is a pseudodifferential operator of order $j-2$ that maps locally supported functions to compactly supported functions. It does not depend on $\epsi$ and
\begin{equation*}
s+j \leq N \ \Rightarrow \ |A_j|_{\BB(H^{s+j}, H^s)} \leq C \|W_k\|_N.
\end{equation*}
\item[$(ii)$] $\mathfrak{R}_n$ is a pseudodifferential operator of order $n-1$ and maps locally supported functions to compactly supported functions. It depends on $\epsi$ and uniformly in $\epsi$ near $0$,
\begin{equation*}
s+n+1 \leq N \ \Rightarrow \ |\mathfrak{R}_n|_{\BB(H^{s+n+1},  H^s)} \leq C \|W_k\|_N.  
\end{equation*}
\end{enumerate}
\end{lem}

\begin{proof} For $k = 0$ there is nothing to prove. Thus we assume $k \neq 0$. In Appendix \ref{app:2}  we prove that if $R(\xi,\lambda) = (\xi^2-\lambda^2)^{-1}$ then
\begin{equation}\label{eq:11u}
R(\xi+k/\epsi,\lambda) = \left(\sum_{j=2}^{n-1} \epsi^j p_{j-2}(\xi,\lambda)\right) + \epsi^n p_{n-2}(\xi,\lambda) + \epsi^{n+1} p_{n-1}(\xi,\lambda) + \epsi^n \dfrac{r_n(\xi,\lambda,\epsi)}{(\xi^2-k/\epsi)^2-\lambda^2}.
\end{equation}
Here the $p_j(\xi,\lambda)$ are polynomials in $\xi$ and $\lambda$ of degree at most $j$ in $\xi$; and $r_n(\xi,\lambda,\epsi)$ is a polynomial in $\xi$ and $\lambda$ of degree at most $n+1$ in $\xi$ and whose coefficients depend smoothly of $\epsi$. In particular,
\begin{equation}\label{eq:11f}
\sup_{\xi \in \R^d}\left|\dfrac{r_n(\xi,\lambda,\epsi)}{\lr{\xi}^{n+1}}\right| = O(1) \text{ uniformly as } \epsi \rightarrow 0.
\end{equation}
Since $e^{-ik\bullet/\epsi}De^{ik\bullet/\epsi} = D+k/\epsi$ we have for $\Im \lambda > 0$,
\begin{equation*}
e^{-ik\bullet/\epsi} R_0(\lambda) e^{ik\bullet/\epsi} = \left((D+k/\epsi)^2-\lambda^2\right)^{-1}.
\end{equation*}
Therefore the expansion \eqref{eq:11u} implies that for $\Im \lambda > 0$,
\begin{equation*}
e^{-ik\bullet/\epsi} R_0(\lambda) e^{ik\bullet/\epsi} = \left(\sum_{j=2}^{n-1} \dfrac{\epsi^j}{|k|^2} p_{j-2}(D)\right) + \epsi^n\left( p_{n-2}(D) + \epsi p_{n-1}(D) + R_0(\lambda)r_n(D,\epsi)\right).
\end{equation*}
This identity extends analytically to  $X_d$ and yields
\begin{equation*}\begin{gathered}
e^{-ik\bullet/\epsi} K_{W_k} e^{ik\bullet/\epsi} = A_0 + ... +\epsi^{n-1} A_{n-1} + \epsi^n \mathfrak{R}_n,\\
A_0=A_1=0, \ \ \ 
A_j = \dfrac{1}{|k|^2} \rho p_{j-2}(D) W_k \text{ for } j \in [2,n-1], \\
\mathfrak{R}_n = \rho \left( p_{n-2}(D) + \epsi p_{n-1}(D) + R_0(\lambda)r_n(D,\epsi) \right)W_k.
\end{gathered}
\end{equation*}
The operators $A_j$ are pseudodifferential of order $j-2$ and map locally supported functions to compactly supported functions. For $\Im \lambda > 0$ the operator $K_{W_k}(\lambda)$ is pseudodifferential; the operator $K_{W_k}(\lambda) - K_{W_k}(-\lambda)$ is smoothing. Hence $K_{W_k}(\lambda)$ is pseudodifferential for all $\lambda \in X_d$. As
\begin{equation*}
\mathfrak{R}_n = \dfrac{e^{-ik\bullet/\epsi} K_{W_k} e^{ik\bullet/\epsi}- A_0 - ... -\epsi^{n-1} A_{n-1}}{\epsi^n}
\end{equation*}
and the RHS is pseudodifferential $\mathfrak{R}_n$ must also be pseudodifferential. To evaluate its order we note that $p_{n-2}(D)$ (resp. $p_{n-1}(D)$) is a differential operator of order $n-4$ (resp. $n-3$) and that $r_n(D)$ is a differential operator of order $n+1$. Thus $R_0(\lambda) r_n(D)$ maps $H^{n+1}$ to $H^2$ and $\mathfrak{R}_n$ must be of order $n-1$. To prove the required bounds, we note that for $s \leq N$, the multiplication operator $u \mapsto W_k u$ from $H^s$ to itself has norm bounded by $\|W_k\|_N$. Therefore for $s+j \leq N$,
\begin{equation*}
|A_j|_{\BB(H^{s+j}, H^s)} \leq C |p_{j-2}(D)|_{\BB(H^{s+j}, H^s)} |W_k|_{\BB(H^{s+j}, H^{s+j})} \leq C \|W_k\|_N.
\end{equation*}
This proves $(i)$. Now we prove $(ii)$. For $s+n+1 \leq N$ the bound \eqref{eq:11f} implies that the operator $r_n(D,\epsi)$ (which is a differential operator) satisfies the bound
\begin{equation*}
|r_n(D,\epsi)|_{\BB(H^{s+n+1},H^s)} = O(1) \text{ uniformly as } \epsi \rightarrow 0.
\end{equation*}
Let $\chi \in C_0^\infty(\Bb^d(0,L))$ with $\chi=1$ on $\supp(\chi)$. The operator $\rho R_0(\lambda) \chi$ maps $H^s$ to itself. Consequently uniformly as $\epsi \rightarrow 0$
\begin{equation*}
|\rho R_0(\lambda) r_n(D,\epsi) W_k|_{\BB(H^{s+n+1},H^s)} \leq |\rho R_0(\lambda) \chi|_{\BB(H^s,H^s)} |r_n(D,\epsi) W_k|_{\BB(H^{s+n+1},H^s)} = O(\| W_k\|_N)
\end{equation*}
The operators $\rho p_{n-2}(D)W_k$ and $\rho p_{n-1}(D)W_k$ do not depend on $\epsi$ and are bounded from $H^{s+n+1}$ to $H^s$. This shows $(ii)$ and completes the proof of the lemma. \end{proof}

Now we prove the same kind of expansion for product of operators of the form \eqref{eq:11p}.

\begin{lem}\label{lem:41} Let $\{k_\ell\}_{1 \leq \ell \leq \nu}$ be a sequence of $d$-tuples in $\Z^d$. There exists some operators $\AAA_0, ..., \AAA_{N-1}, \RR_N$ with
\begin{equation*}
e^{-i\sigma_1\bullet/\epsi} \left(\prod_{\ell=1}^{\nu} K_{W_{k_\ell}} e^{ik_\ell \bullet/\epsi}\right) = \AAA_0 +  ... + \epsi^{N-1} \AAA_{N-1}+\epsi^N \RR_N
\end{equation*}
where $\sigma_1 = k_1+...+k_\nu$ and
\begin{enumerate}
\item[$(i)$] $\AAA_j$ is a pseudodifferential operator of  order $j-2$ and maps locally supported functions to compactly supported functions. It does not depend on $\epsi$ and
\begin{equation*}
s+j \leq N \ \Rightarrow \ |\AAA_j|_{\BB(H^{s+j}, H^s)} \leq C^\nu \prod_{\ell=1}^\nu \|W_{k_\ell}\|_N
\end{equation*}
\item[$(ii)$] $\RR_N$ is a pseudodifferential operator of order $N-1$ mapping locally supported functions to compactly supported functions. It depends on $\epsi$ and uniformly in $\epsi$ near $0$,
\begin{equation*}
s \leq -1 \ \Rightarrow \ |\RR_N|_{\BB(H^{s+N+1}, H^s)} \leq C^\nu \prod_{\ell=1}^\nu \|W_{k_\ell}\|_N. 
\end{equation*}
\end{enumerate}
\end{lem}

\begin{proof} We prove this lemma by recursion. For $\nu=1$ it is the result of Lemma \ref{lem:9}. Now assume that Lemma \ref{lem:41} holds true for all sequences $\{k_j\}$ of length less or equal to $\nu-1$. Let $\{k_j\}$ be a sequence of length $\nu$. Define $\sigma_2=k_2+...+k_\nu$ so that
\begin{equation*}
e^{-i\sigma_1\bullet/\epsi} \prod_{\ell=1}^{\nu} K_{W_{k_\ell}} e^{ik_\ell \bullet/\epsi}  = \left(e^{-i\sigma_1\bullet/\epsi} K_{W_{k_1}} e^{i\sigma_1 \bullet/\epsi}\right) \cdot \left( e^{-i\sigma_2\bullet/\epsi}\prod_{\ell=2}^{\nu} K_{W_{k_\ell}}e^{ik_\ell \bullet/\epsi}\right).
\end{equation*}
Using the recursion hypothesis we have
\begin{equations*}
 \left(e^{-i\sigma_1\bullet/\epsi} K_{W_{k_1}} e^{i\sigma_1 \bullet/\epsi}\right) \cdot \left( e^{-i\sigma_2\bullet/\epsi}\prod_{\ell=2}^{\nu} K_{W_{k_\ell}}e^{ik_\ell \bullet/\epsi}\right) \\ 
=  \left(e^{-i\sigma_1\bullet/\epsi} K_{W_{k_1}} e^{i\sigma_1 \bullet/\epsi}\right) \AAA_0 + ... + \epsi^{N-1} \left(e^{-i\sigma_1\bullet/\epsi} K_{W_{k_1}} e^{i\sigma_1 \bullet/\epsi}\right) \AAA_{N-1} + \epsi^N \left(e^{-i\sigma_1\bullet/\epsi} K_{W_{k_1}} e^{i\sigma_1 \bullet/\epsi}\right) \RR_N.
\end{equations*}
We expand below $e^{-i\sigma_1\bullet/\epsi} K_{W_{k_1}} e^{i\sigma_1 \bullet/\epsi}$ at order $N-j$ as given by Lemma \ref{lem:9}:
\begin{equation}
e^{-i\sigma_1\bullet/\epsi} K_{W_{k_1}} e^{i\sigma_1 \bullet/\epsi} = A_0 + \epsi A_1 + ... + \epsi^{N-j-1} A_{N-j-1} + \epsi^{N-j} \mathfrak{R}_{N-j}.
\end{equation}
It leads to
\begin{equation*}
\epsi^j \left(e^{-i\sigma_1\bullet/\epsi} K_{W_{k_1}} e^{i\sigma_1 \bullet/\epsi}\right) \AAA_j = \epsi^j A_0 \AAA_j + ... + \epsi^{N-1} A_{N-1-j} \AAA_j + \epsi^N \mathfrak{R}_{N-j} \AAA_j.
\end{equation*}
The operator $A_{j'} \AAA_j$ has order $j'-2+j-2 = j'+j-4 \leq j'+j-2$ and in the above expression it is weighted with a term $\epsi^{j'+j}$. Moreover if $s+j'+j \leq N$ then
\begin{equation*}
|A_{j'} \AAA_j|_{\BB(H^{s+j'+j}, H^s)} \leq |A_{j'}|_{\BB(H^{s+j'},H^s)}|\AAA_j|_{\BB(H^{s+j'+j}, H^{s+j'})} \leq C^\nu \prod_{\ell=1}^\nu \|W_{k_\ell}\|_N. 
\end{equation*}
The remainder $\mathfrak{R}_{N-j} \AAA_j$ has order $N-j-1+j-2=N-3 \leq N-1$ and satisfies
\begin{equation*}
|\mathfrak{R}_{N-j} \AAA_j|_{\BB(H^{N+1+s}, H^s)} \leq |\mathfrak{R}_{N-j}|_{\BB(H^{s+N-j+1}, H^s)}|\AAA_j|_{\BB(H^{N+1+s}, H^{N+1+s-j})} \leq C^\nu \prod_{\ell=1}^\nu \|W_{k_\ell}\|_N. 
\end{equation*}
The term $e^{-i\sigma_1\bullet/\epsi} K_{W_{k_1}} e^{i\sigma_1 \bullet/\epsi} \RR_N$ is of order $N-3 \leq N-1$ and satisfies
\begin{align*}
\left|  e^{-i\sigma_1\bullet/\epsi} K_{W_{k_1}} e^{i\sigma_1 \bullet/\epsi}  \RR_N \right|_{\BB(H^{s+N+1}, H^s)}  \leq & \left| e^{-i\sigma_1\bullet/\epsi} K_{W_{k_1}} e^{i\sigma_1 \bullet/\epsi} \right|_{\BB(H^s, H^s)} |\RR_N|_{\BB(H^{s+N+1}, H^s)}\\ 
\leq & C^\nu \prod_{\ell=1}^\nu \|W_{k_\ell}\|_N.
\end{align*}
This prove that the lemma holds for all sequences of length $\nu$. This completes the recursion and ends the proof.\end{proof}

The expansion of Lemma \ref{lem:41} implies a trace expansion as follows:

\begin{lem}\label{lem:6} Let $\{ k_\ell \}_{1 \leq \ell \leq \nu}$ be a constructive sequence with $\gamma$ non-vanishing terms. Assume that $\nu \geq N(\gamma+1).$ Then there exists $a_0, a_1, ..., a_{N-1}$ holomorphic functions on $X_d$ such that locally uniformly on $X_d$, $ |a_j(\lambda)| \leq C^\nu \prod_{\ell=1}^\nu \|W_{k_\ell}\|_N$ and
\begin{equation*}
\left|\trace\left(\prod_{\ell=1}^{\nu} K_{W_{k_\ell}} e^{ik_\ell \bullet/\epsi}\right) - a_0(\lambda) - \epsi a_1(\lambda) + ... - \epsi^{N-1} a_{N-1}(\lambda)\right| \leq \epsi^N C^\nu \prod_{\ell=1}^\nu \|W_{k_\ell}\|_N.
\end{equation*}
\end{lem}

\begin{proof} Since $\nu \geq N(\gamma+1)$, there exists a subsequence of $\{ k_\ell \}_{1 \leq \ell \leq \nu}$ made of $N$ vanishing $d$-tuples. Using the cyclicity of the trace we can assume that $k_{\nu-N+1} = ... = k_\nu = 0$. The sequence $k_1, ..., k_{\nu-N}$ is constructive. Therefore we can apply Lemma \ref{lem:41} to obtain the expansion
\begin{equation*}
\prod_{\ell=1}^{\nu-N} K_{W_{k_\ell}} e^{ik_\ell \bullet/\epsi} = \AAA_0 +  ... + \epsi^{N-1} \AAA_{N-1}+\epsi^N \RR_N.
\end{equation*}
Here $\AAA_j$ is pseudodifferential of order $j-2$ and does not depend on $\epsi$ and $\RR_N$ is pseudodifferential of order $N-1$ and satisfies the bound
\begin{equation*}
|\RR_N|_{\BB(H^{N+1},L^2)} \leq C^\nu\prod_{\ell=1}^\nu \|W_{k_\ell}\|_N.
\end{equation*}
Both these operators map locally supported functions to compactly supported functions. As $k_{\nu-N+1} = ... = k_\nu = 0$ we obtain
\begin{equation}\label{eq:11o}
\prod_{\ell=1}^{\nu} K_{W_{k_\ell}} e^{ik_\ell \bullet/\epsi} = \AAA_0 K_{W_0}^N +  ... + \epsi^{N-1} \AAA_{N-1} K_{W_0}^N+\epsi^N \RR_N K_{W_0}^{N}.
\end{equation}
We recall that $N \geq d$. The operators $\AAA_j K_{W_0}^N$ have order $j-2-2N \leq -2-N \leq -2-d$ therefore they are trace class. The operator $\RR_N K_{W_0}^N$ has order $-N-1 \leq -d$ hence it is also trace class. It satisfies the bound
\begin{equation*}
|\RR_N K_{W_0}^N|_{\BB(H^{1-N}, L^2)} \leq |\RR_N|_{\BB(H^{N+1}, L^2)} |K_{W_0}^N|_{\BB(H^{1-N}, H^{N+1})} \leq C^\nu \prod_{\ell=1}^\nu \|W_{k_\ell}\|_N.
\end{equation*}
By \cite[Equation (B.3.9)]{DyaZwo} this implies $|\RR_N K_{W_0}^N|_\LL \leq |\RR_N K_{W_0}^N|_{\BB(H^{1-N}, L^2)}  \leq C^\nu \prod_{\ell=1}^\nu \|W_{k_\ell}\|_N$. Taking the trace of both sides of \eqref{eq:11o} yields
\begin{equation*}
\left|\trace\left(\prod_{\ell=1}^{\nu} K_{W_{k_\ell}} e^{ik_\ell \bullet/\epsi}\right) - \trace\left(\AAA_0 K_{W_0}^N\right) -  ... - \epsi^{N-1} \trace\left(\AAA_{N-1} K_{W_0}^N\right)\right| \leq \epsi^N C^\nu \prod_{\ell=1}^\nu \|W_{k_\ell}\|_N.
\end{equation*}
This gives the required expansion. We now need to prove the estimate on the coefficients $a_0, ..., a_{N-1}$ appearing in the expansion. By \cite[Equation (B.3.9)]{DyaZwo} and the estimate $(i)$ of Lemma \ref{lem:41},
\begin{equations*}
|\trace(\AAA_j K_{W_0}^N)|   \leq |\AAA_j K_{W_0}^N|_\LL    \leq C |\AAA_j K_{W_0}^N|_{\BB(H^{-N},L^2)} \\
   \leq C |\AAA_j|_{\BB(H^N,L^2)} |K_{W_0}^N|_{\BB(L^2,H^N)}  \leq C^\nu \prod_{\ell=1}^\nu \|W_{k_\ell}\|_N.
\end{equations*}
This completes the proof.\end{proof}

Fix $a \in [1,N]$ and $n_1, ..., n_a \in [1,2N-1]$. The operator $\CC_{n_1, ..., n_a}$ defined by \eqref{eq:2w} is a linear combination of operators of the form
\begin{equation}\label{eq:1r}
L[k_\ell^j] = \prod_{j=1}^a \TT_X \prod_{\ell=1}^{m_j} K_{W_{k_\ell^j}} e^{ik_\ell^j \bullet/\epsi},
\end{equation}
where
\begin{enumerate}
\item[$(i)$] For every $j \in [1,a]$, $m_j \geq p$.
\item[$(ii)$] The sequence $\{k_\ell^j\}^{1 \leq j \leq a}_{1 \leq \ell \leq m_j}$ is constructive.
\item[$(iii)$] For every $j \in [1,a]$, the sequence $\{k_\ell^j\}_{1 \leq \ell \leq m_j}$ has $n_j$ non-vanishing terms.
\end{enumerate}
In order to prove Lemma \ref{lem:61} we prove an expansion for operators of the form \eqref{eq:1r} where $\{k_\ell^j\}$ satisfies $(i), (ii)$ and $(iii)$. We fix $s=2(N+2N^2+1)$.

\begin{lem}\label{lem:2w} Let $L[k_\ell^j]$ be an operator of the form \eqref{eq:1r} where $\{k_\ell^j\}$ satisfies $(i), (ii)$ and $(iii)$ above. 
Then there exist $b_0[k_\ell^j], ..., b_{N-1}[k_\ell^j]$ holomorphic functions on $X_d$ such that locally uniformly on $X_d$ we have $|b_i[k_\ell^j]| \leq C^\nu \prod_{\ell=1}^\nu \|W_{k_\ell}\|_s$ and
\begin{equation}\label{eq:11r}
\left|\trace\left( L[k_\ell^j] \right) - b_0[k_\ell^j]+ ... - b_{N-1}[k_\ell^j] \epsi^{N-1}\right| \leq \epsi^N C^\nu \prod_{j=1}^a\prod_{\ell=1}^{m_j} \|W_{k_\ell^j}\|_s.
\end{equation} 
\end{lem}

\begin{proof}[Proof of Lemma \ref{lem:61}] Fix $a \in [1,N]$, $n_1, ..., n_a \in [1,2N-1]$ and $k_\ell^j$ satisfying $(i), (ii)$ and $(iii)$ above. Let $\gamma = n_1+...+n_a$ be the number of non-vanishing terms of $\{k_\ell^j\}$. We divide the proof below in 5 main steps.

1. Write $\TT_X = C_0+C_1$ where $C_0, C_1$ were given in \eqref{eq:1s}. Then
\begin{align*}
\trace\left(L[k_\ell^j]\right)=  \trace\left(  \prod_{j=1}^a \TT_X \prod_{\ell=1}^{m_j} K_{W_{k_\ell^j}} e^{ik_\ell^j \bullet/\epsi} \right)  = \sum_{\epsilon_1, ... \epsilon_a \in \{0,1\}^a}\trace\left( \prod_{j=1}^a C_{\epsilon_j} \prod_{\ell=1}^{m_j} K_{W_{k_\ell^j}} e^{ik_\ell^j \bullet/\epsi} \right).
\end{align*}
We recall that since $\{k_\ell^j\}$ has $\gamma \leq (2N-1) a$ non-vanishing terms and length $\nu \geq pa$ we have $\nu \geq N(\gamma+1)$. Fix a sequence $\epsilon_j \in \{0,1\}^a$. In order to prove the lemma it suffices to prove that the term 
\begin{equation}\label{eq:99}
\trace\left( \prod_{j=1}^a C_{\epsilon_j} \prod_{\ell=1}^{m_j} K_{W_{k_\ell^j}} e^{ik_\ell^j \bullet/\epsi} \right)
\end{equation}
admits an expansion in powers of $\epsi$ at order $N$.

2. Assume that $\epsilon_1= ... = \epsilon_a = 0$. Recall that $C_0$ is a polynomial in $K_{W_0}$. In this case \eqref{eq:99} is a finite sum of terms studied in Lemma \ref{lem:6}. These all admit an expansion in powers of $\epsi$ and thus so does \eqref{eq:99}. 

3. Assume that $\epsilon_j$ has at least one non-zero term. Without loss of generality $\epsilon_1 = 1$. The indexes $j_1, ..., j_r$ such that $\epsilon_j = 1$ split the sequence $(k_1^1, ..., k_{m_1}^1, k_1^2, ..., k_{m_a}^a)$ into $r+1$ subsequences
\begin{equation}\label{eq:1t}
(k_1^1, ..., k_{m_{j_1}}^{j^1}), \ ..., \ (k_1^{j_r}, ..., k_{m_a}^{a}).
\end{equation}
Assume that each of the subsequences in \eqref{eq:1t} is constructive. Then since $C_1= K_{W_0}^{N+d} B_{W_0} K_{W_0}^{N+d}$ we can write \eqref{eq:99} as the trace of a product of operators of the form
\begin{equation}\label{eq:3h}
B_{W_0} K_{W_0}^{N+d} \prod_{j=j_t}^{j_{t+1}-1} C_{\epsilon_j} \prod_{\ell=1}^{m_j} K_{W_{k_\ell^j}} e^{ik_\ell^j \bullet/\epsi} K_{W_0}^{N+d}.
\end{equation}
By Lemma \ref{lem:41} the operator 
\begin{equation*}
K_{W_0}^{N+d} \prod_{j=j_t}^{j_{t+1}-1} C_{\epsilon_j} \prod_{\ell=1}^{m_j} K_{W_{k_\ell^j}} e^{ik_\ell^j \bullet/\epsi} K_{W_0}^{N+d}
\end{equation*}
admits an operator-valued expansion in powers of $\epsi$. Thus so does the operator \eqref{eq:3h}. Multiplying these expansions over $t=1, ..., r$ leads to an operator-valued expansion for the operator
\begin{equation*}
\prod_{j=1}^a C_{\epsilon_j} \prod_{\ell=1}^{m_j} K_{W_{k_\ell^j}} e^{ik_\ell^j \bullet/\epsi}
\end{equation*}
in the spirit of Lemma \ref{lem:41}. Taking the trace and adapting  the proof of Lemma \ref{lem:6} shows that \eqref{eq:99} admits an expansion in powers of $\epsi$. 

4. Assume that at least one of the sequences in \eqref{eq:1t} is destructive. Without loss of generality $(k_1^1, ..., k_{j_1}^{m_{j_1}})$ is destructive. Since $C_1 = K_{W_0}^{N+d} B_{W_0} K_{W_0}^{N+d}$ the operator
\begin{equation}\label{eq:11q}
K_{W_0}^{N+d} \left(\prod_{j=1}^{j_1} C_0 \prod_{\ell=1}^{m_j} K_{W_{k_\ell^j}} e^{ik_\ell^j \bullet/\epsi}\right) K_{W_0}^{N+d}
\end{equation}
appears as one of the factors in the product 
\begin{equation*}
\prod_{j=1}^a C_{\epsilon_j} \prod_{\ell=1}^{m_j} K_{W_{k_\ell^j}} e^{ik_\ell^j \bullet/\epsi}.
\end{equation*}
In addition since $C_0$ is a polynomial in $K_{W_0}$ it is associated with a destructive sequence, that starts and ends with $N+d$ zeros. Consequently Lemma \ref{lem:3} applies and yields
\begin{equation*}
\left|K_{W_0}^{N+d} \left(\prod_{j=1}^{j_1} C_0 \prod_{\ell=1}^{m_j} K_{W_{k_\ell^j}} e^{ik_\ell^j \bullet/\epsi}\right) K_{W_0}^{N+d}\right|_\LL \leq C^{\nu+s^2} \epsi^N \prod_{j=1}^{j_1}\prod_{\ell=1}^{m_j}  \| W_{k^j_\ell} \|_s.
\end{equation*}
This yields the estimate:
\begin{align*}
 \trace\left(\prod_{j=1}^a C_{\epsilon_j}\right. & \left. \prod_{\ell=1}^{m_j} K_{W_{k_\ell^j}} e^{ik_\ell^j \bullet/\epsi}\right) \leq  \left|K_{W_0}^{N+d} B_{W_0}\right|_\BB \left|K_{W_0}^{N+d} \left(\prod_{j=1}^{j_1} C_0 \prod_{\ell=1}^{m_j} K_{W_{k_\ell^j}} e^{ik_\ell^j \bullet/\epsi}\right) K_{W_0}^{N+d}\right|_\LL \\ & \cdot \left|B_{W_0} K_{W_0}^{N+d}\right|_\BB \left|\prod_{\ell=1}^{m_{j_1}+1} K_{W_{k_\ell^{j_1+1}}} e^{ik_\ell^{j_1+1} \bullet/\epsi}\right|_\BB \left|\prod_{j=j_1+2}^a C_{\epsilon_j} \prod_{\ell=1}^{m_j} K_{W_{k_\ell^j}} e^{ik_\ell^j \bullet/\epsi}\right|_\BB \\
& \leq C^{\nu+s^2} \epsi^N \prod_{j=1}^a \prod_{\ell=1}^{m_j} \|W_{k_\ell^j}\|_s.
\end{align*}
This shows that such sequences $\epsilon_1, ..., \epsilon_a$ induce negligible contributions.

5. Points $2,3,4$ include all the possible values of $\epsilon_1, ..., \epsilon_a$. The expansion \eqref{eq:11r} follows now from a summation over $\epsilon_1, ..., \epsilon_a \in \{0,1\}^a$ of the expansions obtained in Points 2,3. This ends the proof.\end{proof}

We are now ready to prove Lemma \ref{lem:61}. 

\begin{proof}[Proof of Lemma \ref{lem:61}] Let us recall that for $a \in [1,N]$ and $n_1, ... n_a \in [1,2N-1]$ the operator $\CC_{n_1, ..., n_a}$ is defined by
\begin{equation*}
\CC_{n_1, ..., n_a} = \sum_{m_1, ..., m_a=p}^\infty \az_{m_1}... \az_{m_a}  \sum_{\substack{\{k_\ell^1\} \in \SSSS_{m_1}^{n_1}, \ ..., \ \{k_\ell^a\} \in \SSSS_{m_a}^{n_a}, \\ k_1^1+...+k_{m_a}^a = 0}} \ \ \ L[k_\ell^j].
\end{equation*}
Here $L[k_\ell^j]$ is given by \eqref{eq:1r}, $\SSSS_m^n$ is the set of sequences of length $m$ with $n$ non-vanishing terms and $\az_m = \Psi^{(m)}(0)/m!$. The proof consist in showing that the sum of the expansions of $\trace(L[k_\ell^j])$ provided by Lemma \ref{lem:2w} is convergent. By Lemma \ref{lem:2w},
\begin{align*}
 & \sum_{m_1, ..., m_a=p}^\infty  |\az_{m_1}... \az_{m_a}|  \sum_{\substack{\{k_\ell^1\} \in \SSSS_{m_1}^{n_1}, \ ..., \ \{k_\ell^a\} \in \SSSS_{m_a}^{n_a}, \\ k_1^1+...+k_{m_a}^a = 0}} \ \ \ \left|\trace\left( L[k_\ell^j] \right) - b_0[k_\ell^j]+ ... - b_{N-1}[k_\ell^j] \epsi^{N-1}\right| \\ \leq  & \epsi^N \sum_{m_1, ..., m_a=p}^\infty  |\az_{m_1}... \az_{m_a}|  \sum_{\substack{\{k_\ell^1\} \in \SSSS_{m_1}^{n_1}, \ ..., \ \{k_\ell^a\} \in \SSSS_{m_a}^{n_a}, \\ k_1^1+...+k_{m_a}^a = 0}}  C^{m_1+...+m_a} \prod_{j=1}^a\prod_{\ell=1}^{m_j} \|W_{k_\ell}\|_s \\
   & \leq \epsi^N \sum_{m_1, ..., m_a=p}^\infty  |\az_{m_1} ... \az_{m_a}| C^{m_1+...+m_a} |W|_{X^s}^{m_1+...+m_a}   = \epsi^N (\Phi(C|W|_{Z^s}))^a,
\end{align*}
where $\Phi(z) = \sum_{m=p}^\infty |\az_m| z^m$. It follows that $\trace(\CC_{n_1, ..., n_a})$ has an expansion given by
\begin{align*}
\trace(\CC_{n_1, ..., n_a}) & = O(\epsi^N)+\sum_{m_1, ..., m_a=p}^\infty \az_{m_1}... \az_{m_a}  \sum_{\substack{\{k_\ell^1\} \in \SSSS_{m_1}^{n_1}, \ ..., \ \{k_\ell^a\} \in \SSSS_{m_a}^{n_a}, \\ k_1^1+...+k_{m_a}^a = 0}} \ \ \  b_0[k_\ell^j]+ ... + b_{N-1}[k_\ell^j] \epsi^{N-1} \\
    & = \varphi_0 + ... + \epsi^{N-1} \varphi_{N-1} + O(\epsi^N),
\end{align*}
where
\begin{equation*}
\varphi_i = \sum_{m_1, ..., m_a=p}^\infty \az_{m_1}... \az_{m_a}  \sum_{\substack{\{k_\ell^1\} \in \SSSS_{m_1}^{n_1}, \ ..., \ \{k_\ell^a\} \in \SSSS_{m_a}^{n_a}, \\ k_1^1+...+k_{m_a}^a = 0}} b_i[k_\ell^j].
\end{equation*}
This ends the proof. \end{proof}

Since 
\begin{equation*}
\prod_{j=1}^a \TT_X F_{n_j} = \CC_{n_1, ..., n_a} + \DD_{n_1, ..., n_a},
\end{equation*}
the combination of Lemma \ref{lem:4}, Lemma \ref{lem:5} and Lemma \ref{lem:61} proves Lemma \ref{lem:2}. This in turn shows that $D_V(\lambda)$ admits an expansion in powers of $\epsi$. In the next section we conclude the proof of Theorem \ref{thm:5} by computing explicitly the first few coefficients in the expansion. 

\subsection{Computation of coefficients in the expansion}\label{subsec:3}

Here we use the first steps in the proof of Theorem \ref{thm:5} to compute a few coefficients in the expansion of $D_V$. These coefficients are holomorphic functions of $\lambda$. Hence it suffices to compute them for $\Im \lambda \gg 1$ and extends the obtained expression to $\C$ by the unique continuation principle. Fix $N \geq 4$ and $p=4N(d+N)$. If $\Im \lambda$ is large enough then $|K_V^p|_\LL < 1$. In this case the series
\begin{equation*}
\ln(1+\Psi(K_V)) = -\sum_{m=p}^\infty \dfrac{(-K_V)^m}{m}
\end{equation*}
converges in $\LL$. This implies that for $\Im \lambda \gg 1$
\begin{equation}\label{eq:2s}
D_V(\lambda)) = \exp \left( -\sum_{m=p}^\infty (-1)^m\dfrac{\trace\left(K_V^m\right)}{m} \right).
\end{equation}
We now explain how to obtain an expansion of $\trace(K_V^m)$ for $\Im \lambda \gg 1$. Split $\trace(K_V^m)$ in constructive and destructive parts:
\begin{equations*}
\trace(K_V^m)  = \sum_{\substack{k_1, ..., k_m, \\ k_1+...+k_m=0}} \trace\left( \prod_{\ell=1}^m K_{W_{k_\ell}} e^{ik_\ell \bullet/\epsi} \right) + \sum_{\substack{k_1, ..., k_m, \\ k_1+...+k_m \neq 0}} \trace\left( \prod_{\ell=1}^m K_{W_{k_\ell}} e^{ik_\ell \bullet/\epsi} \right).
\end{equations*}
By Lemma \ref{lem:3} destructive sequences induce negligible contributions: 
\begin{equation}\label{eq:1n}
\trace(K_V^m) = \sum_{\substack{k_1, ..., k_m, \\ k_1+...+k_m=0}} \trace \left( \prod_{\ell=1}^m K_{W_{k_\ell}} e^{ik_\ell \bullet/\epsi} \right)+ O(\epsi^4).
\end{equation}
We next reproduce the proofs of Lemma \ref{lem:41} and \ref{lem:6}. Fix $k_1, ..., k_m$ with $k_1+...+k_m=0$ and define $\sigma_j=k_j+...+k_\nu$. 
\begin{equation*}
\prod_{j=1}^m K_{W_{k_j}} e^{ik_j \bullet/\epsi} = \prod_{j=1}^m \rho R(D+\sigma_j/\epsi) W_{k_j},
\end{equation*}
where $R(\xi,\lambda) = (\xi^2-\lambda^2)^{-1}$. The expansion of $R(\xi+\sigma/\epsi,\lambda)$ given in Appendix \ref{app:2} induces
\begin{equation*}
R(D+\sigma/\epsi) = \epsi^2\dfrac{\Id}{|\sigma|^2} - \epsi^3\dfrac{2 \sigma \cdot D}{|\sigma|^4}+  O_{\BB(H^{s+5}, H^s)}(\epsi^4).
\end{equation*}
Thus if two or more of the $\sigma_j$ are non-zero then 
\begin{equation*}
\trace\left( \prod_{j=1}^m K_{W_{k_j}} e^{ik_j \bullet/\epsi}  \right) = O(\epsi^4).
\end{equation*}
On the other hand the only constructive sequences with at most one non-vanishing $\sigma_j$ are the cyclic permutations of $(0, ..., 0, -k,k)$. For every $k \neq 0$ there are $m$ such sequences. It yields
\begin{align*}
 & \trace(K_V^m) = \trace(K_{W_0}^m) + m \sum_{k \neq 0} \trace\left(K_{W_0}^{m-2} K_{W_{-k}} R(D+k/\epsi) W_{k}\right) + O(\epsi^4) \\
   = & \trace(K_{W_0}^m) + m \sum_{k \neq 0} \dfrac{\epsi^2}{|k|^2} \trace\left(K_{W_0}^{m-2} K_{W_{-k}} W_k\right) - 2m \sum_{k \neq 0} \dfrac{\epsi^3}{|k|^4}\trace\left( K_{W_0}^{m-2} K_{W_{-k}}  (k \cdot D) W_k  \right) + O(\epsi^4).
\end{align*}
Note moreover that as operators,
\begin{align*}
2 \sum_{k \neq 0} \dfrac{W_k  (k \cdot D)  W_{-k}}{|k|^4} = 2\sum_{k \neq 0} \dfrac{ W_{-k} ((k \cdot D)W_k)}{|k|^4}.
\end{align*}
Thus the LHS is actually a multiplication operator and thus can be seen as a potential. This leads to
\begin{gather*}
\trace(K_V^m) = \trace(K_{W_0}^m) + m \epsi^2 \trace\left(K_{W_0}^{m-2} K_{\Lambda_0}\right) + m \epsi^3\trace\left( K_{W_0}^{m-2} K_{\Lambda_1} \right) + O(\epsi^4), \\
\Lambda_0 = \sum_{k \neq 0} \dfrac{W_{-k}W_k}{|k|^2}, \ \ \ \ \ \Lambda_1 = -2\sum_{k \neq 0} \dfrac{ W_{-k}((k\cdot D)W_k)}{|k|^4}.
\end{gather*}
Coming back to \eqref{eq:2s},
\begin{align*}
D_V(\lambda) & = \exp\left(-\sum_{m=p}^\infty (-1)^m\dfrac{\trace\left(K_V\right)}{m}\right) \\
   & = \exp\left(-\sum_{m=p}^\infty (-1)^m\dfrac{\trace\left(K_{W_0}^m\right)}{m} - \sum_{m=p-2}^\infty (-1)^m\trace\left(K_{W_0}^m K_{\epsi^2\Lambda_0 + \epsi^3 \Lambda_1}\right) + O(\epsi^4)\right) \\
   & = D_{W_0}(\lambda) \left(1 - \trace\left((\Id + K_{W_0})^{-1} (-K_{W_0})^{p-2} K_{\epsi^2\Lambda_0 + \epsi^3\Lambda_1} \right) \right) + O(\epsi^4).
\end{align*}
This gives the value of $a_0, a_1, a_2, a_3$. It is in practice possible to use this method to compute all the other coefficients $a_4, ..., a_{N-1}$ given by Theorem \ref{thm:5}. 

\subsection{The case $\lambda_0 = 0$ in dimension one}\label{subsec:4} In this part we prove Lemma \ref{lem:2s}. Thus we assume $d=1$. For $\lambda \neq 0$ the operator $K_\VV$ is trace class. This allows us to define $d_\VV(\lambda) = \Det(\Id + K_\VV)$. By \cite[Theorem 2.6]{DyaZwo} the function $\lambda \mapsto \lambda d_\VV(\lambda)$ is entire. It is related to the modified Fredholm determinant $D_\VV$  by the identity
\begin{equation}\label{eq:3l}
\lambda \exp\left(-\sum_{m=1}^{p-1} (-1)^m \dfrac{\trace(K_\VV^m)}{m} \right)D_\VV(\lambda) = \lambda d_\VV(\lambda).
\end{equation}
If $\varphi$ is a meromorphic function with a pole at $0$ we write $\varphi = \sum_{m \in \Z} \beta_m z^m$ and we define $\sing(\varphi)$ the meromorphic function $\sing(\varphi)(z) = \sum_{m < 0} \beta_m z^m$. We recall that $\Lambda$ is the potential given by
\begin{equation}\label{eq:2q}
\Lambda = \epsi^2 \Lambda_0 + \epsi^3 \Lambda_1 =  \epsi^2 \sum_{k \neq 0} \dfrac{W_k W_{-k}}{k^2} - 2 \epsi^3 \sum_{k \neq 0} \dfrac{W_k (DW_{-k})}{k^3}.
\end{equation}

\begin{lem}\label{lem:1g} Let $d=1$ and $N \geq 4$. For every $m \geq 2$ there exists a holomorphic function $t_m : \C \setminus \{0\} \rightarrow \C$ with the following:
\begin{enumerate}
\item[$(i)$] $\sing(t_m)=\sing(\trace(K_V^m))$.
\item[$(ii)$] Locally  uniformly on $\C \setminus \{0\}$,
\begin{equation*}
t_m(\lambda) = \trace(K_{W_0}^m) + m \trace(K_{W_0}^{m-2} K_\Lambda) + ... + O(\epsi^N).
\end{equation*}
\end{enumerate}
\end{lem}

\begin{proof}[Proof of Lemma \ref{lem:2s} assuming Lemma \ref{lem:1g}] Let $p=4N(N+1)$ and set
\begin{equation*}
h_V(\lambda) = \lambda \exp\left(-\sum_{m=1}^{p-1} (-1)^m \dfrac{t_m(\lambda)}{m} \right)D_V(\lambda)
\end{equation*}
where $D_\VV(\lambda)$ is the determinant defined in \eqref{eq:1d}. Equation \eqref{eq:3l} implies that 
\begin{equation*}
\lambda d_V(\lambda) = h_V(\lambda) \exp\left(\sum_{m=1}^{p-1} (-1)^m \dfrac{t_m(\lambda)-\trace(K_V^m)}{m} \right).
\end{equation*}
The function 
\begin{equation*}
\sum_{m=1}^{p-1} (-1)^m \dfrac{t_m(\lambda)-\trace(K_V^m)}{m}
\end{equation*}
is entire thanks to point $(i)$ of Lemma \ref{lem:1g}. Consequently resonances of $V$ (counted with multiplicity) are exactly zeros of $h_V$ (counted with multiplicity).

We next show that the function $h_V$ has an expansion in powers of $\epsi$ on all of $\C$. For that we use Lemma \ref{lem:1i} with $S_0 = \{0\}$, $E=\C$,
\begin{equation*}
f(\lambda,\epsi) = \lambda \exp\left(-\sum_{m=1}^{p-1} (-1)^m \dfrac{t_m(\lambda)}{m} \right), \ \  \ \ g(\lambda,\epsi) = D_V(\lambda).
\end{equation*}
Both $f,g$ are meromorphic on $\C$ and their only pole is at $0$. They both admit an expansion away from $\{0\}$ by Lemma \ref{lem:1g} for $f$ and by Theorem \ref{thm:5} for $g$. Their product $h_V=fg$ is entire. Consequently $h_V$ admits an expansion of the form
\begin{equation*}
h_V(\lambda) = h_0(\lambda) + \epsi h_1(\lambda) + ... + \epsi^{N-1} h_{N-1}(\lambda) + O(\epsi^N)
\end{equation*}
that holds locally uniformly for $\lambda$  in $\C$. We next compute the first few terms in this expansion. Because of $(ii)$ in Lemma \ref{lem:1g} and of Theorem \ref{thm:5} we have
\begin{align*}
h_V(\lambda) & = \lambda \exp\left(-\sum_{m=1}^{p-1} (-1)^m \dfrac{t_m(\lambda)}{m} \right)D_V(\lambda) \\
    & = \lambda d_{W_0}(\lambda) \exp \left( - \sum_{m=0}^{p-3} (-1)^m\trace(K_{W_0}^{m} K_\Lambda)  \right) \left( 1 - \trace((\Id + K_{W_0})^{-1} K_{W_0}^{p-2} K_\Lambda)\right) + O(\epsi^4) \\
    & =  \lambda d_{W_0}(\lambda) \left( 1 - \trace((\Id + K_{W_0})^{-1} K_\Lambda)\right) + O(\epsi^4).
\end{align*}
This ends the proof of Lemma \ref{lem:2s}.\end{proof}

We next prove Lemma \ref{lem:1g}. We start with a preliminary lemma:

\begin{lem}\label{lem:1c} Let $k \in \Z \setminus \{0\}$ and $\varphi : \R \rightarrow \C$ be a smooth compactly supported function. Let $p_N$ be the polynomial defined by
\begin{equation*}
p_N(X) = -2\left(1+2X+3X^2+...+ (N+1) X^N\right).
\end{equation*}
Then for every $N \geq 2$,
\begin{equation}\label{eq:2p}
\left|\int_\R \varphi(x) e^{ikx/\epsi} |x| dx  - \epsi^2 \left(p_{N-3}(\epsi D/k) \varphi\right)(0)\right| \leq  C N \left(\dfrac{\epsi}{k}\right)^N \|\varphi\|_{N+1}.
\end{equation}
where the constant $C$ depends only on the support of $\varphi$.
\end{lem}

\begin{proof} By rescaling $\epsi$ to $\epsi/k$ we see that it suffices to prove the lemma in the case $k=1$. Define 
\begin{equation*}
I[\varphi] = \int_\R e^{ix/\epsi} \varphi(x) |x| dx, \ \ \ \  J[\varphi] = \dfrac{1}{i}\int_\R e^{ix/\epsi} \varphi(x) \sgn(x) dx.
\end{equation*}
By an integration by parts
\begin{equation*}\begin{gathered}
I[\varphi] = -\epsi \left( J[\varphi] + I[D\varphi]\right), \\
J[\varphi] = \epsi \left( 2\varphi(0) - J[D\varphi]\right).\end{gathered}
\end{equation*}
Consequently:
\begin{equation}\label{eq:2n}
I[\varphi] = \epsi^2 \left( -2 \varphi(0) + 2J[D\varphi] + I[D^2\varphi]\right).
\end{equation}
We prove by recursion: for every $n \geq 0$
\begin{equation}\label{eq:2o}
I[\varphi] = \epsi^2 (p_n(-\epsi D) \varphi)(0) + \epsi^{n+2} I[(-D)^{n+2} \varphi] - \epsi^{n+2} (n+2)J[(-D)^{n+1} \varphi]
\end{equation}
where $p_n=-2(1+2X+3X^2+...+(n+1) X^n)$. For $n=2$ this holds by equation \eqref{eq:2n}. Now assume \eqref{eq:2o} holds for some $n$. Then
\begin{align*}
I[\varphi] = & \epsi^2 [p_n(-\epsi D) \varphi](0) + \epsi^{n+3} \left( -J[(-D)^{n+2}\varphi] + I[(-D)^{n+3}\varphi]\right)  \\  &  - \epsi^{n+3} (n+2) \left( 2[(-D)^{n+1}\varphi](0) + J[(-D)^{n+2}\varphi]\right) \\
            = & \epsi^2 [p_{n+1}(-\epsi D) \varphi](0) + \epsi^{n+3} I[(-D)^{n+3}\varphi] - \epsi^{n+3} (n+3) J[(-D)^{n+2}\varphi]
\end{align*}
where $p_{n+1} = p_n - 2 (n+2) x^{n+1}$. This ends the recursion. Equation \eqref{eq:2p} follows from \eqref{eq:2o} and the estimate $|I[D^N \varphi]| \leq C \epsi \|\varphi \|_{N+1}$, $|J[D^N\varphi]| \leq C \epsi \|\varphi\|_{N+1}$.\end{proof}

\begin{proof}[Proof of Lemma \ref{lem:1g}] In dimension one the kernel of the free resolvent $R_0(\lambda)$ is given by $R_0(\lambda,x,y) = i e^{i\lambda|x-y|}/(2\lambda)$. We decompose it as follows:
\begin{equation*}
\begin{gathered}
R_0(\lambda,x,y)  = \dfrac{f_0(\lambda,x-y)}{\lambda} + f_1(\lambda,x-y) |x-y|, \\
f_0(\lambda,x-y) = \dfrac{i}{2} \cos(\lambda|x-y|), \ \ \ f_1(\lambda,x-y) = -\dfrac{\sin(\lambda|x-y|)}{2\lambda|x-y|}. \end{gathered}
\end{equation*}
The functions $f_0$ and $f_1$ are both smooth on $\C \times \R$. This induces a decomposition of $K_\VV(\lambda)$ given by
\begin{equation*}\begin{gathered}
K_\VV(\lambda) = E_{\VV,0}(\lambda) + E_{\VV,1}(\lambda), \\ 
E_{\VV,\epsilon}(\lambda,x,y) = \rho(x)\dfrac{f_0(\lambda,x-y)}{\lambda} \VV(y), \ \ \ E_1(\lambda,x,y) = \rho(x) f_1(\lambda,x-y) |x-y|\VV(y).\end{gathered}
\end{equation*}
Thus $K_\VV(\lambda)$ is the sum of a smoothing operator $E_{\VV,0}(\lambda)$ with a pole at $\lambda = 0$ and of an operator $E_{\VV,1}(\lambda)$ which is not smoothing but has no pole. We now define
\begin{equation}\label{eq:0q}
t_m(\lambda) =  \systeme{\trace(K_V^m) - \trace(E_{V,1}^m) + \trace(E_{W_0,1}^m)  + m\trace(E_{W_0,1}^{m-2} E_{\Lambda,1}) \ \  \text{ if } m \geq 3 \\ \trace(K_V^2) + \trace(E_{W_0,1}^2) - \trace(E_{V,1}^2) \ \ \ \ \ \ \ \ \ \ \ \ \ \ \ \ \ \ \ \ \ \ \ \ \ \ \   \text{ if } m=2}
\end{equation}
where $\Lambda$ is the potential given by \eqref{eq:2q}. Since $\trace(E_{W_0,1}^m)- \trace(E_{V,1}^m) + m\trace(E_{W_0,1}^{m-2} E_{\Lambda,1})$ and $\trace(E_{W_0,1}^2) - \trace(E_{V,1}^2)$ are both entire function of $\lambda$ we have $\sing(t_m) = \sing(\trace(K_V^m))$. It remains to show that the function $t_m$ satisfies the expansion given by $(ii)$. Write
\begin{equations*}
\trace(K_V^m)  = \sum_{\epsilon_1, ..., \epsilon_m \in \{0,1\}^m} \trace\left( \prod_{j=1}^m E_{V,\epsilon_j} \right) \\
    = \trace\left(E_{V,1}^m\right) +  \sum_{k_1, ..., k_m} \ \ \  \sum_{\substack{\epsilon_1 ..., \epsilon_m \in \{0,1\}^m \\ \epsilon_1 \cdot ... \cdot \epsilon_m = 0 }} \trace\left( \prod_{j=1}^{m} E_{W_{k_j},\epsilon_j} e^{ik_j\bullet/\epsi} \right).
\end{equations*}
We first claim that for every $N$ and locally uniformly on $\C \setminus \{0\}$
\begin{equation}\label{eq:4a}
\sum_{k_1 + ...+ k_m \neq 0} \ \ \  \sum_{\substack{\epsilon_1 ..., \epsilon_m \in \{0,1\}^m \\ \epsilon_1 \cdot ... \cdot \epsilon_m = 0 }} \trace\left( \prod_{j=1}^{m} E_{W_{k_j},\epsilon_j} e^{ik_j\bullet/\epsi} \right) \leq C \epsi^N \|W\|_{Z^N}^m.
\end{equation}
Fix a sequence $\epsilon_1, ..., \epsilon_m \in \{0,1\}^m$ with $\epsilon_1 \cdot ... \cdot \epsilon_m =0$ and $k_1, ..., k_m \in \Z$ with $k_1+...+k_m \neq 0$. There exists $j_0$ with $\epsilon_{j_0} = 0$. Using the cyclicity of the trace we can assume without loss of generality that $j_0=1$. Let $n=m-\epsilon_1-...-\epsilon_m$. Using the explicit expression of the kernel of the operators $E_{\VV,\epsilon}$ we have
\begin{equation*}
\trace\left( \prod_{j=1}^{m} E_{{W_{k_j}},\epsilon_j} e^{ik_j\bullet/\epsi} \right) 
= \lambda^{-n} \int_{\R^m} \left(\prod_{j=1}^m f_{\epsilon_j}(x_j-x_{j-1}) |x_j-x_{j-1}|^{\epsilon_j} W_{k_j}(x_j)e^{ik_j x_j/\epsi} dx_j\right) dx_1
\end{equation*}
where by convention $x_0=x_m$. The substitution $x_j = y_1+...+y_j$, $j \in [1,m]$ and the explicit expression of the kernels of $E_{\VV,0}$ and $E_{\VV,1}$ yield
\begin{equation*}\begin{gathered}
\trace\left( \left(\prod_{j=1}^{m} E_{W_{k_j},\epsilon_j} e^{ik_j\bullet/\epsi}\right) \right) 
= \lambda^{-n} \int_{\R} e^{i\sigma_1 y_1/\epsi} I(y_1) dy_1
\end{gathered}
\end{equation*}
where $\sigma_j = k_j+...+k_m$, $z=y_2+...+y_{m-1}$ and
\begin{equation*}
I(y_1) = W_{k_1}(y_1) \int_{\R^{m-1}}f_0(z+y_m) \prod_{j=2}^m f_{\epsilon_j}(y_j) |y_j|^{\epsilon_j} W_{k_j}(y_1+...+y_j)e^{i\sigma_j y_j/\epsi} dy_j.
\end{equation*} 
The function $y_1 \mapsto I(y_1)$ is smooth and compactly supported. Since $\sigma_1 \neq 0$ $N$ integrations by parts give the estimate
\begin{equation*}
\left|\int_{\R} e^{i\sigma_1 y_1/\epsi} I(y_1) dx_1\right| \leq C\epsi^N\|I\|_N \leq C \epsi^N \prod_{j=1}^N \|W_{k_j}\|_N.
\end{equation*}
Therefore
\begin{equation*}\begin{gathered}
\left|\sum_{k_1 + ...+ k_m \neq 0} \ \ \  \sum_{\substack{\epsilon_1 ..., \epsilon_m \in \{0,1\}^m \\ \epsilon_1 \cdot ... \cdot \epsilon_m = 0 }} \trace\left( \prod_{j=1}^{m} E_{W_{k_j},\epsilon_j} e^{ik_j\bullet/\epsi} \right)\right| \leq C \epsi^N \sum_{k_1, ..., k_m}\prod_{j=1}^N \|W_{k_j}\|_N  \leq C \epsi^N \|W\|_{Z^N}^m.\end{gathered}
\end{equation*}
This proves \eqref{eq:4a}.

We next show that the function
\begin{equation*}
\sum_{k_1 + ...+ k_m = 0} \ \ \  \sum_{\substack{\epsilon_1 ..., \epsilon_m \in \{0,1\}^m \\ \epsilon_1 \cdot ... \cdot \epsilon_m = 0 }} \trace\left( \prod_{j=1}^{m} E_{W_{k_j},\epsilon_j} e^{ik_j\bullet/\epsi} \right).
\end{equation*}
admits an expansion in powers of $\epsi$. It suffices to prove that for any fixed sequence $\{\epsilon_j\}$ with $\epsilon_1=0$ the function
\begin{equation}\label{eq:0x}
\sum_{k_1 + ...+ k_m = 0} \trace\left( \prod_{j=1}^{m} E_{W_{k_j},\epsilon_j} e^{ik_j\bullet/\epsi} \right)
\end{equation}
admits an expansion in powers of $\epsi$. Fix $k_1, ..., k_m$ with $k_1+...+k_m=0$. We define $F_{m-1}$ and $F_s, s \in [1,m-1]$ recursively as follows:
\begin{equation*}
\systeme{F_{m-1}(y_1,...,y_{m-1}) =  \int_{\R} f_0(z+y_m)  f_{\epsilon_m}(y_m) W_{k_m}(y_1+...+y_m) e^{i\sigma_my_m/\epsi} |y_m|^{\epsilon_m} dy_m \\ 
F_{s-1}(y_1,...,y_{s-1}) = \int_\R f_{\epsilon_s}(y_s) W_{k_s}(y_1+...+y_s) F_s(y_1, ..., y_s) e^{i\sigma_s y_s/\epsi} |y_s|^{\epsilon_s} dy_s}
\end{equation*}
where $z=y_2+...+y_{m-1}$. Let $F_0(\lambda)$ be given by
\begin{equation*}
F_0(\lambda) = \trace\left( \prod_{j=1}^{m} E_{W_{k_j},\epsilon_j} e^{ik_j\bullet/\epsi}\right) 
= \lambda^{-n} \int_{\R} W_{k_1}(y_1) F_1(y_1) dy_1.
\end{equation*}
We prove recursively that $F_{m-1}, F_{m-2}, ..., F_1, F_0$ admit an expansion in powers of $\epsi$. The fact that $F_{m-1}$ admits an expansion in powers of $\epsi$ is a consequence of Lemma \ref{lem:1c}. The coefficients are smooth functions of $y_1, ..., y_{m-1}$. The recursive formula defining $F_{m-2}$ shows that $F_{m-2}$ also admits an expansion in powers of $\epsi$ whose coefficients are smooth functions of $y_1, ..., y_{m-2}$. The same recursive scheme shows that $F_{m-3}, ..., F_0$ admit an expansion in powers of $\epsi$. The sum over $k_1, ..., k_m$ with $k_1+...+k_m=0$ of the coefficients converge (we skip the details) and we conclude that \eqref{eq:0x} admits an expansion in powers of $\epsi$. Finally we sum over all sequences $\{\epsilon_j\}$ with at least one vanishing term and we use \eqref{eq:4a} to deduce that
\begin{align*}
t_m(\lambda) = & \trace(K_V^m) + \trace(E_{W_0,1}^m) - \trace(E_{V,1}^m) + \delta_{m \neq 2} m\trace(E_{W_0,1}^{m-2} E_{\Lambda,1}) \\
   = & \sum_{\substack{\epsilon_1, ..., \epsilon_m \in \{0,1\}^m, \\ \epsilon_1\cdot ... \cdot \epsilon_m=0}} \trace\left( \prod_{j=1}^m E_{V,\epsilon_j} \right) + \delta_{m \neq 2} m  \trace(E_{W_0,1}^{m-2} E_{\Lambda,1}) + \trace(E_{W_0,1}^m) 
\end{align*}
admits an expansion in powers of $\epsi$.

To end the proof we must compute the first terms terms in the expansion of the function $t_m$. We fix $N=4$ and work modulo $O(\epsi^4)$. The only sequence $\{k_j\}$ that can generate non-negligible terms is $(0, ..., 0, -k, k)$ up to cyclic permutation -- see \S \ref{subsec:3}. We fix $\{\epsilon_j\}$ and we estimate
\begin{equation*}
\sum_{k \neq 0} \trace\left(\left(\prod_{j=1}^{m-2} E_{W_0,\epsilon_j}\right) E_{W_{-k},\epsilon_{m-1}} e^{-ik\bullet/\epsi} E_{W_k,\epsilon_m} e^{ik\bullet/\epsi}  \right).
\end{equation*}
Assume that $m \geq 3$ and define $G$ by
\begin{align*}
G(\lambda,y_1,...,y_{m-1}) =  \sum_{k \neq 0} W_{-k}(y_1+z) \int_{\R} f_0(z+y_m) f_{\epsilon_m}(y_m) W_k(y_1+z+y_m) e^{ik y_m/\epsi} |y_m|^{\epsilon_m} dy_m
\end{align*}
where we recall that $z=y_2+...+y_{m-1}$. We first deal with the case $\epsilon_m=1$. This implies $f_{\epsilon_m}(0) = f_1(0)= -1/2$. Apply Lemma \ref{lem:1c} to obtain the asymptotic
\begin{equations*}
G(\lambda,y_1,...,y_{m-1}) =  \sum_{k \neq 0} \left(\dfrac{\epsi}{k}\right)^2 f_0(z) W_{-k}(y_1+z)W_k(y_1+z) \\  - 2 \sum_{k \neq 0} \left(\dfrac{\epsi}{k}\right)^3 f_0(z) W_{-k}(y_1+z)(DW_k)(y_1+z)  - 2\sum_{k \neq 0} \left(\dfrac{\epsi}{k}\right)^3 (Df_0)(z) W_{-k} (y_1+z)W_k(y_1+z)  + O(\epsi^4).
\end{equations*}
Since $\sum_{k \neq 0} {W_{-k} W_k}/{k^3} = 0$ we can remove the last term that appears in the expansion of $G$ and write $G(\lambda,y_1,...,y_{m-1}) = f_0(z) \Lambda(y_1+z) + O(\epsi^4)$.  This expansion combined with the inverse substitution $y \mapsto x$ variables yields
\begin{equations*}
 \sum_{k \neq 0} \trace\left(\left(\prod_{j=1}^{m-2} E_{W_0,\epsilon_j}\right) E_{W_{-k},\epsilon_{m-1}} e^{-ik\bullet/\epsi} E_{W_k,1} e^{ik\bullet/\epsi}\right) + O(\epsi^4) \\ 
=   \sum_{k \neq 0} \lambda^{-n}\int_{\R^{m-1}} f_0(z) W_0(y_1) \left(\prod_{j=2}^{m-2} f_{\epsilon_j}(y_j) |y_j|^{\epsilon_j} W_0(y_1+...+y_j) dy_j\right)  f_{\epsilon_{m-1}}(y_{m-1})\Lambda(z)dy_1 dy_{m-1} \\
    =  \sum_{k \neq 0} \lambda^{-n} \int_{\R^{m-1}}  \left(\prod_{j=1}^{m-2} f_{\epsilon_j}(y_j-y_{j-1}) |x_j-x_{j-1}|^{\epsilon_j} W_0(x_j) dx_j\right) f_{\epsilon_{m-1}}(x_{m-1}-x_{m-2})\Lambda(x_{m-1}) dx_1 dx_{m-1}  \\
    =  \trace\left(\left(\prod_{j=1}^{m-2} E_{W_0,\epsilon_j}\right) E_{\Lambda,\epsilon_{m-1}}\right).
\end{equations*}
This gives an estimate of $G$ in the case $\epsilon_m=1$. In the case $\epsilon_m=0$ the kernel of $E_{W_k,0}$ is smooth and we can integrate by parts by parts to obtain $G(\lambda,y_1, ..., y_{m-1}) = O(\epsi^4)$. Summing these estimates of $G$ over all possible values of $\epsilon_1, ..., \epsilon_{m-1}, \epsilon_m$ and using the cyclicity of the trace yield
\begin{equations*}
    \sum_{\substack{\epsilon_1, ..., \epsilon_m \in \{0,1\}^m \\ \epsilon_1 \cdot ... \cdot \epsilon_m = 0}} \trace\left( \prod_{j=1}^{m} E_{V,\epsilon_j} \right) - \sum_{\substack{\epsilon_1, ..., \epsilon_m \in \{0,1\}^m \\ \epsilon_1 \cdot ... \cdot \epsilon_m = 0}} \trace\left( \prod_{j=1}^{m} E_{W_0,\epsilon_j}\right) \\  
   =   m  \sum_{\substack{\epsilon_1, ..., \epsilon_m \in \{0,1\}^m \\ \epsilon_1 \cdot ... \cdot \epsilon_m = 0}} \sum_{k \neq 0} \trace\left( \left(\prod_{j=1}^{m-2} E_{W_0,\epsilon_j}\right) E_{W_{-k},\epsilon_{m-1}} e^{-ik\bullet/\epsi} E_{W_k,\epsilon_{m}} e^{ik\bullet/\epsi}\right)+ O(\epsi^4) \\
   =  m \ \sum_{\substack{\epsilon_1, ..., \epsilon_m \in \{0,1\}^m \\ \epsilon_1 \cdot ... \cdot \epsilon_m = 0, \ \epsilon_m=1}} \trace\left( \left(\prod_{j=1}^{m-2} E_{W_0,\epsilon_j}\right) E_{\Lambda,\epsilon_{m-1}} \right) + O(\epsi^4).
\end{equations*}
Recall that $t_m(\lambda)$ is given by \eqref{eq:0q} to conclude that
\begin{align*}
 t_m(\lambda) = & \sum_{\substack{\epsilon_1, ..., \epsilon_m \in \{0,1\}^m \\ \epsilon_1 \cdot ... \cdot \epsilon_m = 0}} \trace\left( \prod_{j=1}^{m} E_{W_0,\epsilon_j}\right) + \trace(E_{W_0,1}^m) + \sum_{\substack{\epsilon_1, ..., \epsilon_m \in \{0,1\}^m \\ \epsilon_1 \cdot ... \cdot \epsilon_m = 0}} \trace\left( \prod_{j=1}^{m} E_{V,\epsilon_j} \right) + m \trace(E_{W_0,1}^{m-2} E_{\Lambda,1}) \\
= & \trace(K_{W_0}^m)+ m \ \sum_{\substack{\epsilon_1, ..., \epsilon_m \in \{0,1\}^m \\ \epsilon_1 \cdot ... \cdot \epsilon_m = 0, \ \epsilon_m=1}} \trace\left( \left(\prod_{j=1}^{m-2} E_{W_0,\epsilon_j}\right) E_{\Lambda,\epsilon_{m-1}} \right) + O(\epsi^4) \\
= &  \trace(K_{W_0}^m)+ m \trace\left( K_{W_0}^{m-2}  K_{\Lambda} \right) + O(\epsi^4).
\end{align*}
We finally deal with the case $m=2$. If $\epsilon_1, \epsilon_2 \in \{0,1\}$ then
\begin{equation*}\begin{gathered}
\trace\left(E_{W_{-k},\epsilon_{m-1}} e^{-ik\bullet/\epsi} E_{W_k,\epsilon_m} e^{ik\bullet/\epsi}\right) \\  = \lambda^{\epsilon_1+\epsilon_2-2} \int_{\R} f_{\epsilon_1}(y_2) f_{\epsilon_2}(y_2) W_{-k}(y_1) W_k(y_1+y_2) |y_2|^{\epsilon_1+\epsilon_2} e^{iky_2/\epsi} dy_1 dy_2.\end{gathered}
\end{equation*}
If $\epsilon_1+\epsilon_2$ is even then one can integrate by parts many times in $y_2$ and obtain $O(\epsi^4)$. Otherwise $\epsilon_1+\epsilon_2 = 1$ and $f_{\epsilon_1} f_{\epsilon_2} = f_0f_1$. In particular $f_0f_1(0) = 1/(4i)$ and $(f_0f_1)'(0) = 0$. This yields
\begin{equations*}
   \sum_{\epsilon_1, \epsilon_2, \ \epsilon_1\epsilon_2=0} \sum_{k \neq 0} \trace(E_{W_{-k},\epsilon_{m-1}} e^{-ik\bullet/\epsi} E_{W_k,\epsilon_m} e^{ik\bullet/\epsi})  \\ =  2\lambda^{-1}  \sum_{k \neq 0} \int_{\R} W_{-k}(y_1) \int_\R f_0(y_2) f_1(y_2) W_k(y_1+y_2) |y_2| e^{iky_2/\epsi} dy_2 dy_1 \\ 
   =  2 \sum_{k \neq 0} \dfrac{i}{2\lambda}\int_\R W_{-k}(y_1) \left( \left( \dfrac{\epsi}{k} \right)^2  W_k(y_1) - 2\left( \dfrac{\epsi}{k} \right)^3  DW_k(y_1) \right) dy_1 + O(\epsi^4) =  2\trace(K_\Lambda) + O(\epsi^4).
\end{equations*}
Together with \eqref{eq:4a} this gives $t_2(\lambda) = \trace(K_{W_0}^2) + 2 \trace(K_\Lambda) + O(\epsi^4)$. This completes the proof of the lemma. \end{proof}

\section{Appendices}
\subsection{Analytic continuation of some Fredholm operators}\label{app:1} Let $T(\lambda)$ be a holomorphic family of trace-class operators on a Hilbert space. In finite dimension, the operator $\det(\Id + T(\lambda)) (\Id + T(\lambda))^{-1}$, defined away from the poles of $(\Id+T(\lambda))^{-1}$, extends to an entire family of operators known as the comatrix of $\Id + T(\lambda)$. In infinite dimension a similar statement holds:

\begin{lem} Consider $\HH$ a Hilbert space, $\UU$ an open connected subset of $\C$ and $T(\lambda)$ a holomorphic family of trace class operators for $\lambda \in \UU$. Assume that $\Id + T(\lambda_0)$ is invertible for some $\lambda_0 \in \UU$. Then the family of operators
\begin{equation*}
\TT(\lambda) = \Det(\Id + T(\lambda))(\Id + T(\lambda))^{-1}
\end{equation*}
initially defined for $\lambda$ away from the poles of $(\Id + T(\lambda))^{-1}$ extends to a holomorphic family of operators on $\UU$. Moreover, 
\begin{equation}\label{eq:4i}
|\TT(\lambda)|_{\BB(\HH)} \leq \Det\left( \Id + \left(T(\lambda)^* T(\lambda)\right)^{1/2} \right) \leq e^{2|T(\lambda)|_{\LL}}.
\end{equation}
\end{lem}

\begin{proof} The proof uses the Gohberg-Sigal theory of residues -- see \cite[Appendix C.4]{DyaZwo}. By analytic Fredholm theory, $(\Id + T(\lambda))^{-1}$ defines a meromorphic family of operators with poles of finite rank. Fix $\mu \in \UU$ a pole of $(\Id + T(\lambda))^{-1}$ and $\lambda$ in a punctured neighborhood of $\mu$. We can write
\begin{equation*}
\Id + T(\lambda) = U_1(\lambda) \left( P_0 + \sum_{m=1}^N (\lambda - \mu)^{\kappa_m} P_m \right) U_2(\lambda)
\end{equation*}
where $U_1(\lambda), U_2(\lambda)$ are holomorphic families of invertible operators, $\kappa_m \geq 1$, $P_m$ has rank $1$ for $m > 0$, $P_m P_{m'} = \delta_{mm'}P_m$, $\text{rank} (\Id - P_0)  < \infty$. Therefore
\begin{equation*}
(\Id + T(\lambda))^{-1} = U_2(\lambda)^{-1} \left( P_0 + \sum_{m=1}^N (\lambda - \mu)^{-\kappa_m} P_m \right) U_1(\lambda)^{-1}.
\end{equation*}
The holomorphic function $\lambda \mapsto \Det(\Id + T(\lambda))$ has a zero at $\mu$, of multiplicity $\sum_{m=1}^N \kappa_m$
-- see \cite[equation $(C.4.7)$]{DyaZwo}. It follows that the operator $\TT(\lambda)$ can indeed be analytically continued at $\lambda = \mu$ with
\begin{equation*}
\TT(\mu) = \systeme{ 0 & \text{ if } N > 1 \\ \left.\dfrac{\Det(\Id + T(\lambda))}{(\lambda-\mu)^{\kappa_1}}\right|_{\lambda=\mu} U_2(\mu)^{-1} P_1 U_1(\mu)^{-1} & \text{ if } N=1. }
\end{equation*}
The first bound in \eqref{eq:4i} follows from \cite{DyaZwo}, (B.4.7). For the second one, note first that
\begin{equation*}
\Det\left( \Id + \left(T(\lambda)^* T(\lambda)\right)^{1/2} \right) \leq \exp\left(\left|\left(T(\lambda)^* T(\lambda)\right)^{1/2}\right|_{\LL}\right).
\end{equation*}
Finally we note that
\begin{equations*}
s_{2j}\left( \left(T(\lambda)^* T(\lambda)\right)^{1/2} \right) \leq s_j\left(T(\lambda)^*\right)^{1/2}s_j\left(T(\lambda)\right)^{1/2} \leq s_j\left(T(\lambda)\right),  \\
\left|\left(T(\lambda)^* T(\lambda)\right)^{1/2} \right|_{\LL} \leq 2\sum_{j=0}^\infty s_{2j}\left( \left(T(\lambda)^* T(\lambda)\right)^{1/2} \right) \leq 2\sum_{j=0}^\infty s_j\left(T(\lambda)\right) \leq 2 |T(\lambda)|_{\LL}.
\end{equations*}
This concludes the proof. \end{proof}

\subsection{Expansion of $R(\lambda,\xi+k/\epsi)$}\label{app:2} We study here the Taylor development of rational functions of the form $F(\epsi) = (1+a\epsi +b\epsi^2)^{-1}$. Such functions are analytic for small values of $\epsi$ and therefore there exists $u_k \in \C$ with $F(\epsi) = \sum_{j \geq 0} u_j \epsi^j$. Since $F(\epsi) (1+a\epsi +b\epsi^2) = 1$, the $u_k$ must satisfy the recursion relation
\begin{equation*}
\systeme{u_0= 1, \\ u_1=-a, \\ u_j=-au_{j-1} -bu_{j-2}.}
\end{equation*}
For $\epsi$ small enough the Taylor development of $F$ takes the form
\begin{equation*}
F(\epsi) = \sum_{j=0}^{J-1} u_j \epsi^j + r_J(\epsi), \ \ r_J(\epsi) = \sum_{j = J}^\infty u_j \epsi^j.
\end{equation*}
We have moreover
\begin{align*}
(1+a\epsi+b\epsi^2) r_J(\epsi) & = (1+a\epsi+b\epsi^2) \sum_{j = J}^\infty u_j \epsi^j \\
    & = u_J \epsi^J + u_{J+1} \epsi^{J+1} + au_J \epsi^{J+1} + \sum_{j = J+2}^\infty (u_j+au_{j-1} + b u_{j-2}) \epsi^j \\
    & = u_J \epsi^J + u_{J+1} \epsi^{J+1} + au_J \epsi^{J+1}.
\end{align*}
Consequently for small values of $\epsi$,
\begin{equation*}
F(\epsi) = \left(\sum_{j=0}^{J-1} u_j \epsi^j\right) + \dfrac{u_J  + u_{J+1} \epsi + au_J \epsi}{1+a\epsi +b\epsi^2}\epsi^J
\end{equation*}
and this identity extends meromorphically to all of $\C$. If $a$ and $b$ are polynomial of respective degree $1$ and $2$ in a parameter $\xi$ then by an immediate recursion $u_j$ is a polynomial of degree at most $j$ in $\xi$. In particular, \eqref{eq:11u} holds:
\begin{align*}
R(\xi+k/\epsi) & = \dfrac{\epsi^2}{|k|^2} \dfrac{1}{1+\epsi k\cdot \xi/|k|^2 + \epsi^2(\xi^2-\lambda^2)/|k|^2} \\
     & = \dfrac{\epsi^2}{|k|^2} \left( \left(\sum_{j=0}^{J-1} u_j \epsi^j\right) + \dfrac{u_{J}  + u_{J+1} \epsi + au_{J} \epsi}{1-2\epsi k\cdot \xi /|k|^2 + \epsi^2(\xi^2-\lambda^2)/|k|^2}\epsi^J\right) \\
     & =  \left(\sum_{j=2}^{J-1} \dfrac{u_{j-2}}{|k|^2} \epsi^j\right) + \dfrac{u_{J-1}}{|k|^2} \epsi^J+  \dfrac{u_{J-1}}{|k|^2} \epsi^{J+1} + \dfrac{u_J  + u_{J+1} \epsi + au_J \epsi}{(\xi-k/\epsi)^2 - \lambda^2}\epsi^J.
\end{align*}
The first terms in this expansion are given by
\begin{equation*} 
R(\xi+k/\epsi) = \dfrac{\epsi^2}{|k|^2} - 2\epsi^3\dfrac{k \cdot \xi}{|k|^4} - \epsi^4 \dfrac{ \xi^2 - \lambda^2 }{|k|^4} + 4 \epsi^4 \dfrac{(k \cdot \xi)^2}{|k|^6} + O(\epsi^5).
\end{equation*}

\end{document}